\documentclass[11pt]{article}
\usepackage[utf8]{inputenc}

\usepackage{epsfig,epsf,fancybox}
\usepackage{amsmath}
\usepackage{amssymb}
\usepackage{mathrsfs,bbm}
\usepackage{amssymb}
\usepackage{graphicx}
\usepackage{color}
\usepackage{multirow}
\usepackage{paralist}
\usepackage{verbatim}
\usepackage{galois}
\usepackage{algorithm}
\usepackage{algorithmic}
\usepackage{boxedminipage}
\usepackage{booktabs}
\usepackage{accents}
\usepackage{stmaryrd}
\usepackage{subfig}
\usepackage{appendix}
\usepackage{graphicx}
\usepackage{epstopdf}
\usepackage{appendix}
\usepackage{amsthm}
\usepackage{cases}  
\usepackage[top=1in,bottom=1.2in,left=1in,right=1in,xetex]{geometry}
\usepackage{extarrows}
\usepackage{graphics}

\usepackage{titlesec}
\usepackage{titletoc}

\usepackage[unicode=true,pdfusetitle,
bookmarks=true,bookmarksnumbered=true,bookmarksopen=true,bookmarksopenlevel=3,
breaklinks=false,pdfborder={0 0 1},backref=false,colorlinks=false]
{hyperref}

\numberwithin{equation}{section}

\newtheorem{theorem}{Theorem}[section]
\newtheorem{corollary}[theorem]{Corollary}
\newtheorem{lemma}[theorem]{Lemma}
\newtheorem{proposition}[theorem]{Proposition}
\newtheorem{condition}{Condition}[section]

\newtheorem{assumption}{Assumption}[section]

\newtheorem{remark}{Remark}[section]

\newcommand{\e}{\mathbb{E}}
\newcommand{\br}{\mathbb{R}}
\newcommand{\pr}{\mathcal{P}}
\newcommand{\dd}{\partial}

\newcommand{\brn}{{\mathbb{R}^n}}

\newcommand{\de}{\Delta}

\allowdisplaybreaks[2]

\title{A Variational Approach to Mean Field Type Control}

\usepackage{authblk}

\author[a]{Alain Bensoussan\footnote{E-mail: axb046100@utdallas.edu}}
\author[b]{Ziyu Huang\footnote{E-mail: zyhuang19@fudan.edu.cn}}
\author[b]{Sheung Chi Phillip Yam\footnote{E-mail: scpyam@sta.cuhk.edu.hk}}

\affil[a]{\small \it International Center for Decision and Risk Analysis, Naveen Jindal School of Management, University of Texas at Dallas}
\affil[b]{\small \it Department of Statistics and Data Science, The Chinese University of Hong Kong}
\begin{document}
	
\maketitle

\begin{abstract}

Variational methods have been used to study stochastic control for long, see Bensoussan (1982) and Bensoussan-Lions (1978) for the early works. More precisely, variational approaches apply to the study of Bellman equation as a parabolic quasi-linear equation, when the nonlinearity affects only the gradient of the solution, and the second order derivative term is linear and not degenerate. This corresponds to a stochastic control problem, where the state equation is a diffusion process. The primary objective of this article is to extend this approach to mean field control theory, as an alternative to the current approach, which considers a coupled system of Hamilton-Jacobi (HJ) and Fokker-Planck (FP) equations, since the introduction of the theory by Lasry-Lions (2007). The main novelty lies in that the equation studied here is the HJB equation, neither the HJ-FP system nor the master equation; and our results also provide another perspective for probabilistic approaches; see Chassagneux-Crisan-Delarue (2022), Bensoussan-Wong-Yam-Yuan (2024), Bensoussan-Tai-Yam (2025) and Bensoussan-Huang-Tang-Yam (2025) for instance. Within the scope of the PDE methods, the advantage of this article is to solve a larger class of mean field control problems, with moderate regularity; and this kind of variational methods fairly require few conditions on the regularity of the coefficients. \\



\end{abstract}

%\tableofcontents

%\newpage

\section{Introduction}\label{sec:intro}

A mean field type control (MFTC) problem \cite{AB_book,MR4177795} is a control problem for a dynamic system whose state is a probability measure on $\brn$. The evolution is described by a McKean-Vlasov equation \cite{McKean,McKean_book}. If one does not assume that the probability measures have densities, the natural functional space for the state of the system is the Wasserstein space of probability measures \cite{AL,VC}. Since it is a metric space and not a vector space, the classical methods of control theory are difficult to apply. In particular, Dynamic Programming Principle (DPP) \cite{AB_JMPA,DMF,PH} is different, while keeping the same concepts, like the optimality principle. 

% (1) Bellman equation (and DPP), FP-HJB system, master equation for MFTC are complex (2) important concept for MFTC: derivatives, distribution flow (and FP equation); MFG
The analytical approach for the standard stochastic control problems is simply on the Bellman equation which is defined on $\brn$. While for the MFTC problem, since the state of the dynamic is a probability measure, this standard Bellman equation must be replaced with a set of complex equations. Namely, the mean field type Bellman equation, the system of Hamilton-Jacobi-Bellman (HJB) and Fokker-Planck (FP) equations, and the mean field type master equation. The Bellman equation for the MFTC problem (see \eqref{eq:2-21} for instance) is a partial differential equation (PDE) defined on the Wasserstein space of measures rather than on $\brn$. The Wasserstein space is an infinite-dimensional space, so traditional analytical approach cannot be applied to the mean field type Bellman equation directly. The Dynamic Programming Principle (DPP) for MFTC problems is also different from the standard one; and we refer to \cite{DMF,PH} for the study of the DPP for McKean-Vlasov control problem. The couple of HJB-FP equations is a forward-backward deterministic system of coupled ordinary PDEs, with the backward equation (HJB equation) being defined on $\brn$, while the forward equation (FP equation) describing the distribution flow (or the density flow) corresponding to the state. This kind of forward-backward PDEs in the mean field theory were introduced in \cite{JM3} as a coupled system of Hamilton-Jacobi (HJ) and Fokker-Planck equations for the mean field games (MFGs); this HJ-FP system for MFGs was then studied in the literature, such as \cite{GDA}; and we also refer to \cite{AB_book,AB_JMPA} for the HJB-FP systems associated with mean field type control problems. The mean field type master equation is a backward PDE whose solution is a function of both finite dimensional and infinite dimensional arguments; see \eqref{eq:3-3} for instance and also \cite{AB_JMPA,AB_SPA,MR3332710} in the literature. All these abovementioned three kinds of equations are related, but bring different information. Indeed, the solution of the master equation can be characterized as the linear functional derivative of value function for the MFTC problem; while the value function can be viewed as the (smooth or viscosity) solution of the Bellman equation under appropriate assumptions on the coefficients. Besides, the master equation gives a decouple field for the HJB-FP system for MFTC problems; also see \cite{AB_JMPA} for a detailed discussion. We shall show these relations in detail in Sections~\ref{sec:HJB} and \ref{sec:m} of this article; and we also refer the authors to \cite{AB_JMPA,CDLL} for more discussions the the relations among these PDEs. Since the MFTC problems involve the measure variable, an important element is the concept of the functional derivatives or Wasserstein derivatives, which is given in Section~\ref{sec:Basics} of this article; and we also refer to \cite{book_mfg} for more detailed discussion on the differential calculus for functionals on the Wasserstein space. For the study on the dynamical distribution process for MFTC problems and the FP equations for McKean-Vlasov SDEs, we refer to \cite{barbu2023uniqueness,ren2022linearization}. Moreover, for the study on the PDEs (Bellman equation, HJ-FP equations and master equations) associated with MFGs, we refer to \cite{AB_book,GW,GDA,GM}. 

% (1) smooth case: PDE techniques, probabilistic techniques, and Hilbert space techniques (2) non-smooth case: viscosity theory 
The mathematical approach to MFTC problem and its related equations have been so far PDE techniques or probabilistic techniques. They are used mostly in smooth cases. For the study on the McKean-Vlasov forward-backward stochastic differential equations (FBSDEs) arsing from the maximum principle for MFTC problems, see \cite{AB10'',CR,book_mfg} for instance; and for the probabilistic approaches to the classical solution of the master equation, we refer to \cite{AB10,CJF}. For the PDE approaches to the classical solution of the Bellman equation, the HJB-FP system and the master equation for MFTC problems, see \cite{AB_book,AB_JMPA,CDLL} for the developments. There is also the technique of control on Hilbert spaces, considered by the authors, also under smoothness assumptions; see \cite{AB6,AB8,AB5} for instance. The Hilbert space approach is linked to the lifting method, introduced by Lasry and Lions \cite{JM3}. It consists of associating to a probability measure a random variable, whose probability law coincides with the probability measure. In this way, the Wasserstein metric space is replaced by the Hilbert space of square integrable random variables. For non-smooth cases, an attempt is to use the viscosity theory, which has been done by Pham-Wei \cite{PH}. They derive first an optimality principle of DPP in the Wasserstein space, and translate it into the Hilbert space of square-integrable random variables, then use a viscosity theory argument, to finally obtain a Bellman equation in the Hilbert space. Viscosity theory applies indeed to Bellman equations in Hilbert spaces. The aim of this article is to use an analytical approach to study the HJB equation (for the linear functional derivative of value function) associated with the MFTC problem by the variational methods, which has not been used in the mean field theory. And we shall see that the variational methods fairly require few smoothness conditions on the regularity of the coefficients: we require only the first-order differentiability of the Hamiltonian function and the second-order differentiability of the terminal cost function, which is the least stringent in contemporary literature.

%Variational method
Variational methods have been used to study standard stochastic control problems for long; we refer to Bensoussan \cite{1} and Bensoussan-Lions \cite{2} for the early works. More precisely, variational approaches apply to the study of Bellman equation as a parabolic quasi-linear equation, when the nonlinearity affects only the gradient of the solution, and the second order derivative term is linear and not degenerate. This corresponds to a stochastic control problem, where the state equation is a diffusion process, with control on the drift, but not on the volatility; see \cite{3} for more details. Variational methods have not been considered in the mean field context. The reason is certainly that they are linked to ordinary PDEs in finite dimensional spaces, while in the MFTC, the Bellman equation and master equation are connected with infinite dimensional spaces. The objective of this article is to extend the variational approach to MFTC problems, and we shall see that we can adapt variational techniques, by focusing on the HJB equation (see \eqref{eq:3-6}), with a fixed point mapping on probability measure flow (see \eqref{eq:4-9} and \eqref{eq:4-10}). Namely, we by-pass the FP equation, by using the probabilistic interpretation and Girsanov theorem. Indeed, we shall see that the probability measure flow can be explicitly defined as a non-local functional (see \eqref{eq:4-9} and \eqref{eq:4-10}) of the solution of the HJB equation for linear functional derivative of value function. Therefore, the HJB equation of MFTC problem appears as a HJB equation of stochastic control with a more complicated non-local Hamiltonian (see \eqref{eq:4-9} and \eqref{eq:4-10}). In this way, the extension is more apparent. The HJB equation will be the central equation of MFTC problems, from which we can recover all the other equations, namely, the master equation and Bellman equation. The main novelty lies in that we provide an alternative to the current approach for MFTC problems, which considers the HJB-FP system or the master equation, and our results also provide another perspective for our previous probabilistic approaches (see \cite{AB10,AB5,AB10''} for instance). The degree of the regularity on the coefficients required for the solvability of the HJB equation in this article is consistent with that for the solvability of the MFTC via the probabilistic approaches in our previous work. Within the scope of the PDE approach, our variational methods fairly require few conditions on the regularity of the coefficients, so it can solve a larger class of mean field control problems, with moderate regularity.

The rest of this article is organized as follows. In Section~\ref{sec:Basics}, we introduce the Wasserstein spaces and various derivatives of functionals. In Section~\ref{sec:MFTC}, we give the formal formulation and assumptions of our MFTC problem; and we also introduce the maximum principle, value function and the Bellman equation. In Section~\ref{sec:HJB}, we introduce the HJB equation for the linear functional derivative of value function and show its role. In Section~\ref{sec:main}, we give the well-posedness of the HJB equation in a rigorous way. In Section~\ref{sec:m}. we give the characterization of the probability measure flow, and use the HJB equation to recover the master equation and the Bellman equation for the MFTC problem, and also give the optimal control. 

\section{Basics of Mean Field Type Control}\label{sec:Basics}

\subsection{Functionals on the Wasserstein Space}

We work with the Wasserstein metric space of probabilities on $\brn$, denoted by $\pr(\brn)$. We denote by $\mathcal{P}_{2}(\brn)$ the space of all probability measures of finite second-order moments on $\brn$. If $d(\mu,\nu)$ denotes the metric of $\mathcal{P}_{2}(\brn),$ it can be shown that:
\begin{equation}\label{eq:1-100}
    d(m_{k},m)\rightarrow0\Longleftrightarrow 
    \left\{
    \begin{aligned}
        &m_{k}\rightharpoonup m, \quad \text{i.e. $m_{k}$ weakly convergence to $m$},\\
        &\int_{\brn}|x|^{2}dm_{k}(x)\rightarrow\int_{\brn}|x|^{2}dm(x);
    \end{aligned}
    \right.
\end{equation}
see Ambrosio-Gigli-Savar\'e \cite{AL} and Villani \cite{VC} for details.

We use the concept of functional derivative for functional $F:\pr_2(\brn)\to\br$. Namely $F$ has a functional derivative, denoted $\dfrac{d}{d\nu}F(m)(x)$ if 
\begin{equation*}
    F(m')-F(m)=\int_{0}^{1}\dfrac{d}{d\nu}F\left(m+\lambda(m'-m)\right)(x)\left(dm'(x)-dm(x)\right)d\lambda
\end{equation*}
for any $m,m'\in\pr_2(\brn)$. The function $x\mapsto\dfrac{d}{d\nu}F(m)(x)$ is defined up to a constant. We shall assume standard assumptions that $(x,m)\mapsto\dfrac{d}{d\nu}F(m)(x)$ is continuous, and 
\[
\int_{\brn}\left|\dfrac{d}{d\nu}F(m)(x)\right|^{2}dm(x)\leq c(m).
\]
Also we have 
\begin{align*}
    \dfrac{d}{d\lambda}F\left(m+\lambda(m'-m)\right)=\ & \int_{\brn}\dfrac{dF}{d\nu}\left(m+\lambda(m'-m)\right)(x)\left(dm'(x)-dm(x)\right), 
    %F(m')-F(m)=\ & \int_{0}^{1}\dfrac{d}{d\nu}F(m+\lambda(m'-m))(x)(dm'(x)-dm(x))d\lambda.
\end{align*}
for any $m,m'\in\pr_2(\brn)$; see Carmona-Delarue
\cite{book_mfg} for details, and also discussions on the relations between the linear functional derivatives and the $L$-derivatives.

%\subsection{Probability Set Up}

We consider a probability space $(\Omega,\mathcal{F},\mathbb{P})$ on which there exists an $n$-dimensional standard Wiener process $w_s=\left(w^1_s,\dots, w^n_s\right),\ 0\le s\le T$. We define $\sigma$-algebras $\mathcal{W}_{t}^{s}=\sigma\left(w_\tau-w_t,\ t\leq\tau\leq s\right)$, and denote by $\mathcal{W}_{t}$ the filtration of $\sigma$-algebras $\mathcal{W}_{t}^{s},\ s\geq t$. In the article, we shall work with random fields $X\in L_{m}^{2}(\Omega,\mathcal{F},\mathbb{P};\brn)$ which is denoted by $X:\brn\times\Omega\ni(x,\omega)\mapsto X(x,\omega)\in\brn$ such that 
\begin{equation*}
    \e\int_{\brn}\left|X(x)\right|^{2}dm(x)<+\infty.
\end{equation*}
We denote by $L^2\left(t,T;L_{m}^{2}(\Omega,\mathcal{F},\mathbb{P};\brn)\right)$ the space of processes $X_\cdot(\cdot)$ such that
\begin{align*}
    \e\int_{t}^{T}\int_{\brn}|X_s(x)|^{2}dm(x)ds<+\infty;
\end{align*}
and denote by $L^2_{\mathcal{W}_{t}}\left(t,T;L_{m}^{2}(\Omega,\mathcal{F},\mathbb{P};\brn)\right)$ the subspace of $L^2\left(t,T;L_{m}^{2}(\Omega,\mathcal{F},\mathbb{P};\brn)\right)$ of all processes adapted to the filtration $\mathcal{W}_{t}$. 

We shall use the pushed forward probability $X(\cdot,\cdot)\#(m\otimes \mathbb{P})\in\pr_2(\brn)$; to simplify the notation, we write simply $X(\cdot,\cdot)\#(m\otimes P)=X(\cdot)\otimes m$,
omitting as usual $\omega$ and $\mathbb{P}$. From the definition, we know that for any test function $\varphi$ which is continuous,
\begin{equation*}
    \int_{\brn}\varphi(\xi)d\left(X(\cdot)\otimes m\right)(\xi)=\e \int_\brn \varphi(X(x))dm(x).
\end{equation*}
We can check easily the following differentiation formula
\begin{equation}\label{relation:derivative}
    \lim_{\epsilon\to 0}\dfrac{F\left(\left(X+\epsilon\widetilde{X}\right)\otimes m\right)(\cdot)-F\left(X(\cdot)\otimes m\right)}{\epsilon}= \e\int_{\brn}D\dfrac{dF}{d\nu}(m)(X(x))\cdot \widetilde{X}(x)dm(x).
\end{equation}
Equality \eqref{relation:derivative} gives the relation between the G\^ateaux derivative of the functional on Hilbert space and the linear functional derivative of functional defined on $\pr_2(\brn)$. For the proof and more detailed discussion, we refer to \cite{AB6,AB10,AB5}.

\subsection{Differential Calculus}

We consider the It\^o process 
\begin{equation}
    X^{m,t}_s(x)=x+\int_{t}^{s}\alpha^{m,t}_\tau(x)d\tau+\sum_{j=1}^{n}\sigma^{j}\left(w^{j}_s-w_{j}(t)\right), \label{eq:2-50}
\end{equation}
where $\sigma^{j}\in \brn$ and the control process $\alpha^{m,t}_\cdot(\cdot)\in L_{\mathcal{W}_{t}}^{2}\left(t,T;L_{m}^{2}(\Omega,\mathcal{F},P;\brn)\right)$; and we denote by 
\begin{equation}
m^{m,t}_s=X^{m,t}_s(\cdot)\otimes m. \label{eq:2-52}
\end{equation}
Here and in the rest of this article, we use the superscript $m,t$ to denote their dependence on the initial condition. We introduce the second order differential operator: for $\varphi\in C^2(\brn)$, 
\begin{equation}
    \mathcal{A}\varphi(x):=-\dfrac{1}{2}\sum_{j=1}^{n}\left(\sigma^{j}\right)^\top D^{2}\varphi(x)\sigma^{j}=-\dfrac{1}{2}\sum_{i,j=1}^{n}\dfrac{\partial^{2}\varphi}{\partial x_{i}\partial x_{j}}a_{ij},
\end{equation}
where $a=\sigma\sigma^\top$. Consider an arbitrary 
 function $f:\pr_2(\brn) \ni m\mapsto f(m)\in\br$ which is functionally differentiable with 
\begin{equation}\label{condition:derivative}
    \text{the derivatives $D_{x_{i}}\dfrac{df}{d\nu}(m)(x)$ and $D_{x_{i}x_{j}}^{2}\dfrac{df}{d\nu}(m)(x)$ being in $L_{m}^{2}(\brn)$.}
\end{equation}
Then, we have the following chain rule, whose proof can be found in \cite{AB5} for instance. 
\begin{equation}
    \dfrac{df}{ds}\left(m^{m,t}_s\right)=\e\int_{\brn}\left[D\dfrac{df}{d\nu}\left(m^{m,t}_s\right)\left(X^{m,t}_s(x)\right)\cdot\alpha^{m,t}_s(x)-\mathcal{A}\dfrac{df}{d\nu}\left(m^{m,t}_s\right)\left(X^{m,t}_s(x)\right)\right]dm(x). \label{eq:2-54}
\end{equation}
Then, for a function $F:[0,T]\times \pr_2(\brn)\to \br$ which is $C^1$ in $t$ and satisfies \eqref{condition:derivative}, we can write the following formula similar to It\^o's formula: 
\begin{align}
    \dfrac{d}{ds}F\left(s,m^{m,t}_s\right)=\dfrac{\partial F}{\partial s}\left(s,m^{m,t}_s\right)+\e\int_{\brn}\bigg[&D\dfrac{dF}{d\nu}\left(s, m^{m,t}_s \right)\left(X^{m,t}_s(x)\right) \cdot \alpha^{m,t}_s(x) \notag \\
    &-\mathcal{A}\dfrac{dF}{d\nu}\left(s,m^{m,t}_s\right)\left(X^{m,t}_s(x)\right) \bigg] dm(x). \label{eq:2-560}
\end{align}
In particular, if the drift term of the process is of the form $\alpha^{m,t}_s(x)=v(s,X_{xmt}(s))$, the formula \eqref{eq:2-560} is equivalent to 
\begin{equation}
    \dfrac{d}{ds}F\left(s,m^{m,t}_s\right)=\dfrac{\partial F}{\partial s}\left(s,m^{m,t}_s\right)+\int_{\brn}\bigg[D\dfrac{dF}{d\nu}\left(s,(m^{m,t}_s\right)(x)\cdot v(s,x)-\mathcal{A}\dfrac{dF}{d\nu}\left(s,m^{m,t}_s\right)(x)\bigg]dm^{m,t}_s(x). \label{eq:2-561}
\end{equation}
When letting $s\downarrow t$, we obtain the limit 
\begin{equation*}
    \dfrac{dF}{dt}\left(t,m^{m,t}_t\right)=\dfrac{\partial F}{\partial t}(t,m)+\int_{\brn}\left[D\dfrac{dF}{d\nu}(t,m)(x)\cdot v(t,x)-\mathcal{A}\dfrac{dF}{dm}(t,m)(x)\right]dm(x). 
\end{equation*}
In the sequel, we shall refer to \eqref{eq:2-560}  as It\^o's formula. 

\section{Mean Field Control Problem}\label{sec:MFTC}

\subsection{Setting of the Problem}

For any initial time $t\in[0,T]$, an admissible control is a field stochastic process 
\begin{align*}
    v^{m,t}_\cdot(\cdot) \in L_{\mathcal{W}_{t}}^{2}\left(t,T;L_{m}^{2}(\Omega,\mathcal{F},P;\brn)\right).
\end{align*}
These processes are open-loop, but depend on parameters $x,m,t$, which are the initial condition. This will suffice to recover the feedback as we shall see, without loss of generality. When there is no ambiguity, we sometimes omit the superscript $m,t$ in a control $v^{m,t}_\cdot(\cdot)$. The associated controlled state process $X^{v,m,t}_\cdot(\cdot)$ corresponding to a control $v^{m,t}$ satisfies the following stochastic differential equation (SDE): 
\begin{equation}\label{eq:2-2}
    X^{v,m,t}_s(x)=x+\int_t^s v^{m,t}_\tau (x)d\tau + \sum_{j=1}^{n}\sigma^{j} \left(w^{j}_s-w^{j}_\tau\right).
\end{equation}
We then introduce the coefficients for the running cost:
\begin{align*}
    &l:[0,T]\times\brn\times \brn \to \br,\quad F:\pr_2(\brn)\to\br;
\end{align*}
and the coefficients for the terminal cost:
\begin{align*}
    &h: \brn\to\br,\quad F_T: \pr_2(\brn)\to\br.
\end{align*}
The cost for the control $v^{m,t}$ is then defined as
\begin{align}
    J_{mt}\left(v^{m,t}\right):=\ & \e\int_{t}^{T}\left[\int_{\brn} l\left(s,X^{v,m,t}_s(x),v^{m,t}_s(x)\right)dm(x)+F\left(X^{v,m,t}_s\otimes m\right) \right]ds \notag \\
    & +\e \int_{\brn} h \left(X^{v,m,t}_T(x)\right) dm(x)+F_{T}\left(X^{v,m,t}_T\otimes m\right). \label{eq:2-6}
\end{align}
Our objective is to minimize the functional $J_{mt}$ over controls $v^{m,t} \in L_{\mathcal{W}_{t}}^{2}(t,T;L_{m}^{2}\left(\Omega,\mathcal{F},P;\brn)\right)$. We define the Hamiltonian function $H$, which is classical in stochastic control and is defined as
\begin{equation*}
    H(s,x,p):=\inf_{v\in\brn}\left\{l(s,x,v)+p\cdot v\right\}.
\end{equation*}

We now state the following assumptions required in this article on the coefficients for the MFTC problem \eqref{eq:2-2}-\eqref{eq:2-6}. Firstly, throughout this article, we need the following non-degenerate condition on the diffusion:

\begin{assumption}[on $\sigma$]\label{eq:5-8}
    The matrix $a:=\sigma\sigma^\top$ is invertible.
\end{assumption}

For other coefficients, all functions are measurable in all arguments, and this global assumption will not be repeated. Conversely, continuity and differentiability will be stated. A generic constant will be denoted by $c$ or $c_{T}$ (for the final conditions at $T)$. Growth assumptions on the arguments are essential. We make following conditions on the coefficients $h$, $F$, $F_T$ and $H$, which are stated in order of increasing strength; indeed, although the final well-posedness result requires the full assumptions, some intermediate results only require the weaker-level assumptions.

\begin{assumption}[on $h$] \label{assumption:h}
    (i) The function $h$ is continuous and satisfies the following growth condition
    \begin{equation*}
        |h(x)|\leq c_{T}(1+|x|^{2}). 
    \end{equation*}
    (ii) Moreover, the function $h$ is differentiable in $x\in\brn$, with the derivative being continuous and satisfying 
    \begin{equation*}
        |Dh(x)|\leq c_{T}(1+|x|). 
    \end{equation*}
    (iii) Moreover, the derivative $Dh$ is differentiable, with the derivative being continuous and satisfying 
    \begin{align*}
        \left|D^{2}h(x)\right|\le c_T.
    \end{align*}
\end{assumption}

\begin{assumption}[on $F$ and $F_T$] \label{assumption:F}
    (i) The functions $F$ and $F_T$ satisfy the following growth conditions
    \begin{align*}
        &|F(m)|\le c\left(1+\int_{\brn}|x|^{2}dm(x)\right),\quad \left|F_{T}(m)\right|\le c_{T}\left(1+\int_{\brn}|x|^{2}dm(x)\right),
    \end{align*}
    and they are functional differentiable, with the functional derivatives being continuous and satisfying 
    \begin{align*}
        &\left|\frac{dF}{d\nu}(m)(x)\right|\le c\left(1+|x|^{2}\right),\quad \left|\frac{dF_T}{d\nu}(m)(x)\right|\le c_{T}\left(1+|x|^{2}\right).
    \end{align*}
    (ii) Moreover, the mappings $x\mapsto \frac{dF}{d\nu}(m)(x)$ and $x\mapsto \frac{dF_T}{d\nu}(m)(x)$ are differentiable, with the derivatives being continuous and satisfying
    \begin{align*}
        \left|D\dfrac{dF}{d\nu}(m)(x)\right|\le c(1+|x|),\quad \left|D\dfrac{dF_T}{d\nu}(m)(x)\right|\le c_{T}(1+|x|).
    \end{align*}
    (iii) Moreover, the mapping $x\mapsto D\frac{dF_T}{d\nu}(m)(x)$ is differentiable, with the derivative being continuous and satisfying
    \begin{align*}
        \left|D^{2}\dfrac{dF_T}{d\nu}(m)(x)\right|\leq c_T. 
    \end{align*}
\end{assumption}

\begin{assumption}[on $H$] \label{assumption:H}
    (i) The Hamiltonian function $H$ satisfies the following growth condition
    \begin{equation*}
        |H(s,x,p)|\le c\left(1+|x|^2\right)+\frac{\delta}{2}|p|^{2}. 
    \end{equation*}
    (ii) Moreover, the function $H$ is differentiable in $(x,q)\in\brn\times\brn$, with the derivatives being continuous and satisfying 
    \begin{align*}
        &\left|D_x H(s,x,p)\right|,\ \left|D_p H(s,x,p)\right|\ \le \ c(1+|x|+|p|).
    \end{align*}
\end{assumption}

We also refer to \cite{AB5} for assumptions on the coefficients for the study of the well-posedness for the classical solution Bellman equation via a probabilistic approach. In this article, since our aim is to use the variational approach to give the well-posedness of the HJB equation (see \eqref{eq:3-6}) in a subset $B_T$ of $L^{2}(t,T;H_{\pi_{\gamma}}^{1}(\brn))$ (defined in Subsection~\ref{subsec:main}), the regularity assumptions on the coefficients are one-order lower than in \cite{AB5}. In general, the regularity assumptions in \cite{AB5} and in this article are the least stringent in the contemporary literature --- for the classical solvability via the probabilistic method, and for the $L^{2}(t,T;H_{\pi_{\gamma}}^{1}(\brn))$ solution via the analytical approach for the PDEs, respectively. 

In the rest of this section, we shall also introduce the forward-backward system of stochastic equations for the MFTC problem \eqref{eq:2-2}-\eqref{eq:2-6} arising from the maximum principle, and also the associated Bellman equation. %In this part we shall proceed formally, to obtain the equations. We shall study them rigorously in the sequel.
Therefore, for this part, we shall also need the following assumption on $l$ and the minimizer $\widehat{v}$ for $H$. 

\begin{assumption}[on $l$] \label{assumption:l}
    The function $l$ has a quadratic growth in $(x,v)\in\brn\times\brn$, and is differentiable in $(x,v)$, with the derivatives $D_xl$ and $D_vl$ having a linear growth in $(x,v)$ and being continuous. Moreover, for any $(s,x,p)\in[t,T]\times\brn\times\brn$, the Hamiltonian function $H$ has a unique minimizer $\widehat{v}(s,x,p)$ such that 
    \begin{equation}\label{def:hat_v}
        D_v l\left(s,x,\widehat{v}(s,x,p)\right)+p=0.
    \end{equation}
\end{assumption}

Under Assumption \eqref{assumption:l}, we shall also have the following relations: for $(s,x,p)\in[t,T]\times\brn\times\brn$,
\begin{equation}\label{D_pH}
    \begin{aligned}
        D_p H(s,x,p)=\ & \widehat{v}(s,x,p), \\
        D_x H(s,x,p)=\ & D_x l\left(s,x,\widehat{v}(s,x,p)\right).
    \end{aligned}
\end{equation}
Indeed, the uniquely existence of $\widehat{v}(s,x,p)$ can be obtained under the regularity conditions in Assumption~\ref{assumption:l} and an additional convexity assumption on the function $l$ in $v$ (see Assumption~\ref{assumption:convex}), which will be assumed in Subsection~\ref{subsec:global} for the global solvability results.

\begin{remark}
    For the detailed study on the minimizer $\widehat{v}(\cdot)$, we refer to \cite[Proposition 5.4]{AB10''}. The existence of this unique minimizer in Assumption~\ref{assumption:l} is not required for the gocal-in-time solvability for the HJB equation in Section~\ref{sec:main}, but is required for the global-in-time solvability of the MFTC problem \eqref{eq:2-2}-\eqref{eq:2-6}. This kind of assumption is also required for the study for MFTC via a probabilistic approaches, see \cite{AB10,AB5} for instance; and we also refer to \cite{AB12} for a more detailed discussion on the monotonicity conditions in mean field theory. 
\end{remark}

\subsection{Pontryagin Maximum Principle and Bellman equation}

Under Assumptions~\ref{assumption:h}-(ii), \ref{assumption:F}-(ii) and \ref{assumption:l}, we define the adjoint state $P^{v,m,t}_s(x)$ for our MFTC problem \eqref{eq:2-2}-\eqref{eq:2-6} associated the control $v^{m,t}$ by the following formula:
\begin{align*}
    P^{v,m,t}_s(x)=\e \bigg[&\int_{s}^{T}\left(D_{x} l\left(\tau,X^{v,m,t}_\tau(x),v^{m,t}_\tau(x) \right)+D\frac{dF}{d\nu}\left(X^{v,m,t}_\tau\otimes m\right)\left(X^{v,m,t}_\tau(x)\right)\right)d\tau\\
    &+Dh\left(X^{v,m,t}_T(x)\right)+D\dfrac{dF_{T}}{d\nu}\left(X^{v,m,t}_T\otimes m\right)\left(X^{v,m,t}_T(x)\right)\ \bigg|\ \mathcal{W}_{t}^{s}\ \bigg],\quad s\in[t,T],
\end{align*}
where $X^{v,m,t}$ is the controlled state process corresponding to the control $v^{m,t}$ defined in \eqref{eq:2-2}. This adjoint process is also a solution of the backward SDE (BSDE):
\begin{align*}
    P^{v,m,t}_s(x)=\ & Dh\left(X^{v,m,t}_T(x)\right)+D\dfrac{dF_{T}}{d\nu}\left(X^{v,m,t}_T\otimes m\right)\left(X^{v,m,t}_T(x)\right) - \int_s^T Q^{v,m,t}_\tau (x) dw_\tau\\
    &+ \int_s^T \left[D_{x}l\left(\tau,X^{v,m,t}_\tau(x),v^{m,t}_\tau(x)\right)+D\dfrac{dF}{d\nu}\left(X^{v,m,t}_\tau\otimes m\right)\left(X^{v,m,t}_\tau(x)\right)\right]d\tau,\quad s\in[t,T].
\end{align*}
Under the above mentioned assumptions, for any $v^{m,t}$ (and the associated controlled state process $X^{v,m,t}$), the last BSDE has a unique solution 
\begin{align*}
    \left(P^{v,m,t}_\cdot(\cdot),Q^{v,m,t}_\cdot(\cdot)\right)\in L_{\mathcal{W}_{t}}^{2}\left(t,T;L_{m}^{2}(\Omega,\mathcal{F},P;\brn)\right)\times L_{\mathcal{W}_{t}}^{2}\left(t,T;L_{m}^{2}(\Omega,\mathcal{F},P;\br^{n\times n})\right);
\end{align*}
we refer to \cite{AB8} for more detials. We then have the following G\^ateaux differential property: for $v,\widetilde{v}\in L_{\mathcal{W}_{t}}^{2}\left(t,T;L_{m}^{2}(\Omega,\mathcal{F},P;\brn)\right)$,
\begin{equation}
    \frac{d}{d\theta}J_{mt}\left(v+\theta\widetilde{v}\right)\bigg|_{\theta=0}=\e\bigg[\int_{t}^{T}\int_{\brn}\left[P^{v,m,t}_s(x)+D_{v}l\left(s,X^{v,m,t}_s(x),v^{m,t}_s(x)\right)\right]\cdot \widetilde{v}_s(x)dm(x)\bigg],
\end{equation}
and we also refer to \cite{AB8} for the detailed computation. We now can state the Pontryagin maximum principle (PMP): let $u^{m,t}$ be the optimal control for MFTC problem \eqref{eq:2-2}-\eqref{eq:2-6}, and let $Y^{m,t}$ be the associated controlled state, $\left(P^{m,t},Q^{m,t}\right)$ be associated adjoint process, and we denote by 
\begin{equation}
    m^{m,t}_s:=Y^{m,t}_s\otimes m,\quad s\in[t,T], \label{eq:2-15}
\end{equation}
then, we have the following optimality condition:
\begin{equation}\label{maximum_principle}
    D_{v}l\left(s,Y^{m,t}_s,u^{m,t}_s(x)\right)+P^{m,t}_s(x)=0, s\in[t,T],\quad x\in\brn.
\end{equation}
This condition gives the following system of FBSDEs:
\begin{equation}\label{eq:2-18}
    \left\{
    \begin{aligned}
        &Y^{m,t}_s(x)=x+\int_t^s u^{m,t}_\tau (x)d\tau+\sum_{j=1}^{n}\sigma^{j}\left(w^j_s-w^j_t\right),\\
        &P^{m,t}_s(x)= Dh\left(Y^{m,t}_T(x)\right)+D\dfrac{dF_{T}}{d\nu}\left(m^{m,t}_T\right)\left(Y^{m,t}_T(x)\right) - \int_s^T Q^{m,t}_\tau (x) dw_\tau\\
        &\quad\qquad\qquad + \int_s^T \left[D_{x}l\left(\tau,Y^{m,t}_\tau(x),u^{m,t}_\tau(x)\right)+D\dfrac{dF}{d\nu}\left(m^{m,t}_\tau\right)\left(Y^{m,t}_\tau(x)\right)\right]d\tau,\\
        &D_{v}l\left(s,Y^{m,t}_s,u^{m,t}_s(x)\right)+P^{m,t}_s(x)=0,\quad s\in[t,T],\quad x\in\brn.
    \end{aligned}
    \right.
\end{equation}
From the definition of the map $\widehat{v}$ in Assumption~\ref{assumption:l}, we know that 
\begin{align}\label{eq:2-102}
    u^{m,t}_s(x)=\widehat{v}\left(s,Y^{m,t}_s(x),P^{m,t}_s(x)\right),\quad s\in[t,T],\quad x\in\brn,
\end{align}
and we can also write the system \eqref{eq:2-18} as follows 
\begin{equation}\label{eq:2-101}
    \left\{
    \begin{aligned}
        &Y^{m,t}_s(x)=x+\int_t^s u^{m,t}_\tau (x)d\tau+\sum_{j=1}^{n}\sigma^{j}\left(w^j_s-w^j_t\right),\\
        &P^{m,t}_s(x)= Dh\left(Y^{m,t}_T(x)\right)+D\dfrac{dF_{T}}{d\nu}\left(m^{m,t}_T\right)\left(Y^{m,t}_T(x)\right) - \int_s^T Q^{m,t}_\tau (x) dw_\tau\\
        &\quad\qquad\qquad + \int_s^T \left[D_x H\left(\tau,Y^{m,t}_\tau(x),P^{m,t}_\tau(x)\right)+D\dfrac{dF}{d\nu}\left(m^{m,t}_\tau\right)\left(Y^{m,t}_\tau(x)\right)\right]d\tau,\\
        &u^{m,t}_s(x)=\widehat{v}\left(s,Y^{m,t}_s(x),P^{m,t}_s(x)\right),\quad s\in[t,T],\quad x\in\brn.
    \end{aligned}
    \right.
\end{equation}
Conversely, as a sufficient condition for MFTC problem \eqref{eq:2-2}-\eqref{eq:2-6}, the solution of FBSDEs \eqref{eq:2-18} (or \eqref{eq:2-101}) can also give an optimal control for MFTC problem \eqref{eq:2-2}-\eqref{eq:2-6} under additional convexity conditions (to be stated in Subsection~\ref{subsec:global}); we refer to \cite{AB8} for the proof for the sufficiency maximum principle and that for the well-posedness of FBSDEs \eqref{eq:2-18}. 

%\subsection{Bellman Equation}

The value function for MFTC problem \eqref{eq:2-2}-\eqref{eq:2-6} is defined by 
\begin{equation}
    \Phi(t,m)=\inf_{v^{m,t}_\cdot(\cdot) \in L_{\mathcal{W}_{t}}^{2}\left(t,T;L_{m}^{2}(\Omega,\mathcal{F},P;\brn)\right)}J_{mt}\left(v^{m,t}\right),\quad (t,m)\in[0,T]\times\pr_2(\brn). 
\end{equation}
Suppose that $u^{m,t}$ is the optimal control, then, we can write
\begin{align*}
    \Phi(t,m)=J_{mt}\left(u^{m,t}\right)=\e\bigg[&\int_t^T \left[\int_{\brn}l\left(s,Y^{m,t}_s(x),u^{m,t}_s(x)\right)dm(x)+F\left(m^{m,t}_s\right)\right] ds\\
    &+\int_{\brn}h\left(Y^{m,t}_T(x)\right)dm(x)+F_{T}\left(m^{m,t}_T\right)\bigg]. 
\end{align*}
The next result shows that when the value function is smooth enough, then it satisfies the Bellman equation for MFTC. For the detailed proof, we refer to our previous work \cite{AB8}, but for the convenience of the readers, we also provide the sketch of the proof here. 

\begin{theorem}\label{thm:Bellman}
Suppose that FBSDEs \eqref{eq:2-18} has an adapted solution $\left(Y^{m,t}_s(x),P^{m,t}_s(x),Q^{m,t}_s(x)\right)$. Then, the value function $\Phi$ is functionally differentiable in $m$ and satisfies
\begin{equation}\label{eq:2-26}
    D\dfrac{d\Phi}{d\nu}(t,m)(x)=P^{m,t}_t(x).
\end{equation}
Moreover, if $\Phi$ is continuous in $t$ and the derivative $D^2\frac{d\Phi}{d\nu}(t,m)(x)$ is continuous, then $\Phi$ is a solution of the following Bellman equation 
\begin{equation}\label{eq:2-21}
    \left\{
        \begin{aligned}
            &-\dfrac{\partial \Phi}{\partial s}(s,m)+\int_{\brn}\mathcal{A}\frac{d\Phi}{d\nu}(s,m)(x)dm(x)=\int_{\brn}H\left(s,x,D\dfrac{d\Phi}{d\nu}(s,m)(x)\right)dm(x)+F(m),\\
            &\Phi(T,m)=\int_{\brn}h(x)dm(x)+F_{T}(m).
        \end{aligned}
    \right.
\end{equation} 
\end{theorem}

\begin{proof}
We first check \eqref{eq:2-26}. Let $X$ be an arbitrary element of $L_{m}^{2}(\brn)$, then, we can compute that 
\begin{equation}
    \lim_{\epsilon\to 0} \dfrac{\Phi(t,(I(\cdot)+\epsilon X(\cdot))\#m)-\Phi(t,m)}{\epsilon}=\int_{\brn}D\frac{d\Phi}{d\nu}(t,m)(x)\cdot X(x) dm(x). \label{eq:2-27}
\end{equation}
On the other hand, from the definition of the value function $\Phi$, we have for any $\epsilon>0$,
\begin{equation}
    \dfrac{\Phi(t,(I+\epsilon X))\#m)-\Phi(t,m)}{\epsilon}\ \le\ \frac{J_{(I+\epsilon X)\#m,t}\left(u^{m,t}\right)-J_{mt}\left(u^{m,t}\right)}{\epsilon}. \label{eq:2-28}
\end{equation}
We denote by $X^{u,m,t}_s\left(x+\epsilon X(x)\right)$ the state process corresponding to the initial condition $x+\epsilon X(x)$ and the control $u^{m,t}$, and we note the fact that  
\begin{align*}
    \frac{X^{u,m,t}_s\left(x+\epsilon X(x)\right)-Y^{m,t}_s(x)}{\epsilon}=X(x),
\end{align*}
then, we obtain also 
\begin{align*}
    &\lim_{\epsilon\to 0} \frac{J_{(I+\epsilon X)\#m,t}\left(u^{m,t}\right)-J_{mt}\left(u^{m,t}\right)}{\epsilon}\\
    =\ & \e \left[\int_{t}^{T} \int_{\brn} \left[D_{x}l\left(s,Y^{m,t}_s(x),u^{m,t}_s(x)\right)+D\frac{dF}{d\nu}\left(m^{m,t}_s\right)\left(Y^{m,t}_s(x)\right)\right]\cdot X(x)dm(x)ds\right]\\
    &+\e\left[\int_{\brn}\left[Dh\left(Y^{m,t}_T\right)+D\frac{dF_T}{d\nu}\left(m^{m,t}_T\right)\left(Y^{m,t}_T(x)\right)\right]\cdot X_{x}dm(x)\right].
\end{align*}
From the BSDE for the adjoint process $P^{m,t}$, we have
\begin{equation}
    \lim_{\epsilon\to 0} \frac{J_{(I+\epsilon X)\#m,t}\left(u^{m,t}\right)-J_{mt}\left(u^{m,t}\right)}{\epsilon}=\int_{\brn}X(x)\cdot P^{m,t}_t(x)dm(x). \label{eq:2-29}
\end{equation}
Combining \eqref{eq:2-27}, \eqref{eq:2-28} and \eqref{eq:2-29}, we conclude that 
\begin{align*}
    \int_{\brn}D\dfrac{d\Phi}{d\nu}(t,m)(x)\cdot X(x)dm(x)\leq\int_{\brn}X(x)\cdot P^{m,t}_t(x)dm(x).
\end{align*}
From the arbitrariness of $X$, we obtain \eqref{eq:2-26}. We now prove \eqref{eq:2-21}. Note that
\begin{align}
    \Phi(t+\epsilon,m)-\Phi(t,m)=\ & \Phi\left(t+\epsilon,m^{m,t}_{t+\epsilon}\right)-\Phi(t,m) -\left[\Phi\left(t+\epsilon,m^{m,t}_{t+\epsilon}\right)-\Phi(t+\epsilon,m)\right]. \label{eq:2-23}
\end{align}
From the dynamic programming principle, we have 
\begin{equation}\label{eq:2-22}
    \Phi(t,m)=\e\left[\int_{t}^{t+\epsilon}\int_{\brn}l\left(s,Y^{m,t}_s(x),u^{m,t}_s(x)\right)dm(x)+F\left(m^{m,t}_s\right)ds\right]+\Phi\left(t+\epsilon, m^{m,t}_{t+\epsilon}\right). 
\end{equation}
From the chain rule \eqref{eq:2-54}, we have
\begin{align}
    &\lim_{\epsilon\to 0}\dfrac{1}{\epsilon}\left[\Phi\left(t+\epsilon,m^{m,t}_{t+\epsilon}\right)-\Phi(t+\epsilon,m)\right] \notag \\
    =\ & \int_{\brn}\left[D\dfrac{d}{d\nu}\Phi(t,m)(x)\cdot u^{m,t}_t(x)-\mathcal{A}\frac{d\Phi}{d\nu}(t,m)(x)\right]dm(x).\label{eq:2-24}
\end{align}
Then, from \eqref{eq:2-23}, \eqref{eq:2-24} and \eqref{eq:2-22}, we obtain
\begin{align}\label{eq:2-25}
    \frac{\partial \Phi}{\partial t}(t,m)=\ & \int_{\brn}\left[\mathcal{A}\frac{d\Phi}{d\nu}(t,m)(x)-D\frac{d\Phi}{d\nu}(t,m)(x)\cdot u^{m,t}_t(x)-l\left(t,x,u^{m,t}_t(x)\right)\right]dm(x)-F(m).
\end{align}
By substituting \eqref{eq:2-26} into \eqref{eq:2-25}, we have
\begin{align*}
    &\int_{\brn}\left[l\left(t,x,u^{m,t}_t(x)\right)+D\frac{d\Phi}{d\nu}(t,m)(x)\cdot u^{m,t}_t(x)\right]dm(x)\\
    =\ & \int_{\brn}\left[l\left(t,x,u^{m,t}_t(x)\right)+P^{m,t}_t(x)\cdot u^{m,t}_t(x)\right]dm(x)\\
    =\ & \int_{\brn}H\left(t,x,D\frac{d\Phi}{d\nu}(m,t)(x)\right)dm(x).
\end{align*}
Therefore, \eqref{eq:2-25} is equivalent to \eqref{eq:2-21}, with $s=t$. Since the final condition at $t=T$ is clear, the result \eqref{eq:2-21} is proven. 
\end{proof}

We also refer to \cite{Buckdahn-Peng} for a similar representation as \eqref{eq:2-26} by using the Malliavin calculus. The regularity conditions on $\Phi$ in Theorem~\ref{thm:Bellman} will be given by Proposition~\ref{prop:6-2}, Theorem~\ref{thm:D_zV} and Corollary~\ref{corollary}. And it can also be proven by studying the Jacobian flow for FBSDEs \eqref{eq:2-18} by a probabilistic approach, and we refer to \cite{AB10} for the details. From \eqref{eq:2-26}, we can also state that 
\begin{equation}\label{eq:2-30}
    u^{m,t}_t(x)=\widehat{v}\left(t,x,P^{m,t}_t(x)\right)= \widehat{v}\left(t,x,D\frac{d\Phi}{d\nu}(t,m)(x)\right).
\end{equation}
This formula defines a feedback function, where the the state argument is the pair $(x,m)$. 

\section{HJB Equation for Linear Functional Derivative of Value Function}\label{sec:HJB}

From \eqref{eq:2-21} we can see that the Bellman equation for the MFTC problem \eqref{eq:2-2}-\eqref{eq:2-6} is PDE defined on the distribution space $\pr_2(\brn)$. The space $\pr_2(\brn)$ is a infinite-dimensional space, which makes the Bellman equation \eqref{eq:2-21} not easy to solve by a traditional analytical approach (while our previous works \cite{AB8,AB10,AB5} give a probabilistic approach). 

This inspires us to introduce and work on the HJB equation for MFTC instead of the Bellman equation. In this section, we shall see that the HJB equation is a backward PDE for $(s,x)$ rather than $(s,m)$, which makes it more convenient to solve and to study the regularity. This inspires us to introduce and work on another backward equation associated with MFTC instead of the Bellman equation, which is a PDE defined on $(s,x)\in[t,T]\times\brn$ rather than $(s,m)\in[t,T]\times\pr_2(\brn)$. We call it HJB equation to make distinguish with Bellman equation, and it is more convenient to solve and to study the regularity, and also requires less regularity assumptions. Indeed, our article provides an alternative method and perspective in studying the mean field type control problems, in addition to the popular HJ-FP system, Bellman equation and the master equation approaches. 

In this section, We shall first deduce the HJB equation from the Bellman equation, and then recover all equations (including Bellman equation, master equation and adjoint process) from this HJB equation, to show the role of the HJB equation and also the motavition of this article. The equations in this section are deduced formally, and the rigorous proofs for the well-posedness are given in the next two sections. 

\subsection{Master Equation}

We define the function $U:[0,T]\times\brn\times\pr_2(\brn)\to\br$ as 
\begin{equation*}
    U(s,x,m)=\dfrac{d\Phi}{d\nu}(s,m)(x), \label{eq:3-1}
\end{equation*}
where $\Phi$ is the value function. Then, the Bellman equation \eqref{eq:2-21} also writes
\begin{equation}\label{eq:3-2}
    \left\{
        \begin{aligned}
            &-\dfrac{\partial \Phi}{\partial s}(s,m)+\int_{\brn}\mathcal{A}U(s,x,m)dm(x)=\int_{\brn}H\left(s,x,DU(s,x,m)\right)dm(x)+F(m),\\
            &\Phi(T,m)=\int_{\brn}h(x)dm(x)+F_{T}(m),
        \end{aligned}
    \right.
\end{equation} 
and we can write \eqref{eq:2-30} as 
\begin{equation*}
    u^{m,s}_s(x)=\widehat{v}\left(s,x,DU(s,x,m)\right). \label{eq:3-20}
\end{equation*}
We can obtain a differential equation for $U$ by differentiating \eqref{eq:3-2} with respect to $m$:
\begin{equation}\label{eq:3-3}
    \left\{
        \begin{aligned}
            &-\dfrac{\partial U}{\partial s}(s,x,m)+\mathcal{A}U(s,x,m)+\int_{\brn}\mathcal{A}_\xi \frac{dU}{d\nu}(s,\xi,m)(x)dm(\xi)\\
            &=H\left(s,x,DU(s,x,m)\right)+\int_{\brn}D_p H\left(s,\xi,DU(s,\xi,m)\right)\cdot D_\xi\frac{dU}{d\nu}(s,\xi,m)(x) dm(\xi)+
            \frac{dF}{d\nu}(m)(x),\\
            &U(T,x,m)=h(x)+\frac{dF_T}{d\nu}(m)(x).
        \end{aligned}
    \right.
\end{equation} 
The functional derivative $U(s,x,m)$ is defined up to a function independent of $x$, %The equation shows immediately that it cannot depend of $s$ either. 
and it is then fixed by the final condition at $T$. Equation \eqref{eq:3-3} is called the master equation for the MFTC, which was proposed in \cite{AB_JMPA}; for the results on the existence of the classical solution for \eqref{eq:3-3}, we refer to \cite{AB10,AB5} for a probabilistic approach. In this article, we shall also give the well-posedness results for \eqref{eq:3-3} in Proposition~\ref{prop:6-1}, Theorem~\ref{thm:D_zV} and Corollary~\ref{corollary}.

\subsection{The Hamilton Jacobi Bellman Equation}

We now define the function $V^{m,t}:[t,T]\times\brn\to\br$ as 
\begin{equation}
    V^{m,t}(s,x)=U\left(s,x,m^{m,t}_s\right), \label{eq:3-4}
\end{equation}
and use superscript $m,t$ for $V^{m,t}$ to denote its dependence on the initial condition for $m^{m,t}_s$. The important result is a consequence of It\^o's formula \eqref{eq:2-561}:
\begin{align}
    \dfrac{\partial V^{m,t}}{\partial s}(s,x)= \dfrac{\partial U}{\partial s}\left(s,x,m^{m,t}_s\right)+\int_{\brn}\bigg[& D_{p}H\left(s,\xi,DU\left(s,\xi,m^{m,t}_s\right)\right)\cdot D_{\xi}\frac{dU}{d\nu}\left(s,\xi,m^{m,t}_s\right)(x) \notag \\
    &-\mathcal{A}_{\xi}\dfrac{dU}{d\nu}\left(s,\xi,m^{m,t}_s \right)(x)\bigg]dm^{m,t}_s(\xi). \label{eq:3-5}
\end{align}
From the master equation \eqref{eq:3-3}, we can deduce that 
\begin{align*}
    &-\frac{\partial U}{\partial s}\left(s,x,m^{m,t}_s\right)+\mathcal{A}U\left(s,x,m^{m,t}_s \right)+\int_{\brn} \mathcal{A}_{\xi}\frac{dU}{d\nu}\left(s,\xi,m^{m,t}_s\right)(x)dm^{m,t}_s(\xi)\\
    =\ & H\left(s,x,DU\left(s,x,m^{m,t}_s\right)\right)+\frac{dF}{d\nu}\left(m^{m,t}_s\right)(x)\\
    & +\int_{\brn} D_{p} H\left(s,\xi,DU\left(s,\xi,m^{m,t}_s \right)\right) \cdot D_{\xi}\dfrac{dU}{d\nu}\left(s,\xi,m^{m,t}_s\right)(x)dm^{m,t}_s(\xi).
\end{align*}
Combining the last equation with \eqref{eq:3-5}, we obtain the following HJB equation for $V^{m,t}$:
\begin{equation}\label{eq:3-6}
    \left\{
    \begin{aligned}
        &-\frac{\partial V^{m,t}}{\partial s}(s,x)+\mathcal{A}V^{m,t}(s,x)=H\left(s,x,DV^{m,t}(s,x)\right)+\frac{dF}{d\nu}\left(m^{m,t}_s\right)(x),\quad s\in[t,T),\\
        &V^{m,t}(T,x)=h(x)+\dfrac{dF_T}{d\nu}\left(m^{m,t}_T\right)(x),\quad x\in\brn.
    \end{aligned}
    \right.
\end{equation}
The main results of this article is to give the well-posedness for the HJB equation \eqref{eq:3-6}, which will be given in the Section~\ref{sec:main}; see Theorem~\ref{theo5-2} for the local-in-time well-posedness and Theorem~\ref{theo5-3} for the global-in-time result.

\subsection{The Measure $m^{m,t}_s$}

The probability measure $m^{m,t}_s$ was introduced in \eqref{eq:2-15}. From the flow property, \eqref{eq:2-26} and \eqref{eq:3-1}, we know that
\begin{align}
    P^{m,t}_s(x)=\ & P^{m^{m,t}_s,s}_s\left(Y^{m,t}_s(x)\right)=D\dfrac{d\Phi}{d\nu}\left(s,m^{m,t}_s\right)\left(Y^{m,t}_s(x)\right) \notag \\
    =\ & DU\left(s,Y^{m,t}_s(x),m^{m,t}_s\right)=DV^{m,t}\left(s,Y^{m,t}_s(x)\right), \label{P_DV}
\end{align}
then, from \eqref{eq:2-102}, we obtain 
\begin{equation*}
    u^{m,t}_s(x)=\widehat{v}\left(s,Y^{m,t}_s(x),DV^{m,t}\left(s,Y^{m,t}_s(x)\right) \right).
\end{equation*}
The process $Y^{m,t}_s(x)$ is then solution of the following SDE:
\begin{equation}\label{eq:4-7}
    Y^{m,t}_s(x)=x+\int_t^s \widehat{v}\left(\tau,Y^{m,t}_s(\tau),DV^{m,t}\left(\tau, Y^{m,t}_\tau(x)\right)\right)d\tau + \sum_{j=1}^{n}\sigma^{j}\left(w^{j}_s-w^{j}_t \right),\quad s\in[t,T].
\end{equation}
From \eqref{eq:4-7} and \eqref{eq:2-15}, we see that the measure $m^{m,t}_s$ is uniquely characterized by
the function $DV^{m,t}$. We have the following result. 

\begin{lemma}\label{lem4-1}
Suppose that $\varphi$ is a continuous function on $\brn$ satisfying $|\varphi(z)|\le C\left(1+|z|^{2}\right)$, then, we have the formula:
\begin{equation}\label{eq:4-9}
    \int_{\brn}\varphi(z)dm^{m,t}_s(z)=\int_{\brn}\Psi^{m,t}(t,z)dm(z),
\end{equation}
where $\Psi^{m,t}(\tau,z)$ is the solution of the following backward equation defined on $(\tau,z)\in[t,s]\times\brn$:
\begin{equation}\label{eq:4-10}
\left\{
    \begin{aligned}
        &-\dfrac{\partial\Psi^{m,t}}{\partial\tau}(\tau,z)+\mathcal{A}\Psi^{m,t}(\tau,z)=D\Psi^{m,t}(\tau,z)\cdot D_p H\left(\tau,z,DV^{m,t}(\tau,z)\right),\quad \tau\in[t,s),\\
        &\Psi^{m,t}(s,z)=\varphi(z),\quad z\in\brn.
    \end{aligned}
\right.
\end{equation}
\end{lemma}

\begin{proof}
Since $m^{m,t}_s=Y^{m,t}_s\otimes m$, we know that
\begin{align*}
    \int_{\brn}\varphi(z)dm^{m,t}_s(z)=\e\int_{\brn} \varphi\left(Y^{m,t}_s(z)\right)dm(z)
\end{align*}
Considering the function $\Psi^{m,t}(\tau,z)$ the solution of \eqref{eq:4-10} and the process $Y^{m,t}_\tau(z)$ the solution of SDE \eqref{eq:4-7} with initial $z$, then, from relation \eqref{D_pH} and the It\^o's formula, we have 
\begin{align*}
    \e\left[\varphi\left(Y^{m,t}_s(z)\right)\right]=\Psi^{m,t}(t,z),
\end{align*}
from which we obtain \eqref{eq:4-9}.
\end{proof}

The rigorous proof for the well-posedness of the measure $m^{m,t}_s$ will be given after the well-posedness of $V^{m,t}$ in Section~\ref{sec:main}; see Subsection~\ref{subsec:m} for details. 

\begin{remark}[Important Observation]
    The HJB equation \eqref{eq:3-6} can be considered as a single equation for the function $V^{m,t}(s,x)$ indexed by parameters $m$ and $t$. This is because the probability measure $m^{m,t}_s$ is explicitly defined as a functional of $DV^{m,t}$. Indeed, from \eqref{eq:4-9} and \eqref{eq:4-10}, the measure $m^{m,t}_s$ is a non-local functional of $DV^{m,t}$, so \eqref{eq:3-6} is an HJB equation with a non-local Hamiltonian. In this way, we have a very natural extension of the HJB equation of stochastic control. We can call it the HJB equation of mean field type control. It will be the central equation of mean field type control problems, from which we can recover all the others, as we shall see in the following subsection. 
\end{remark}

\subsection{Recovering All Equations from the HJB Equation} \label{subsec:RECOVERING-ALL-EQUATIONS}

In the preceding developments, the logic of treatment of the MFTC problem \eqref{eq:2-2}-\eqref{eq:2-6} was first to write a Pontryagin Maximum Principle, then Bellman equation, then the master equation and finally the HJB equation. In the literature, this last equation is combined with a Fokker-Planck equation. We differ in considering only the HJB equation in which appears the measure $m^{m,t}_s$ characterized by the equations \eqref{eq:4-9} and \eqref{eq:4-10}. Let us check now that we can go backward from the HJB equation and recover all other equations and the PMP. We shall then give a rigorous approach of the HJB equation in Section~\ref{sec:main}, extending the methodology of standard stochastic control, that we call ``Variational approach", which permits weaker assumptions than the standard full regularity theory developed in the literature. %(based on Schauder regulatory theory). 

From the HJB equation, we define both the function $V^{m,t}(s,x)$ and the measure $m^{m,t}_s$. Then, we can define the process $Y^{m,t}_s(x)$ through SDE \eqref{eq:4-7} and obtain the relation \eqref{eq:2-15}. %We shall interpret later the SDE (\ref{eq:4-7}) in a weak sense. 
We next obtain the master equation for $U(t,x,m)$. Recalling that $V^{m,t}(s,x)=U\left(s,x,m^{m,t}_s\right)$, we compute the derivative $\dfrac{\partial V^{m,t}}{\partial s}(s,x)$ by using the differentiation rule \eqref{eq:2-54} and introducing the functional derivative $\dfrac{dU}{d\nu}(s,x,m)(\xi)$ to obtain 
\begin{align*}
    \dfrac{\partial V^{m,t}}{\partial s}(s,x)=\dfrac{\partial U}{\partial s}\left(s,x,m^{m,t}_s\right)+\int_{\brn}\bigg[&D_{\xi}\frac{dU}{d\nu}\left(s,x,m^{m,t}_s\right)(\xi)\cdot D_p H\left(s,\xi,DV^{m,t}(s,\xi)\right)\\
    &-\mathcal{A}_{\xi}\dfrac{dU}{d\nu}\left(s,x,m^{m,t}_s\right)(\xi)\bigg]dm^{m,t}_s(\xi);
\end{align*}
and from the HJB equation \eqref{eq:3-6}, we then know that
\begin{align*}
    \dfrac{\partial V^{m,t}}{\partial s}(s,x)=\mathcal{A}U\left(s,x,m^{m,t}_s\right)-H\left(s,x,DU\left(s,x,m^{m,t}_s\right)\right)-\dfrac{dF}{d\nu}\left(m^{m,t}_s\right)(x).
\end{align*}
Writing this relation for $t=s$, we obtain
\begin{equation}\label{eq:4-12}
    \left\{
        \begin{aligned}
            &-\dfrac{\partial U}{\partial s}(s,x,m)+\mathcal{A}U(s,x,m)+\int_{\brn}\mathcal{A}_\xi \frac{dU}{d\nu}(s,x,m)(\xi)dm(\xi)\\
            &=H\left(s,x,DU(s,x,m)\right)+\int_{\brn}D_p H\left(s,\xi,DU(s,\xi,m)\right)\cdot D_\xi\frac{dU}{d\nu}(s,x,m)(\xi) dm(\xi)+
            \frac{dF}{d\nu}(m)(x),\\
            &U(T,x,m)=h(x)+\frac{dF_T}{d\nu}(m)(x).
        \end{aligned}
    \right.
\end{equation} 
By letting $U(s,x,m)=\dfrac{d\Phi}{d\nu}(s,m)(x)$ and plugging
into \eqref{eq:4-12}, we obtain the functional derivative of the Bellman equation \eqref{eq:2-21}, and we can note that 
\begin{align*}
    \dfrac{dU}{d\nu}(s,x,m)(\xi)=\dfrac{d^{2}\Phi}{d\nu^{2}}(s,m)(x,\xi)=\dfrac{d^{2}\Phi}{d\nu^{2}}(s,m)(\xi,x)=\dfrac{dU}{d\nu}(s,\xi,m)(x).
\end{align*}
Therefore, we see that \eqref{eq:4-12} is equivalent to \eqref{eq:3-3} and coincides with the master equation \eqref{eq:3-3}. See Proposition~\ref{prop:6-1} for a rigorous proof on recovering the master equation, and Proposition~\ref{prop:6-2} for that on recovering the Bellman equation; and also see Theorem~\ref{thm:D_zV} for the regularity of $U$.

To recover the equation for $P^{m,t}_s(x)$ in \eqref{eq:2-101}, we note from \eqref{P_DV} that
\begin{align}
    P^{m,t}_s(x)=DV^{m,t}\left(s,Y^{m,t}_s(x)\right). \label{P_DV'}
\end{align}
We first differentiate the HJB equation \eqref{eq:3-6} with respect to $x_i$ (where $x=(x_1,\dots,x_n)$) to obtain
\begin{align*}
    &-\dfrac{\partial}{\partial s}\dfrac{\partial V^{m,t}}{\partial x_{i}}(s,x)+\mathcal{A}\dfrac{\partial V^{m,t}}{\partial x_{i}}(s,x)\\
    =\ & \dfrac{\partial H}{\partial x_{i}}\left(s,x,DV^{m,t}(s,x)\right)+D_p H\left(s,x,DV^{m,t}(s,x)\right) \cdot D\dfrac{\partial V^{m,t}}{\partial x_{i}}(s,x)+\frac{\partial}{\partial x_{i}}\dfrac{d F}{d\nu}\left(m^{m,t}_s\right)(x);
\end{align*}
any by taking $x=Y^{m,t}_s(x)$, we get 
\begin{align}
    &-\dfrac{\partial}{\partial s}\dfrac{\partial V^{m,t}}{\partial x_{i}}\left(s,Y^{m,t}_s(x)\right)+\mathcal{A}\dfrac{\partial V^{m,t}}{\partial x_{i}}\left(s,Y^{m,t}_s(x)\right) \notag \\
    =\ & \dfrac{\partial H}{\partial x_{i}}\left(s,Y^{m,t}_s(x),P^{m,t}_s(x)\right)+u^{m,t}_s(x)\cdot D\dfrac{\partial V^{m,t}}{\partial x_{i}}\left(s,Y^{m,t}_s(x)\right) +\dfrac{\partial}{\partial x_{i}}\dfrac{dF}{d\nu}\left(m^{m,t}_s\right)\left(Y^{m,t}_s(x)\right). \label{eq:4-19}
\end{align}
From \eqref{P_DV'}, by writing $P^{m,t,i}_s(x)=\frac{\partial V^{m,t}}{\partial x_{i}}\left(s,Y^{m,t}_s(x)\right)$ and using It\^o's formula, we get 
\begin{align*}
    -dP^{m,t,i}_s(x)=\ & -\left[\dfrac{\partial}{\partial s}\dfrac{\partial V^{m,t}}{\partial x_{i}}\left(s,Y^{m,t}_s(x)\right)-\mathcal{A}\dfrac{\partial V^{m,t}}{\partial x_{i}}\left(s,Y^{m,t}_s(x)\right)+D\dfrac{\partial V^{m,t}}{\partial x_{i}}\left(s,Y^{m,t}_s(x)\right)\cdot u^{m,t}_s(x) \right]ds\\
    &-\sum_{j=1}^n D\frac{\partial V^{m,t}}{\partial x_{i}}\left(s,Y^{m,t}_s(x)\right)\cdot \sigma^{j}dw^{j}_s,
\end{align*}
and by using \eqref{eq:4-19}, it follows that
\begin{align}
    -dP^{m,t,i}_s(x)=\ & \left[\dfrac{\partial H}{\partial x_{i}}\left(s,Y^{m,t}_s(x),P^{m,t}_s(x)\right) +\dfrac{\partial}{\partial x_{i}}\dfrac{dF}{d\nu}\left(m^{m,t}_s\right)\left(Y^{m,t}_s(x)\right)\right]ds \notag \\
    &-\sum_{j=1}^n D\frac{\partial V^{m,t}}{\partial x_{i}}\left(s,Y^{m,t}_s(x)\right)\cdot \sigma^{j}dw^{j}_s. \label{BSDE:P'}
\end{align}
Therefore, we can write 
\begin{align*}
    P^{m,t}_s(x)=\ & Dh\left(Y^{m,t}_T(x)\right)+D\dfrac{dF_{T}}{d\nu}\left(m^{m,t}_T\right)\left(Y^{m,t}_T(x)\right) - \sum_{j=1}^n \int_s^T D^2 V^{m,t}\left(\tau,Y^{m,t}_\tau(x)\right)\sigma^j dw^j_\tau\\
    &+ \int_s^T \left[D_x H\left(\tau,Y^{m,t}_\tau(x),P^{m,t}_\tau(x)\right)+D\dfrac{dF}{d\nu}\left(m^{m,t}_\tau\right)\left(Y^{m,t}_\tau(x)\right)\right]d\tau,
\end{align*}
and we recover the backward equation in \eqref{eq:2-101}), and also with the additional information 
\begin{equation*}
    Q^{m,t,j}_s(x)=D^2 V^{m,t}\left(s,Y^{m,t}_s(x)\right)\sigma^{j},\quad j=1,2,\dots,n.
\end{equation*}

\section{Study of the HJB Equation}\label{sec:main}

%\subsection{General Comments}

Although the developments above are very formal, they show the key role of the HJB equation, in a way which is very similar to that of standard stochastic control. The objective now is to give the well-posedness in a rigorous way the following set of equations: given parameters $t$ and $m$, find a solution $V^{m,t}(s,x)$ of the following HJB equation:
\begin{equation}\label{eq:5-150}
    \left\{
    \begin{aligned}
        &-\frac{\partial V^{m,t}}{\partial s}(s,x)+\mathcal{A}V^{m,t}(s,x)=H\left(s,x,DV^{m,t}(s,x)\right)+\frac{dF}{d\nu}\left(m^{m,t}_s\right)(x),\quad s\in[t,T),\\
        &V^{m,t}(T,x)=h(x)+\dfrac{dF_T}{d\nu}\left(m^{m,t}_T\right)(x),\quad x\in\brn,
    \end{aligned}
    \right.
\end{equation}
in which the measure $m^{m,t}_s$ is obtained from $DV^{m,t}(s,x)$
by the relation 
\begin{equation}
    \int_{\brn}\varphi(z)dm^{m,t}_s(z)=\int_{\brn}\Psi^{m,t}(t,z)dm(z), \label{eq:5-151}
\end{equation}
where $\Psi^{m,t}(\tau,z)$ is the solution of the equation 
\begin{equation}\label{eq:5-152}
\left\{
    \begin{aligned}
        &-\dfrac{\partial\Psi^{m,t}}{\partial\tau}(\tau,z)+\mathcal{A}\Psi^{m,t}(\tau,z)=D\Psi^{m,t}(\tau,z)\cdot D_p H\left(\tau,z,DV^{m,t}(\tau,z)\right),\quad \tau\in[t,s),\\
        &\Psi^{m,t}(s,z)=\varphi(z),\quad z\in\brn,
    \end{aligned}
\right.
\end{equation}
for any test function $\varphi$ which is continuous and bounded. We shall need of course to make precise the functional space for $V^{m,t}(s,x)$, %and $m^{m,t}_s$, 
which will be given in Subsection~\ref{subsec:main}. 

In this section, to simplify notation in the sequel, we shall write $V^{m,t}(s,x)=V(s,x)$ and $m^{m,t}_s=m_s$; and we shall use Assumptions~\ref{eq:5-8}, \ref{assumption:h}, \ref{assumption:F} and \ref{assumption:H}. We shall first give \textit{a priori} estimate on $V$ and its deratives in Subsection~\ref{subsec:priori:V}-\ref{subsec:V_x_i}, and give the solvability of HJB equation \eqref{eq:5-150} for $T-t$ sufficiently small in Subection~\ref{subsec:main}, then, we shall also give the global-in-time solvability results under additional convexity assumptions on the cost functions in Subsection~\ref{subsec:global}.

\subsection{\textit{A Priori} Estimate on the boundedness of $|V(s,x)|$}\label{subsec:priori:V}

In this part, we give \textit{a priori} estimate on the growth condition for the $V(s,x)$.  From Equation \eqref{eq:5-150} and Assumptions~\ref{assumption:H}-(i), \ref{assumption:h}-(i) and \ref{assumption:F}-(i), we have 
\begin{equation}\label{eq:5-10}
    \begin{aligned}
        \left|-\dfrac{\partial V}{\partial s}(s,x)+\mathcal{A}V(s,x) \right| \le\ & \dfrac{\delta}{2}|DV(s,x)|^{2}+2c\left(1+|x|^{2}\right),\\
        |V(T,x)|\le\ &  2c_{T} \left(1+|x|^{2}\right).
    \end{aligned}
\end{equation}
According to the method of majorants, we consider the following equation 
\begin{equation}\label{eq:5-11}
    \left\{
    \begin{aligned}
        &-\dfrac{\partial z}{\partial s}(s,x)+\mathcal{A}z(s,x)=\dfrac{\delta}{2}|Dz(s,x)|^{2}+2c\left(1+|x|^{2}\right),\quad s\in[t,T),\\
        &z(T,x)=2 c_{T} \left(1+|x|^{2}\right),\quad x\in\brn,
    \end{aligned}
    \right.
\end{equation}
and to check whether it has a solution of the form 
\begin{equation}\label{eq:5-12}
    z(s,x)=\beta(s)\dfrac{|x|^{2}}{2}+\mu(s),
\end{equation}
with $\beta(s)$ and $\mu(s)$ being bounded on $[t,T]$. We see easily that $\beta(s)$ satisfies a Riccati equation 
\begin{equation*}
    \left\{
    \begin{aligned}
        &-\dfrac{\beta'(s)}{2}=2c+\dfrac{\delta}{2}\beta^{2}(s),\quad s\in[t,T),\\
        &\beta(T)=4 c_{T}. 
    \end{aligned}
    \right.
\end{equation*}
It has a bounded solution if 
\begin{equation}\label{eq:5-14}
    \arctan\left(2c_{T}\sqrt{\dfrac{\delta}{c}}\right) + (T-t)2\sqrt{\delta c}<\dfrac{\pi}{2},
\end{equation}
which is 
\begin{equation}\label{eq:5-15}
    \beta(s)=2\sqrt{\dfrac{c}{\delta}}\tan\left(\arctan\left(2c_{T}\sqrt{\dfrac{\delta}{c}} \right)+(T-s)2\sqrt{\delta c}\right);
\end{equation}
then, we see easily that 
\begin{equation}\label{eq:5-16}
    \mu(s)=2c_{T}+\int_{s}^{T}\left[2c+\dfrac{1}{2}\beta(\tau)\text{tr}(a)\right]d\tau. 
\end{equation}
We have the following \textit{a priori} estimate for $V(s,x)$, which is based on the standard maximum principle for parabolic equations; see \cite{9} for instance. 

\begin{proposition}\label{prop5-1}
Under Assumptions~\ref{eq:5-8}, \ref{assumption:h}-(i), \ref{assumption:F}-(i) and \ref{assumption:H}-(i), we make the smallness condition on $T-t$ such that \eqref{eq:5-14} holds. Define $z(s,x)$ as in \eqref{eq:5-11}, \eqref{eq:5-15} and \eqref{eq:5-16}. Suppose that a solution $V$ satisfies the HJB equation \eqref{eq:5-150}, and it also satisfies the following  conditions at infinity:
\begin{equation}\label{eq:5-30}
    \limsup_{|x|\rightarrow+\infty}\left(V(s,x)-z(s,x)\right)\le 0,\quad \liminf_{|x|\rightarrow+\infty}\left(V(s,x)+z(s,x)\right) \geq 0.
\end{equation}
Then, we have the following estimate 
\begin{equation}\label{eq:5-31}
    -z(s,x)\le V(s,x)\leq z(s,x),\quad (s,x)\in[t,T]\times\brn.
\end{equation}
\end{proposition}

\begin{proof}
Let us denote by $\zeta^{-}(s,x):=V(s,x)-z(s,x).$ From inequality \eqref{eq:5-10}, equation \eqref{eq:5-11} and condition \eqref{eq:5-30}, we know that $\zeta^{-}$ satisfies
\begin{equation}\label{eq:5-32}
\begin{aligned}
    &-\dfrac{\partial \zeta^{-}}{\partial s}(s,x)+\mathcal{A}\zeta^{-}(s,x)\le \dfrac{\delta}{2}D\zeta^{-}(s,x)\cdot (DV(s,x)+Dz(s,x)), \quad s\in[t,T),\\
    &\qquad\qquad\qquad\qquad \zeta^{-}(T,x)\le 0,\qquad \limsup_{|x|\rightarrow+\infty}\:\zeta^{-}(s,x)\le 0.
\end{aligned}
\end{equation}
If $\zeta^{-}(s,x)$ is strictly positive at a point $\left(s^{*},x^{*}\right)$, then it has a positive local maximum (which we also call $\left(s^{*},x^{*}\right)$ for the sake of convenience). This is because of the continuity of
the function $\zeta^{-}$ and the condition at infinity; and from \eqref{eq:5-32} we know that $s^{*}<T$. By the maximum principle and the non-degeneracy assumption \ref{eq:5-8}, we also know that 
\begin{align*}
    -\dfrac{\partial \zeta^{-}}{\partial s}\left(s^{*},x^{*}\right)\geq 0, \quad \mathcal{A}\zeta^{-}\left(s^{*},x^{*}\right)>0,\quad D\zeta^{-}\left(s^{*},x^{*}\right)=0.
\end{align*}
Substituting these estimate back into the first inequality of \eqref{eq:5-32}, we get a contradiction. This proves the second inequality in \eqref{eq:5-31}. A similar argument holds for the first inequality of \eqref{eq:5-31}, which completes the proof.
\end{proof}

\subsection{\textit{A Priori} Estimate on the Boundedness of $|DV(s,x)|$}\label{subsec:DV}

In this part, we derive \textit{a priori} estimate on the gradient $DV(s,x)$ under Assumptions~\ref{eq:5-8}, \ref{assumption:h}-(ii), \ref{assumption:F}-(ii) and \ref{assumption:H}. We define 
\begin{equation*}\label{eq:5-33}
    \varpi(s,x)=\dfrac{1}{2}|DV(s,x)|^{2},\quad (s,x)\in[t,T]\times\brn, 
\end{equation*}
and we know from the definition that
\begin{align*}
    \dfrac{\partial \varpi}{\partial s}(s,x)=\sum_{k=1}^{n}\dfrac{\partial V}{\partial x_{k}}(s,x)\dfrac{\partial^{2}V}{\partial x_{k}\partial s}(s,x),\quad x=\left(x_1,\dots,x_n\right).
\end{align*}
To obtain the equation for $\varpi$, we differentiate the HJB equation \eqref{eq:5-150} with respect to $x_{k}$ to get 
\begin{align*}
    & -\dfrac{\partial^{2} V}{\partial s\partial x_{k}}(s,x)-\dfrac{1}{2}\sum_{i,j=1}^{n}a_{ij}\dfrac{\partial^{3}V}{\partial x_{i}\partial x_{j}\partial x_{k}}(s,x)\\
    =\ & \dfrac{\partial H}{\partial x_{k}}(s,x,DV(s,x))+\sum_{i=1}^{n}\dfrac{\partial H}{\partial p_{i}}(s,x,DV(s,x))\dfrac{\partial^{2}V}{\partial x_{i}\partial x_{k}}(s,x)+\dfrac{\partial}{\partial x_{k}}\dfrac{dF}{d\nu}\left(m_s\right)(x),
\end{align*}
and hence,
\begin{align*}
    &-\dfrac{\partial \varpi}{\partial s}(s,x)-\dfrac{1}{2}\sum_{i,j,k=1}^{n}a_{ij}\dfrac{\partial^{3}V}{\partial x_{i}\partial x_{j}\partial x_{k}}(s,x)\dfrac{\partial V}{\partial x_{k}}(s,x)\\
    =\ & \sum_{k=1}^{n}\dfrac{\partial H}{\partial x_{k}}(s,x,DV(s,x))\dfrac{\partial V}{\partial x_{k}}(s,x)+\sum_{i=1}^{n}\dfrac{\partial H}{\partial p_{i}}(s,x,DV(s,x))\dfrac{\partial \varpi}{\partial x_{i}}(s,x)\\
    & +\sum_{k=1}^{n}\dfrac{\partial}{\partial x_{k}}\dfrac{d F}{d\nu}F\left(m_s\right)(x)\dfrac{\partial V}{\partial x_{k}}(s,x).
\end{align*}
Finally, we get the equation for $\varpi$ as follows:
\begin{equation}\label{eq:5-34}
    \left\{
    \begin{aligned}
        &-\dfrac{\partial \varpi}{\partial s}(s,x)+\mathcal{A}\varpi(s,x)+\dfrac{1}{2}\sum_{i,j,k=1}^{n}a_{ij}\dfrac{\partial^{2}V}{\partial x_{k}\partial x_{i}}\dfrac{\partial^{2}V}{\partial x_{k}\partial x_{j}}= D_{x}H(s,x,DV(s,x))\cdot DV(s,x)\\
        &\qquad +D_{p} H(s,x,DV(s,x))\cdot D_{x}\varpi(s,x)+D\dfrac{dF}{d\nu}\left(m_s\right)(x)\cdot DV(s,x),\quad s\in[t,T),\\
        &\varpi(T,x)=\dfrac{1}{2}\left|Dh(x)+D\dfrac{d}{d\nu}F_{T}\left(m_T\right)(x)\right|^{2},\quad x\in\brn.
    \end{aligned}
    \right.
\end{equation}
From Assumption~\ref{assumption:H}-(ii), we know that 
\begin{align*}
    |D_{x}H(s,DV(s,x),x)|,\ |D_{p}H(s,DV(s,x),x)|\ \le\  c(1+|x|)+c\sqrt{2\varpi(s,x)}, 
\end{align*}
and therefore, from Assumption~\ref{assumption:F}-(ii),
\begin{align*}
    D_{x}H(s,DV(s,x),x)\cdot DV(s,x)\le\ &  c(1+|x|)\sqrt{2\varpi(s,x)}+2c\varpi(s,x),\\
    D\dfrac{dF}{d\nu}\left(m_s\right)(x)\cdot DV(s,x)\le\ & c(1+|x|)\sqrt{2\varpi(s,x)},\\
    D_{p}H(s,DV(s,x),x).D\varpi(s,x)\le\ & c(1+|x|)|D\varpi(s,x)|+c\sqrt{2\varpi(s,x)}|D\varpi(s,x)|.
\end{align*}
Collecting these results and applying Young's inequality, we know that for any $\eta>0$,
\begin{equation}\label{eq:5-40}
\begin{aligned}
    -\dfrac{\partial \varpi}{\partial s}(s,x)+\mathcal{A}\varpi(s,x)\le\ &  c\eta|D\varpi(s,x)|^{2}+c\left(4+\dfrac{1}{\eta}\right)\varpi(s,x)+c\left(1+\dfrac{1}{\eta}\right)\left(1+|x|^{2}\right),\\
    \varpi(T,x)\le\ & |Dh(x)|^{2}+\left|D\dfrac{dF_{T}}{d\nu}(m_T)(x)\right|^{2}\leq c_T(1+|x|^{2}).
\end{aligned}
\end{equation}
According to the method of majorants, similar as in Proposition~\ref{prop5-1}, we introduce the equation 
\begin{equation}\label{eq:5-200}
\left\{
\begin{aligned}
    &-\dfrac{\partial \overline{z}}{\partial s}(s,x)+\mathcal{A}\overline{z}(s,x)=c\eta|D\overline{z}(s,x)|^{2}+c\left(4+\dfrac{1}{\eta}\right)\overline{z}(s,x)+C\left(1+\dfrac{1}{\eta}\right)\left(1+|x|^{2}\right),\quad s\in[t,T),\\
    &\overline{z}(T,x)=c_T \left(1+|x|^{2}\right),\quad x\in\brn,
\end{aligned}
\right.
\end{equation}
and we want a solution of \eqref{eq:5-200} of the form
\begin{equation}\label{eq:5-201}
    \overline{z}(s,x)=\dfrac{1}{2}\overline{\beta}(s)|x|^{2}+\overline{\mu}(s),
\end{equation}
with $\beta(s)$ and $\mu(s)$ being bounded on $[t,T]$. It can be easily shown that $\beta(s)$ should satisfies the following equation:
\begin{equation}\label{eq:5-202}
    \left\{
    \begin{aligned}
        &-\dfrac{1}{2}\overline{\beta}'(s)=c\left(1+\dfrac{1}{\eta}\right)+\dfrac{c}{2}\left(4+\dfrac{1}{\eta}\right)\overline{\beta}(s)+c\eta\left|\overline{\beta}(s)\right|^{2},\quad s\in[t,T),\\
        &\overline{\beta}(T)=2c_T;
    \end{aligned}
    \right.
\end{equation}
while $\mu(s)$ should satisfies the following equation:
\begin{equation}\label{eq:5-203}
    \left\{
    \begin{aligned}
        &-\overline{\mu}'(s)-\dfrac{1}{2}\text{tr }a\:\overline{\beta}(s)=c\left(1+\dfrac{1}{\eta}\right)+c\left(4+\dfrac{1}{\eta}\right)\overline{\mu}(s),\quad s\in[t,T),\\
        &\overline{\mu}(T)=c_T.
    \end{aligned}
    \right.
\end{equation}
The first equation in \eqref{eq:5-202} also writes
\begin{equation*}\label{eq:5-204}
    -\dfrac{1}{2}\overline{\beta}'(s)=c\left(1+\dfrac{1}{\eta}\right)+c\eta\left(\overline{\beta}(s)+\frac{1}{\eta}+\dfrac{1}{4\eta^2}\right)^{2}-\dfrac{c}{\eta}\left(1+\dfrac{1}{4\eta}\right)^{2},\quad s\in[t,T).
\end{equation*}
We define the function for $\eta$ as
\begin{align*}
    l(\eta):=\dfrac{1}{\eta}\left(1+\dfrac{1}{4\eta}\right)^{2}-1-\dfrac{1}{\eta},\quad \eta>0,
\end{align*}
and note the fact that this function is monotone decreasing from $+\infty$ to $-1$ for $l\in(0,+\infty)$. Hence we can choose $\eta^*>0$ such that $l\left(\eta^*\right)=0$, i. e. 
\begin{align*}
    \left(1+\dfrac{1}{4\eta^*}\right)^{2}-\eta^*-1=0.
\end{align*}
For this $\eta^*$, equation \eqref{eq:5-202} reads 
\begin{equation*}
    \left\{
    \begin{aligned}
        &-\dfrac{1}{2}\overline{\beta}'(s)=c\eta^*\left(\overline{\beta}(s)+\frac{1}{\eta^*}+\dfrac{1}{4\left(\eta^*\right)^2}\right)^{2},\quad s\in[t,T),\\
        &\overline{\beta}(T)=2c_T,
    \end{aligned}
    \right.
\end{equation*}
which means 
\begin{align*}
    \dfrac{d}{ds}\left[\dfrac{1}{\overline{\beta}(s)+\frac{1}{\eta^*}+\dfrac{1}{4\left(\eta^*\right)^2}}\right]=2c\eta^*,\quad s\in[t,T),
\end{align*}
and then 
\begin{equation}\label{eq:5-205}
    \dfrac{1}{\overline{\beta}(s)+\frac{1}{\eta^*}+\dfrac{1}{4\left(\eta^*\right)^2}}=\dfrac{1}{2c_T+\frac{1}{\eta^*}+\dfrac{1}{4\left(\eta^*\right)^2}}-2c\eta^*(T-s),\quad s\in[t,T]. 
\end{equation}
To keep $\overline{\beta}(s)$ bounded on $(t,T)$, we need to assume the following smallness condition for $T-t$:
\begin{equation}
    2c\eta^*(T-t)\left(2c_T+\frac{1}{\eta^*}+\dfrac{1}{4\left(\eta^*\right)^2}\right)<1. \label{eq:5-206}
\end{equation}
Then, after knowing $\overline{\beta}(s)$, we obtain $\overline{\mu}(s)$ from \eqref{eq:5-203}: 
\begin{align}
    \overline{\mu}(s)=\ & c_T\exp \left[C\left(4+\dfrac{1}{\eta^*}\right)(T-s)\right] \notag \\
    &+ \int_{s}^{T}\left[c\left(1+\dfrac{1}{\eta^*}\right)+\dfrac{1}{2}\text{tr}(a)\cdot \overline{\beta}(\tau)\right]\exp \left[c\left(4+\dfrac{1}{\eta^*}\right)(T-\tau)\right]d\tau.  \label{eq:5-207}
\end{align}
We now can give \textit{a priori} eatimate for $DV(s,x)$.

\begin{proposition}\label{prop5-3}
Under Assumptions~\ref{eq:5-8}, \ref{assumption:h}-(ii), \ref{assumption:F}-(ii) and \ref{assumption:H}, we make the smallness condition on $T-t$ such that \eqref{eq:5-206} holds. Define $\overline{z}(s,x)$ as in \eqref{eq:5-201}, \eqref{eq:5-205} and \eqref{eq:5-207}. We also assume the following condition at infinity:
\begin{equation}\label{eq:5-51}
    \limsup_{|x|\rightarrow+\infty}\left(\varpi(s,x)-\overline{z}(s,x)\right)\le 0.
\end{equation}
Then, we have 
\begin{equation}\label{eq:5-52}
    \varpi(s,x)-\overline{z}(s,x)\le 0,\quad (s,x)\in[t,T]\times\brn.
\end{equation}
\end{proposition}

\begin{proof}
Similar as Proposition \ref{prop5-1}, we denote by $\overline{\zeta}(s,x):=\varpi(s,x)-\overline{z}(s,x)$. From \eqref{eq:5-40} and \eqref{eq:5-200}, we know that $\overline{\zeta}(s,x)$ satisfies the relation 
\begin{align*}
    -\dfrac{\partial \overline{\zeta}}{\partial s}(s,x)+\mathcal{A}\overline{\zeta}(s,x)\le\ & c\left(4+\dfrac{1}{\eta^*}\right)\overline{\zeta}(s,x)+c\eta^*\left(|D\varpi(s,x)|^{2}-|D\overline{z}(s,x)|^{2}\right),\quad s\in[t,T),\\
    \overline{\zeta}(T,x)\le\ & 0.
\end{align*}
By applying the same reasoning as that done in the proof of Proposition~\ref{prop5-1} to the function $\overline{\zeta}(s,x)\exp \left[c\left(4+\dfrac{1}{\eta^*}\right)s\right]$, we obtain the statement \eqref{eq:5-52}. 
\end{proof}

\subsection{\textit{A Priori} Estimate for Regularity in Time}

In this part, we shall give \textit{a priori} estimates on the regularity of $V$ in time under assumptions in Proposition \ref{prop5-3}. We first write the HJB equation \eqref{eq:5-150} as 
\begin{equation}\label{eq:5-53}
    \left\{
    \begin{aligned}
        &-\dfrac{\partial V}{\partial s}(s,x)+\mathcal{A}V(s,x)=Q(s,x),\quad s\in[t,T),\\
        &V(T,x)=Q_{T}(x),\quad x\in\brn,
    \end{aligned}
    \right.
\end{equation}
with the functions $Q$ and $Q_T$ being defined as
\begin{equation*}\label{eq:5-54}
    \begin{aligned}
        Q(s,x)=\ & H(s,x,DV(s,x))+\dfrac{dF}{d\nu}\left(m_s\right)(x),\\
        Q_{T}(x)=\ & h(x)+\dfrac{dF_{T}}{d\nu}\left(m_T\right)(x)
    \end{aligned}
\end{equation*}
From the results in Proposition~\ref{prop5-3}, we can state that under Assumptions~\ref{eq:5-8}, \ref{assumption:h}-(ii), \ref{assumption:F}-(ii) and \ref{assumption:H}, the small time condition \eqref{eq:5-206}, and the condition at infinity \eqref{eq:5-51} (linear growth for $DV$), we have
\begin{equation}\label{eq:5-55}
    |Q(s,x)|\le C(T)\left(1+|x|^{2}\right),\quad \left|Q_{T}(x)\right|\le C(T)\left(1+|x|^{2}\right),\quad \left|DQ_{T}(x)\right|\le C(T)(1+|x|);
\end{equation}
here and in the rest of this article, we use $C$ to denote by a generic constant depending only on the parameters $\left(\text{tr}(a),c_T,c,\delta\right)$ in Assumptions~\ref{eq:5-8}, \ref{assumption:h}, \ref{assumption:F} and \ref{assumption:H} on the coefficients $\left(\sigma,h,F,F_T,H\right)$, and use $C(T)$ to denote a generic constant depending on the abovementioned constants and also on $T$ as well. We first give \textit{a priori} estimate on the integrability for the derivative $\frac{\dd V}{\dd s}$. 

\begin{proposition}\label{prop5-4}
Under assumptions in Proposition \ref{prop5-3}, we have the following estimate 
\begin{equation}\label{eq:5-56}
    \int_{t}^{T}\int_{\brn}\left|\dfrac{\partial V}{\partial s}(s,x)\right|^{2}\pi_{\gamma}(x)dxds\le C(T),
\end{equation}
where 
\begin{equation}\label{eq:5-100}
    \pi_{\gamma}(x)=\dfrac{1}{\left(1+|x|^{2}\right)^{\gamma}},\quad x\in\brn,\quad \gamma>\dfrac{n}{2}+2.
\end{equation}
\end{proposition}

\begin{proof}
From the equation \eqref{eq:5-53}, we can write 
\begin{align}
    &\int_{t}^{T}\int_{\brn}\left|\dfrac{\partial V}{\partial s}(s,x)\right|^{2}\pi_{\gamma}(x)dxds+\dfrac{1}{2}\sum_{i,j}^n a_{ij}\int_{t}^{T}\int_{\brn}\dfrac{\partial^{2}V}{\partial x_{i}\partial x_{j}}(s,x)\dfrac{\partial V}{\partial s}(s,x)\pi_{\gamma}(x)dxds \notag \\
    =\ & -\int_{t}^{T}\int_{\brn}Q(s,x)\dfrac{\partial V}{\partial s}(s,x)\pi_{\gamma}(x)dxds. \label{eq:5-100-1}
\end{align}
Note the fact that
\begin{align}
    &\dfrac{1}{2}\sum_{i,j}^n a_{ij}\int_{t}^{T}\int_{\brn}\dfrac{\partial^{2}V}{\partial x_{i}\partial x_{j}}(s,x)\dfrac{\partial V}{\partial s}(s,x)\pi_{\gamma}(x)dxds \notag \\
    =\ & -\dfrac{1}{2}\sum_{i,j}^n a_{ij}\int_{t}^{T}\int_{\brn}\dfrac{\partial V}{\partial x_{j}}(s,x)\dfrac{\partial^{2}V}{\partial s\partial x_{i}}(s,x)\pi_{\gamma}(x)dxds \notag \\
    &+\gamma\sum_{i,j}^n a_{ij}\int_{t}^{T}\int_{\brn}\dfrac{\partial V}{\partial x_{j}}(s,x)\dfrac{\partial V}{\partial s}(s,x)\dfrac{x_{i}}{1+|x|^{2}}\pi_{\gamma}(x)dxds, \label{eq:5-100-2}
\end{align}
and the first term of the right hand side of \eqref{eq:5-100-2} also satisfies
\begin{align}
    &-\dfrac{1}{2}\sum_{i,j}^n a_{ij} \int_{t}^{T}\int_{\brn}\dfrac{\partial V}{\partial x_{j}}(s,x)\dfrac{\partial^{2}V}{\partial s\partial x_{i}}(s,x)\pi_{\gamma}(x)dxds \notag \\
    =\ & -\frac{1}{4} \int_{t}^{T} \frac{d}{ds}\left[\sum_{i,j}^n a_{ij}\int_{\brn}\left(\dfrac{\partial V}{\partial x_{j}}(s,x)\dfrac{\partial V}{\partial x_{i}}(s,x)\right)\pi_{\gamma}(x) dx\right]ds \notag \\
    =\ & -\dfrac{1}{4}\sum_{i,j}^n a_{ij}\int_{\brn}\dfrac{\partial Q_{T}}{\partial x_{j}}(x)\dfrac{\partial Q_{T}}{\partial x_{i}}(x)\pi_{\gamma}(x)dx+\dfrac{1}{4}\sum_{i,j}^n a_{ij}\int_{\brn}\dfrac{\partial V}{\partial x_{j}}(t,x)\dfrac{\partial V}{\partial x_{i}}(t,x)\pi_{\gamma}(x)dx, \label{eq:5-100-3}
\end{align}
where we use the fact that the matrix $a$ is symmetric. Substituting \eqref{eq:5-100-2} and \eqref{eq:5-100-3} into \eqref{eq:5-100-1}, we have
\begin{align*}
    &\int_{t}^{T}\int_{\brn}\left|\dfrac{\partial V}{\partial s}(s,x)\right|^{2}\pi_{\gamma}(x)dxds  \\
    =\ & -\int_{t}^{T}\int_{\brn}Q(s,x)\dfrac{\partial V}{\partial s}(s,x)\pi_{\gamma}(x)dxds +\dfrac{1}{4}\sum_{i,j}^n a_{ij}\int_{\brn}\dfrac{\partial Q_{T}}{\partial x_{j}}(x)\dfrac{\partial Q_{T}}{\partial x_{i}}(x)\pi_{\gamma}(x)dx \\
    &  - \dfrac{1}{4}\sum_{i,j}^n a_{ij}\int_{\brn}\dfrac{\partial V}{\partial x_{j}}(t,x)\dfrac{\partial V}{\partial x_{i}}(t,x)\pi_{\gamma}(x)dx \\
    &-\gamma\sum_{i,j}^n a_{ij}\int_{t}^{T}\int_{\brn}\dfrac{\partial V}{\partial x_{j}}(s,x)\dfrac{\partial V}{\partial s}(s,x)\dfrac{x_{i}}{1+|x|^{2}}\pi_{\gamma}(x)dxds,
\end{align*}
then, by using Proposition~\ref{prop5-3} for $|DV(s,x)|^2$, estimates in \eqref{eq:5-55} and Young's inequality to the first and last term of the right hand side of the last inequality, %standard majorations,
we obtain that
\begin{align*}
    &\int_{t}^{T}\int_{\brn}\left|\dfrac{\partial V}{\partial s}(s,x)\right|^{2}\pi_{\gamma}(x)dxds\leq C(T)\int_{\brn}\left(1+|x|^{2}\right)^{2}\pi_{\gamma}(x)dx,
\end{align*}
and since $\gamma>\frac{n}{2}+2$, we obtain \eqref{eq:5-56}. 
\end{proof}

Another type of regularity in time follows from the regularity in space. Specifically, we give the following result, which shows the continuity of $V$ in time.

\begin{proposition}\label{prop:5-5}
Under assumptions in Proposition \ref{prop5-3}, we have the following estimate:
\begin{equation}\label{eq:5-500}
    \left|V\left(s_{2},x\right)-V\left(s_{1},x\right)\right|\le C(T)(1+|x|)\left(s_{2}-s_{1}\right)^{\frac{1}{2}}+C(T)(1+|x|^{2})\left(s_{2}-s_{1}\right),\quad t\le s_{1}<s_{2}\le T.
\end{equation}
\end{proposition}

\begin{proof}
The fastest way is to use the following Feynman-Kac formula for $V(s,x)$:
\begin{equation}\label{eq:5-501}
    V(s,x)=\e\left[\int_{s}^{T}Q\left(\tau,x+\sigma(w_\tau-w_s)\right)d\tau+Q_{T}\left(x+\sigma(w_T-w_s)\right)\right].
\end{equation}
Take $s_{2}>s_{1}$, we can write that
\begin{align*}
    V(s_{2},x)-V(s_{1},x)=I+II, 
\end{align*}
where
\begin{align*}
    I:=\ & \e \bigg[\int_{s_{2}}^{T}\left[Q\left(\tau,x+\sigma(w_\tau-w_{s_{2}})\right)-Q\left(\tau,x+\sigma(w_\tau-w_{s_{1}})\right)\right]d\tau\\
    &\quad + Q_{T}\left(x+\sigma(w_T-w_{s_{2}})\right)-Q_{T}\left(x+\sigma(w_T-w_{s_{1}})\right)\bigg],\\
    II:=\ & -\e \left[\int_{s_{1}}^{s_{2}}Q\left(\tau,x+\sigma(w_\tau-w_{s_{1}})\right)d\tau\right].
\end{align*}
From \eqref{eq:5-501} and the independence of the Brownian increment, we know that
\begin{align*}
    I=\ & V(s_{2},x)- \e \bigg[\int_{s_{2}}^{T} Q\left(\tau,\left(x+\sigma(w_{s_2}-w_{s_{1}})\right)+\sigma(w_\tau-w_{s_{2}})\right) d\tau  \\
    &\qquad\qquad\qquad + Q_{T}\left(\left(x+\sigma(w_{s_2}-w_{s_{1}})\right)+\sigma(w_T-w_{s_{2}})\right)\bigg]\\
    %=\ & V(s_{2},x)-\e \left[V\left(s_2, x+\sigma(w_{s_{2}}-w_{s_{1}})\right)\right]\\
    =\ & \e\left[V(s_{2},x)-V\left(s_2, x+\sigma(w_{s_{2}}-w_{s_{1}})\right) \right]\\
    =\ & \e\left[\int_0^1 DV(s_2, x+\theta\sigma (w_{s_{2}}-w_{s_{1}}))\cdot \sigma (w_{s_{2}}-w_{s_{1}})d\theta\right],
\end{align*}
then, from the estimate on $DV$ in Proposition~\ref{prop5-3}, it easily follows 
\begin{align}
    |I|\le\ & C(T)(1+|x|) \sqrt{\e\left[ \left|w_{s_{2}}-w_{s_{1}}\right|^2\right]} \le C(T)(1+|x|)(s_{2}-s_{1})^{\frac{1}{2}}.\label{eq:5-502}
\end{align}
Furthermore, from \eqref{eq:5-55}, we can also obtain that
\begin{align}
    |II|\le\ &  \e\left[\int_{s_{1}}^{s_{2}}\left|Q\left(\tau,x+\sigma(w_\tau-w_{s_{1}})\right)\right|d\tau\right] \notag \\
    \le\ & C(T) \e\left[\int_{s_{1}}^{s_{2}}\left(1+|x|^2+\left|w_\tau-w_{s_{1}}\right|^2\right)d\tau\right] \notag \\
    \le\ &  C(T)(1+|x|^{2})(s_{2}-s_{1}).\label{eq:5-503}
\end{align}
From \eqref{eq:5-502} and \eqref{eq:5-503}, we obtain \eqref{eq:5-500}.
\end{proof}

\subsection{More \textit{A Priori} Estimates on $\dfrac{\partial V}{\partial x_{i}}(s,x)$}\label{subsec:V_x_i}

In this part, we aim to give further the \textit{a priori} estimate on the derivative 
\begin{equation}\label{def:zeta-V_xi}
    \mathcal{J}^i (s,x):=\dfrac{\partial V}{\partial x_{i}}(s,x),\quad i=1,2,\dots,n,
\end{equation}
and to simplify the notation, we omit the superscript $i$ of $\mathcal{J}^i$. We require Assumptions~\ref{eq:5-8}, \ref{assumption:h}, \ref{assumption:F} and \ref{assumption:H} hold true, and assume the small time condition \eqref{eq:5-206}, and the condition at infinity \eqref{eq:5-51}. We know from Proposition~\ref{prop5-3} that
\begin{equation}\label{eq:5-105}
    |\mathcal{J}(s,x)|\le C(T)(1+|x|),\quad (s,x)\in[t,T]\times\brn;
\end{equation}
here we note that $C(T)$ also depends on the dimension $n$. By differentiating the HJB equation \ref{eq:5-150} with respect to $x_{i}$, we know that
\begin{align*}
    -\dfrac{\partial}{\partial s}\dfrac{\partial V}{\partial x_{i}}(s,x)+\mathcal{A}\;\dfrac{\partial V}{\partial x_{i}}(s,x)=\ & D_{p}H(s,x,DV(s,x))\cdot D\dfrac{\partial V}{\partial x_{i}}(s,x)\\
    &+\dfrac{\partial H}{\partial x_{i}}(s,x,DV(s,x))+\dfrac{\partial}{\partial x_{i}}\dfrac{dF}{d\nu}\left(m_s\right)(x),\quad s\in[t,T),\\
    \dfrac{\partial V}{\partial x_{i}}(T,x)=\ & \dfrac{\partial h}{\partial x_{i}}(x)+\dfrac{\partial}{\partial x_{i}}\dfrac{d F_{T}}{d\nu}\left(m_T\right)(x),\quad x\in\brn.
\end{align*}
We obtain the following equation for $\mathcal{J}$:
\begin{equation}\label{eq:5-101}
    \left\{
    \begin{aligned}
        &-\dfrac{\partial\mathcal{J}}{\partial s}(s,x)+\mathcal{A}\mathcal{J}(s,x)-g(s,x)\cdot D\mathcal{J}(s,x)=l(s,x),\quad s\in[t,T),\\
        &\mathcal{J}(T,x)=l_{T}(x),\quad x\in\brn,
    \end{aligned}
    \right.
\end{equation}
where
\begin{equation}\label{define_g}
    \begin{aligned}
        g(s,x):=\ & D_{p}H(s,x,DV(s,x)),\\
        l(s,x):=\ & \dfrac{\partial}{\partial x_{i}}\dfrac{dF}{d\nu}(m_s)(x)+\dfrac{\partial H}{\partial x_{i}}(s,x,DV(s,x)),\\
        l_{T}(x):=\ & \dfrac{\partial h}{\partial x_{i}}(x)+\dfrac{\partial}{\partial x_{i}}\dfrac{d F_{T}}{d\nu}\left(m_T\right)(x).
    \end{aligned}
\end{equation}
From Assumptions~\ref{assumption:h}, \ref{assumption:F} and \ref{assumption:H} and Proposition~\ref{prop5-3}, we know that functions $(g,l,l_T)$ defined in \eqref{define_g} have the properties
\begin{align}\label{condition_g}
    |g(s,x)|\le C(T)(1+|x|)\quad |l(s,x)|\leq C(T)(1+|x|),\quad |l_{T}(x)|\leq C(1+|x|),\quad |Dl_{T}(x)|\leq C.
\end{align}
The equation \eqref{eq:5-101} has the difficulty that $g(s,x)$ can be unbounded and non-differentiable. Here, we consider \eqref{define_g} as a generic equation with more general coefficients $(g,l,l_T)$ without restricted in \eqref{define_g}; instead, we assume they are general coefficients satisfying the following assumptions, which is slightly weaker than \eqref{condition_g}:
\begin{equation}\label{eq:5-140}
\begin{aligned}
    &|g(s,x)|\leq C(T)(1+|x|),\quad |l(s,x)|\leq C(T)(1+|x|^{2}),\quad s\in[t,T],\\
    &|l_{T}(x)|\leq C(T)(1+|x|^{2}),\quad |Dl_{T}(x)|\le C(T)(1+|x|),\quad x\in\brn.
\end{aligned}
\end{equation}
We now give the following \textit{a priori} estimate for the solution of the generic equation \eqref{eq:5-101}.

\begin{proposition}\label{prop5-5}
Under Assumption~\ref{eq:5-8} and conditions in \eqref{eq:5-140}, the solution $\mathcal{J}$ of equation \eqref{eq:5-101} satisfies the following estimates
\begin{align}
    &|\mathcal{J}(s,x)|\le C(T)\left(1+|x|^{2}\right), \label{eq:5-130}\\
    &\int_{t}^{T}\int_{\brn}|D\mathcal{J}(s,x)|^{2}\pi_{\gamma}(x)dxds\le C(T), \label{eq:5-106} \\
    &\int_{t}^{T}\int_{\brn}\left|\dfrac{\partial \mathcal{J}}{\partial s}(s,x)\right|^{2}\pi_{\gamma+1}(x)dxds\le C(T), \label{eq:5-107}
\end{align}
with $\gamma\geq\dfrac{n}{2}+3$ and $\pi_\gamma$ being defined in \eqref{eq:5-100}.
\end{proposition}

\begin{proof}
We define 
\begin{align*}
    \widetilde{\mathcal{J}}(s,x)=\dfrac{\mathcal{J}(s,x)\exp(\alpha s)}{1+|x|^{2}},\quad (s,x)\in[t,T]\times\brn,
\end{align*}
where $\alpha>0$ is to be chosen. A simple calculation gives
\begin{align*}
    &-\dfrac{\partial\widetilde{\mathcal{J}}}{\partial s}(s,x)+\mathcal{A}\widetilde{\mathcal{J}}(s,x)+\left(\alpha-\dfrac{\text{tr}(a)}{1+|x|^{2}}-\dfrac{2x\cdot g(s,x)}{1+|x|^{2}}\right)\widetilde{\mathcal{J}}(s,x)\\
    =\ & D\widetilde{\mathcal{J}}(s,x)\cdot \left[g(s,x)+\dfrac{2ax}{1+|x|^{2}}\right]+\dfrac{l(s,x)\exp(\alpha s)}{1+|x|^{2}},\quad s\in[t,T),\\
    &\widetilde{\mathcal{J}}(T,x)=\dfrac{l_{T}(x)\exp(\alpha T)}{1+|x|^{2}},\quad x\in\brn
\end{align*}
From conditions in \eqref{eq:5-140}, we know that
\begin{align*}
    \dfrac{x\cdot g(s,x)}{1+|x|^{2}},\ \dfrac{l(s,x)}{1+|x|^{2}},\ \dfrac{l_{T}(x)}{1+|x|^{2}}\ \le\ C(T),
\end{align*}
therefore, these exits $\alpha,\alpha_0>0$ depending on $\text{tr}(a)$ and $C(T)$ such that
\begin{equation*}\label{eq:5-142}
    \alpha-\dfrac{\text{tr}(a)}{1+|x|^{2}}-\dfrac{2x\cdot g(s,x)}{1+|x|^{2}}\geq\alpha_{0}>0.
\end{equation*}
Similar to the ways of establishing Propositions~\ref{prop5-1} and \ref{prop5-3}, by the method of majorants, %by using the maximum principle considerations, 
we know that
\begin{equation*}\label{eq:5-143}
    \left|\widetilde{\mathcal{J}}(s,x)\right|\le \exp(\alpha T)\cdot \max\left\{\sup_{(s,x)\in[t,T]\times\brn}\dfrac{|l(s,x)|}{\alpha_0\left(1+|x|^{2}\right)},\quad \sup_{x\in\brn}\dfrac{|l_{T}(x)|}{1+|x|^{2}}\right\},
\end{equation*}
which (together with \eqref{eq:5-140}) immediately implies \eqref{eq:5-130}. We next prove \eqref{eq:5-106}. From equation \eqref{eq:5-101}, we can compute that
\begin{align*}
    &\int_{t}^{T}\int_{\brn}\dfrac{\partial \mathcal{J}}{\partial s}(s,x)\mathcal{J}(s,x)\pi_{\gamma}(x)dxds\\
    =\ &-\dfrac{1}{2}\sum_{i,j}^n a_{ij}\int_{t}^{T}\int_{\brn}\dfrac{\partial^{2}\mathcal{J}}{\partial x_{i}\partial x_{j}}(s,x)\mathcal{J}(s,x)\pi_{\gamma}(x)dxds\\
    &-\int_{t}^{T}\int_{\brn}g(s,x)\cdot D\mathcal{J}(s,x)\mathcal{J}(s,x)\pi_{\gamma}(x)dxds - \int_{t}^{T}\int_{\brn}l(s,x)\mathcal{J}(s,x)\pi_{\gamma}(x)dxds.
\end{align*}
Note that 
\begin{align*}
    \int_{t}^{T}\int_{\brn}\dfrac{\partial \mathcal{J}}{\partial s}(s,x)\mathcal{J}(s,x)\pi_{\gamma}(x)dxds=\int_{t}^{T}\frac{d}{ds}\left[\frac{1}{2}\int_{\brn}\mathcal{J}^2(s,x)\pi_{\gamma}(x)dx\right]ds,
\end{align*}
therefore, from the formula for integration by parts, we have
\begin{align}
    & \dfrac{1}{2}\int_{\brn}l_{T}^{2}(x)\pi_{\gamma}(x)dx - \dfrac{1}{2}\int_{\brn}\mathcal{J}^{2}(t,x)\pi_{\gamma}(x)dx \notag \\
    =\ &\dfrac{1}{2}\sum_{i,j}^n a_{ij}\int_{t}^{T}\int_{\brn}\dfrac{\partial\mathcal{J}}{\partial x_{j}}(s,x)\dfrac{\partial\mathcal{J}}{\partial x_{i}}(s,x)\pi_{\gamma}(x)dxds \notag \\
    &-\gamma\sum_{i,j}^n a_{ij}\int_{t}^{T}\int_{\brn}\dfrac{\partial\mathcal{J}}{\partial x_{j}}(s,x)\mathcal{J}(s,x)\dfrac{x_{i}}{1+|x|^{2}}\pi_{\gamma}(x)dxds \notag \\
    &-\int_{t}^{T}\int_{\brn}g(s,x)\cdot D\mathcal{J}(s,x)\mathcal{J}(s,x)\pi_{\gamma}(x)dxds-\int_{t}^{T}\int_{\brn}l(s,x)\mathcal{J}(s,x)\pi_{\gamma}(x)dxds. \notag
\end{align}
Then, from the Young's inequality, we know that for any $\epsilon>0$, we have
\begin{align}
    &\dfrac{1}{2}\sum_{i,j}^n a_{ij}\int_{t}^{T}\int_{\brn}\dfrac{\partial\mathcal{J}}{\partial x_{j}}(s,x)\dfrac{\partial\mathcal{J}}{\partial x_{i}}(s,x)\pi_{\gamma}(x)dxds \notag \\
    \le\ & \gamma\sum_{i,j}^n a_{ij}\int_{t}^{T}\int_{\brn}\left|\dfrac{\partial\mathcal{J}}{\partial x_{j}}(s,x)\right| \cdot\left|\mathcal{J}(s,x)\right|\cdot\dfrac{|x_{i}|}{1+|x|^{2}}  \pi_{\gamma}(x) dxds \notag \\
    &+\int_{t}^{T}\int_{\brn}|g(s,x)|\cdot|\mathcal{J}(s,x)|\cdot |D\mathcal{J}(s,x)| \pi_{\gamma}(x)dxds \notag \\
    &+\int_{t}^{T}\int_{\brn} |l(s,x)|\cdot |\mathcal{J}(s,x)|\pi_{\gamma}(x)dxds +\dfrac{1}{2}\int_{\brn}l_{T}^{2}(x)\pi_{\gamma}(x)dx \notag \\
    \le\ & \epsilon \int_{t}^{T}\int_{\brn}\left|D\mathcal{J}(s,x)\right|^2 \pi_{\gamma}(x) dxds \notag \\
    &+ \frac{C(\text{tr}(a),\gamma,n)}{\epsilon}\int_{t}^{T}\int_{\brn} \left|\mathcal{J}(s,x)\right|^2\cdot\left[ |g(s,x)|^2 + \dfrac{|x|^2}{\left(1+|x|^{2}\right)^2}  \right]\pi_{\gamma}(x) dxds \notag \\
    &+\int_{t}^{T}\int_{\brn} |l(s,x)|\cdot |\mathcal{J}(s,x)|\pi_{\gamma}(x)dxds +\dfrac{1}{2}\int_{\brn}l_{T}^{2}(x)\pi_{\gamma}(x)dx, \notag 
\end{align}
where $C(\text{tr}(a),\gamma,n)$ is a constant depending only on $(\text{tr}(a),\gamma,n)$. By using Assumption~\ref{eq:5-8}, we can choose $\epsilon$ small enough to obtain
\begin{align}
    &\int_{t}^{T}\int_{\brn}\left|D\mathcal{J}(s,x)\right|^2 \pi_{\gamma}(x) dxds \notag \\
    \le\ & C\int_{t}^{T}\int_{\brn} \left|\mathcal{J}(s,x)\right|^2\cdot\left[ |g(s,x)|^2 + 1  \right]\pi_{\gamma}(x) dxds \notag \\
    &+C\int_{t}^{T}\int_{\brn} |l(s,x)|^2\pi_{\gamma}(x)dxds +C\int_{\brn}l_{T}^{2}(x)\pi_{\gamma}(x)dx. \label{prop5-5-1}
\end{align}
From \eqref{eq:5-140} and \eqref{eq:5-130}, and the fact that $\gamma\geq\dfrac{n}{2}+3$, we know that
\begin{align*}
    \int_{\brn}\left|\mathcal{J}(s,x)\right|^2\cdot\left[ |g(s,x)|^2 + 1  \right]\pi_{\gamma}(x)dx \le\ &  C(T) \int_{\brn}\left(1+|x|^{2}\right)^{3}\pi_{\gamma}(x)dx\le C(T),\\
    \int_{t}^{T}\int_{\brn} |l(s,x)|^2\pi_{\gamma}(x)dxds +\int_{\brn}l_{T}^{2}(x)\pi_{\gamma}(x)dx \le\ & C(T) \int_{\brn}\left(1+|x|^{2}\right)^{2}\pi_{\gamma}(x)dx\le C(T).
\end{align*}
Substituting the last estimates into \eqref{prop5-5-1}, we obtain \eqref{eq:5-106}. We next prove \eqref{eq:5-107}. By multiplying both sides of equation \eqref{eq:5-101} by $\dfrac{\partial\mathcal{J}}{\partial s}(s,x)\pi_{\gamma+1}(x)$ and then integrating it, we know that
\begin{align}
    0=\ & \int_{t}^{T}\int_{\brn}\left|\dfrac{\partial\mathcal{J}}{\partial s}(s,x)\right|^{2}\pi_{\gamma+1}(x)dxds+\dfrac{1}{2}\sum_{i,j}^n a_{ij}\int_{t}^{T}\int_{\brn}\dfrac{\partial^{2}\mathcal{J}}{\partial x_{i}\partial x_{j}}(s,x)\dfrac{\partial\mathcal{J}}{\partial s}(s,x)\pi_{\gamma+1}(x)dxds \notag \\
    &+\int_{t}^{T}\int_{\brn}g(s,x)\cdot D\mathcal{J}(s,x)\dfrac{\partial\mathcal{J}}{\partial s}(s,x)\pi_{\gamma+1}(x)dxds+\int_{t}^{T}\int_{\brn}l(s,x)\dfrac{\partial\mathcal{J}}{\partial s}(s,x)\pi_{\gamma+1}(x)dxds. \label{prop5-5-2}
\end{align}
The simple integration by parts gives
\begin{align}
    &\dfrac{1}{2}\sum_{i,j}^n a_{ij}\int_{t}^{T}\frac{d}{ds}\left[\int_{\brn} \dfrac{\partial\mathcal{J}}{\partial x_{j}}(s,x)\dfrac{\partial\mathcal{J}}{\partial x_i}(s,x)\pi_{\gamma+1}(x)dx\right]ds \notag \\
    =\ & \sum_{i,j}^n a_{ij} \int_{t}^{T}\int_{\brn}\dfrac{\partial\mathcal{J}}{\partial x_{j}}(s,x)\dfrac{\partial^2\mathcal{J}}{\partial s \partial x_{i}}(s,x)\pi_{\gamma+1}(x)dxds \notag \\
    =\ & - \sum_{i,j}^n a_{ij} \int_{t}^{T} \int_\brn \bigg[\dfrac{\partial^{2}\mathcal{J}}{\partial x_{i}\partial x_{j}}(s,x)\dfrac{\partial\mathcal{J}}{\partial s}(s,x)\pi_{\gamma+1}(x) \notag \\
    &\qquad\qquad\qquad\qquad - 2(\gamma+1) \dfrac{\partial\mathcal{J}}{\partial x_{j}}(s,x)\dfrac{\partial\mathcal{J}}{\partial s}(s,x)  \frac{x_i}{1+|x|^2}\pi_{\gamma+1}(x) \bigg]dxds. \label{prop5-5-3}
\end{align}
Combining \eqref{prop5-5-2} and \eqref{prop5-5-3}, we have
\begin{align*}
    &\int_{t}^{T}\int_{\brn}\left|\dfrac{\partial\mathcal{J}}{\partial s}(s,x)\right|^{2}\pi_{\gamma+1}(x)dxds \\
    =\ &\dfrac{1}{4}\sum_{i,j}^n a_{ij}\int_{\brn}\dfrac{\partial l_{T}}{\partial x_{j}}(x)\dfrac{\partial l_{T}}{\partial x_{i}}(x)\pi_{\gamma+1}(x)dx-\dfrac{1}{4}\sum_{i,j}^na_{ij}\int_{\brn}\dfrac{\partial\mathcal{J}}{\partial x_{j}}(t,x)\dfrac{\partial\mathcal{J}}{\partial x_{i}}(t,x)\pi_{\gamma+1}(x)dx\\
    &-(\gamma+1)\sum_{i,j}^n a_{ij}\int_{t}^{T}\int_{\brn}\dfrac{\partial\mathcal{J}}{\partial x_{j}}(s,x)\dfrac{\partial\mathcal{J}}{\partial s}(s,x)\dfrac{x_{i}}{1+|x|^{2}}\pi_{\gamma+1}(x)dxds\\
    &- \int_{t}^{T}\int_{\brn}g(s,x)\cdot D\mathcal{J}(s,x)\dfrac{\partial\mathcal{J}}{\partial s}(s,x)\pi_{\gamma+1}(x)dxds - \int_{t}^{T}\int_{\brn}l(s,x)\dfrac{\partial\mathcal{J}}{\partial s}(s,x)\pi_{\gamma+1}(x)dxds.
\end{align*}
Then, from Young's inequality, we know that for any $\epsilon>0$, we have
\begin{align*}
    &\int_{t}^{T}\int_{\brn}\left|\dfrac{\partial\mathcal{J}}{\partial s}(s,x)\right|^{2}\pi_{\gamma+1}(x)dxds \\
    \le\ & \epsilon \int_{t}^{T}\int_{\brn}\left|\dfrac{\partial\mathcal{J}}{\partial s}(s,x)\right|^{2}\pi_{\gamma+1}(x)dxds +\dfrac{1}{4}\sum_{i,j}^n a_{ij}\int_{\brn}\dfrac{\partial l_{T}}{\partial x_{j}}(x)\dfrac{\partial l_{T}}{\partial x_{i}}(x)\pi_{\gamma+1}(x)dx \\
    &+ \frac{C(\text{tr}(a),\gamma,n)}{\epsilon} \int_{t}^{T}\int_{\brn}\bigg[|D\mathcal{J}(s,x)|^2\cdot \left(|g(s,x)|^2+\dfrac{|x|^2}{\left(1+|x|^{2}\right)^2}\right) +|l(s,x)|^2 \bigg] \pi_{\gamma+1}(x)dxds,
\end{align*}
and by choosing $\epsilon$ small enough, we obtain
\begin{align}
    &\int_{t}^{T}\int_{\brn}\left|\dfrac{\partial\mathcal{J}}{\partial s}(s,x)\right|^{2}\pi_{\gamma+1}(x)dxds \notag \\
    \le\ & C \int_{\brn}|Dl_{T}(x)|^2\pi_{\gamma+1}(x)dx +C \int_{t}^{T}\int_{\brn} |l(s,x)|^2 \pi_{\gamma+1}(x)dxds \notag \\
    &+ C \int_{t}^{T}\int_{\brn} |D\mathcal{J}(s,x)|^2\cdot \left(|g(s,x)|^2+1\right) \pi_{\gamma+1}(x)dxds.  \label{prop5-5-4}
\end{align}
From \eqref{eq:5-140} and \eqref{eq:5-106}, we know that
\begin{align*}
    &(i) \int_{\brn}|Dl_{T}(x)|^2\pi_{\gamma+1}(x)dx \le C(T) \int_{\brn} \pi_{\gamma}(x)dxds \le C(T),\\
    &(ii)\int_{t}^{T}\int_{\brn} |l(s,x)|^2 \pi_{\gamma+1}(x)dxds \le C(T)\int_{t}^{T}\int_{\brn} \left(1+|x|^2\right)^2 \pi_{\gamma+1}(x)dxds \\
    &\;\qquad\qquad\qquad\qquad\qquad\qquad\qquad \le C(T)\int_{t}^{T}\int_{\brn} \left(1+|x|^2\right) \pi_{\gamma}(x)dxds \le C(T),\\
    &(iii)\int_{t}^{T}\int_{\brn} |D\mathcal{J}(s,x)|^2\cdot \left(|g(s,x)|^2+1\right) \pi_{\gamma+1}(x)dxds \\
    &\;\qquad\qquad\qquad\qquad\qquad\qquad\qquad \le C(T) \int_{t}^{T}\int_{\brn} |D\mathcal{J}(s,x)|^2\cdot \left(1+|x|^2\right) \pi_{\gamma+1}(x)dxds\\
    &\;\qquad\qquad\qquad\qquad\qquad\qquad\qquad = C(T) \int_{t}^{T}\int_{\brn} |D\mathcal{J}(s,x)|^2 \pi_{\gamma}(x)dxds \le C(T). 
\end{align*}
Substituting the last estimates into \eqref{prop5-5-4}, we obtain \eqref{eq:5-107}.
\end{proof}

We next give the following key result, which is a local H\"older estimate for the function $\mathcal{J}$.

\begin{theorem}\label{theo5-1}
Under Assumption~\ref{eq:5-8} and conditions in \eqref{eq:5-140}, for any bounded smooth domain $\mathcal{O}$ of $\brn$, there exists $\delta_{0}<\frac{1}{2}$ depending only on $\mathcal{\overline{O}}$, $T-t$ and the parameters in assumptions on the coefficients, such that the solution $\mathcal{J}$ of equation \eqref{eq:5-101} satisfies
\begin{equation}\label{eq:5-108}
    \sup_{x_{1},x_{2}\in\mathcal{\overline{O}},\ s_{1},s_{2}\in[t,T]}\dfrac{\left|\mathcal{J}(s_1, x_{1})-\mathcal{J}(s_2, x_{2})\right|}{|x_{1}-x_{2}|^{2\delta_{0}}+|s_{1}-s_{2}|^{\delta_{0}}}\le C(T,\mathcal{O}),
\end{equation}
where $C(T,\mathcal{O})$ is a constant depending only on $(T,\mathcal{O})$ and the parameters in assumptions on the coefficients.
\end{theorem}

\begin{proof}
The proof follows a similar idea as our previous result in \cite[Section 7]{4}, which gives a local H\"older estimate for class of quasilinear parabolic equations associated with BSDEs. Here, the equations, settings and assumptions are a bit different. The proof is highly technical and will be given in a separate article. 
\end{proof}

\begin{remark}
    The H\"older estimate for the linear parabolic equation \eqref{eq:5-101} under conditions in \eqref{eq:5-140} is used not only here for the HJB equation, but will also be used for the study for the linear functional derivatives for the HJB equation and the master equation (and associated auxiliary PDEs), which makes it important for the study for the mean field type control problems. 
\end{remark}

As a direct consequence of Proposition~\ref{prop5-5}, Theorem~\ref{theo5-1} and the properties in \eqref{condition_g}, we know that the function defined in \eqref{def:zeta-V_xi} satisfies \eqref{eq:5-130}, \eqref{eq:5-106}, \eqref{eq:5-107} and \eqref{eq:5-108}, under Assumptions~\ref{eq:5-8}, \ref{assumption:h}, \ref{assumption:F} and \ref{assumption:H}, the small time condition \eqref{eq:5-206}, and the condition at infinity \eqref{eq:5-51}. 

\subsection*{Summary of Properties of $V(s,x)$ and $m_s$}

We here collect all the \textit{a priori} estimates of the solution $V(s,x)$ of the HJB equation \eqref{eq:5-150}. They are independent of the probability measure $m_s$. Under Assumptions~\ref{eq:5-8}, \ref{assumption:h}, \ref{assumption:F} and \ref{assumption:H}, we assume that $T-t$ satisfies the limitations \eqref{eq:5-14} and \eqref{eq:5-206}, and assume naturally that the conditions at infinity \eqref{eq:5-30} and \eqref{eq:5-51} (quadratic growth for $V$ and linear growth for $DV$) hold true. The function $V(s,x)$ satisfies the following \textit{a priori} estimates: for $s,s_1,s_2\in[t,T]$ and $x\in\brn$,
\begin{table}[!ht]
    \centering
    \resizebox{\textwidth}{!}{
    \begin{tabular}{|c|c|c|}\hline
        \multirow{2}{*}{Growth condition} & $|V(s,x)|\le C(T)(1+|x|^{2})$ \\ \cline{2-2}
        \multirow{2}{*}{} & $|DV(s,x)|\le C(T)(1+|x|)$ \\ \hline
        Continuity in time & $\displaystyle \left|V(s_2, x)-V(s_1,x)\right|\le C(T)(1+|x|)\left|s_{2}-s_{1}\right|^{\frac{1}{2}}+C(T)\left(1+|x|^{2}\right)\left|s_{2}-s_{1}\right|$ \\ \hline
        \multirow{3}{*}{Integrability condition} & $\displaystyle \int_{t}^{T}\int_{\brn}\left|\dfrac{\partial V}{\partial s}(s,x)\right|^{2}\pi_{\gamma}(x)ds\le C(T)$,\qquad $\gamma\geq\dfrac{n}{2}+3$ \\ \cline{2-2}
        \multirow{3}{*}{} & $\displaystyle \int_{t}^{T}\int_{\brn}\left|D^{2}V(s,x)\right|^{2}\pi_{\gamma}(x)ds\leq C(T)$,\qquad $\gamma\geq\dfrac{n}{2}+3$ \\ \cline{2-2}
        \multirow{3}{*}{} & $\displaystyle \int_{t}^{T}\int_{\brn}\left|\dfrac{\partial}{\partial s}DV(s,x)\right|^{2}\pi_{\gamma+1}(x)ds\leq C(T)$,\qquad $\gamma\geq\dfrac{n}{2}+3$ \\ \hline
        \multirow{2}{*}{Local H\"older regularity} & $\forall\mathcal{O}\subset\brn$ bounded and smooth, $\exists\ \delta_{0}<\frac{1}{2}$ depending on $\mathcal{\overline{O}}$ and $T$, s.t.  \\ 
        \multirow{2}{*}{} & $\displaystyle \sup_{x_{1},x_{2}\in\mathcal{\overline{O}},\ s_{1},s_{2}\in[t,T]}\dfrac{\left|DV(s_{1},x_{1})-DV(s_{2},x_{2})\right|}{\left|x_{1}-x_{2}\right|^{2\delta_{9}}+\left|s_{1}-s_{2}\right|^{\delta_{0}}}\leq C(T,\mathcal{O})$ \\ \hline
    \end{tabular} }
    \caption{\textit{A priori} estimates for $V$}\label{table:1}
\end{table}
\noindent 
Here, the constants $C(T)$ (and $C(T,\mathcal{O})$) depend only on $T$ (and also $\mathcal{O}$), and the parameters in Assumptions~\ref{eq:5-8}, \ref{assumption:h}, \ref{assumption:F} and \ref{assumption:H} on the coefficients; and they do not depend on the probability measure $m_s$.

The probability $m_s$ is uniquely defined in terms of $DV(s,x),\ t\le s\le T$ by the formula 
\begin{equation}\label{eq:5-114}
    \int_{\brn}\varphi(x)dm_s(x)=\int_{\brn}\Psi(t,x)dm(x)
\end{equation}
for all continuous and bounded  function $\varphi(x)$ on $\brn$, where $\Psi(\tau,x),\ t\le\tau\le s$ is the solution of the equation 
\begin{equation}\label{eq:5-115}
\left\{
\begin{aligned}
    &-\dfrac{\partial\Psi(\tau,x)}{\partial\tau}+\mathcal{A}\Psi(\tau,x)=D\Psi(\tau,x)\cdot D_{p}H(\tau,x,DV(\tau,x)),\quad t\le \tau<s,\\
    &\Psi(s,x)=\varphi(x),\quad x\in\brn.
\end{aligned}
\right.
\end{equation}

\subsection{Solution of the Local-in-time HJB Equation}\label{subsec:main}

In this subsection, we aim to show the existence of a solution $V(x,s$) of the HJB equation \eqref{eq:5-150}. 
%We recall that $m,t$ are initial conditions and are mostly omitted. We may use $V(s,x)$ in lieu of $V_{mt}(s,x).$ The number $T-t$ is not too large. 
The functional space is the Hilbert space $L^{2}\left(t,T;H_{\pi_{\gamma}}^{1}(\brn)\right)$ of functions $\varphi:[t,T]\times\brn\to \br$ such that 
\begin{align*}
    \int_{t}^{T}|\varphi(s,x)|^{2}\pi_{\gamma}(x)dxds+\int_{t}^{T}|D\varphi(s,x)|^{2}\pi_{\gamma}(x)dxds<+\infty,
\end{align*}
with $\pi_{\gamma}(x)$ being given by \eqref{eq:5-100} for $\gamma\geq\frac{n}{2}+3$. We define the subset of $L^{2}(t,T;H_{\pi_{\gamma}}^{1}(\brn))$, denoted by $B_{T}$, of functions satisfying conditions in Table~\ref{table:1}. The following result shows that the HJB equation \eqref{eq:5-150} has a solution in $B_{T}$. 

\begin{theorem}\label{theo5-2} 
Under Assumptions~\ref{eq:5-8}, \ref{assumption:h}, \ref{assumption:F} and \ref{assumption:H}, we assume the interval $T-t$ is restricted by conditions \eqref{eq:5-14} and \eqref{eq:5-206}. We add to equation \eqref{eq:5-150} conditions at infinity \eqref{eq:5-30} and \eqref{eq:5-51} (quadratic growth for $V(s,x)$ and linear growth for $DV(s,x))$. Then, there exists a solution $V$ of the HJB equation \eqref{eq:5-150} in the subset $B_{T}$ of the Hilbert space $L^{2}\left(t,T;H_{\pi_{\gamma}}^{1}(\brn)\right)$,
defined by conditions in Table~\ref{table:1}. 
\end{theorem}

\begin{proof}
We first show that $B_{T}$ is a compact subset of $L^{2}\left(t,T;H_{\pi_{\gamma}}^{1}(\brn)\right)$. It is obviously a convex closed bounded set, therefore, $B_{T}$ is weakly compact. So if $V^{h}(s,x)$ is a sequence of $B_{T}$, we can extract a subsequence, still denoted by $V^{h}$, such that there exists $V\in B_T$,
\begin{equation*}\label{eq:5-121}
    V^{h}\rightharpoonup V;\quad DV^{h}\rightharpoonup DV;\quad D^{2}V^{h}\rightharpoonup D^{2}V,\quad \text{weakly in } L^{2}(t,T;L_{\pi_{\gamma}}^{2}(\brn)),
\end{equation*}
where the last property comes from the integrability condition for $D^2V^h$ in Table~\ref{table:1}. However, from the growth estimate for $DV^h$ and the continuity of $V^h$ in time in Table~\ref{table:1}, it follows that the sequence $V^{h}$ is uniformly continuous on compact sets. And from the growth estimate for $V^h$ in Table~\ref{table:1}, we know that the sequence $V^{h}$ is also uniformly bounded on compact sets. Thus, from Arzela-Ascoli theorem, there exists a subsequence (still denoted by $V^{h}$) which uniformly converges to $V$ on any compact sets. Moreover, from the growth estimate for $V^h$ and $V$, we know that for $\gamma>\frac{n}{2}+3$,
\begin{align*}
    \int_{x\in\brn,\;|x|\geq R} \left|V^{h}(s,x)-V(s,x)\right|^{2}\pi_{\gamma}(x)dxds \le C(T)\int_{x\in\brn,\;|x|\geq R} \frac{\left(1+|x|^2\right)^2}{\left(1+|x|^2\right)^\gamma} dx,
\end{align*}
which convergence to $0$ uniformly in $h$ as $R\to+\infty$. Therefore, together with the uniformly convergence of $V^{h}$ in the set $\{x:|x|\le R\}$,  we know that $V^{h}$ converges to $V$ in norm in $L^{2}(t,T;L_{\pi_{\mu}}^{2}(\brn))$. Similarly, for $DV^h$, from the H\"older estimate and the uniformly growth condition on any compact sets (see Table~\ref{table:1}), again by the Arzela-Ascoli theorem, we know that $DV^h$ uniformly converges to $V^h$ on any compact sets. Moreover, since
\begin{align*}
    \int_{x\in\brn,\;|x|\geq R} \left|DV^{h}(s,x)-DV(s,x)\right|^{2}\pi_{\gamma}(x)dxds \le C(T)\int_{x\in\brn,\;|x|\geq R} \frac{\left(1+|x|\right)^2}{\left(1+|x|^2\right)^\gamma} dx,
\end{align*}
which convergence to $0$ uniformly in $h$ as $R\to+\infty$, we know that $DV^{h}$ also converges to $ DV(s,x)$ in norm in $L^{2}(t,T;L_{\pi_{\gamma}}^{2}(\brn))$. This completes the proof that $B_{T}$ is compact in $L^{2}(t,T;H_{\pi_{\gamma}}^{1}(\brn))$. 

We next check that a solution $V$ of the HJB equation \eqref{eq:5-150} can be considered as a fixed point for a map from $B_{T}$ to $B_{T}$. Indeed, for a function $V\in B_{T}$, we associate to it the probability measure $m^{m,t}_s$ defined by \eqref{eq:5-114}-\eqref{eq:5-115}, and we then define $V^{m,t}(s,x)$ by \eqref{eq:5-150}. Here it is convenient to keep the notation $V^{m,t}(s,x)$ to make the difference with the initial $V(s,x)$, which originated $m^{m,t}_s$. From results in Subsections~\ref{subsec:priori:V}-\ref{subsec:V_x_i} and the fact that $B_{T}$ itself do not depend on $m^{m,t}_s$, we know that the functional $V\mapsto V^{mt}$ maps clearly from $B_{T}$ to $B_{T}$. It remains to show that this map is continuous. It is composed of the map $V\mapsto m^{m,t}$ and the map $m^{m,t}\mapsto V^{m,t}$. Suppose that we have a sequence of measures $m^{h,m,t}$ such that $m^{h,m,t}_s$ converges to $ m^{m,t}_s$ in $\pr_2(\brn)$ for any $s\in[t,T]$. Note the fact that $m^{h,m,t}$ appears only in the terms $\frac{dF}{d\nu}\left(m^{h,m,t}_s\right)(x)$ and $\frac{dF_T}{d\nu}\left(m^{h,m,t}_T\right)(x)$ of the equation \eqref{eq:5-150}, then from the continuity condition for $\frac{dF}{d\nu}(m)(x)$ and $\frac{dF_T}{d\nu}(m)(x)$ in $m$ in Assumption~\ref{assumption:F}, we get easily that $V^{h,m,t}(s,x)$ converges to $V^{m,t}$ in $L^{2}\left(t,T;H_{\pi_{\gamma}}^{1}(\brn)\right)$. Then the continuity of the map $V\mapsto V^{m,t}$ from $B_{T}$ to $B_{T}$ boils down to proving that 
\begin{equation}\label{eq:5-122}
    V\mapsto m^{m,t}_s \quad \text{is continuous from }B_{T} \text{ to}\ \pr_2(\brn),\quad \forall s\in[t,T].
\end{equation}
From the metric of $\pr_2(\brn)$, we call \eqref{eq:1-100}, it suffices to prove the following: for any continuously differentiable function $\varphi$ satisfying $|\varphi(x)|\le C\left(1+|x|^{2}\right)$, consider the solution of the equation
\begin{equation}\label{eq:5-123}
\left\{
\begin{aligned}
    &-\dfrac{\partial\Psi}{\partial\tau}(\tau,x)+\mathcal{A}\Psi(\tau,x)=D\Psi(\tau,x)\cdot D_{p}H(\tau,x,DV(\tau,x)),\quad t\le\tau<s<T,\\
    &\Psi(s,x)=\varphi(x),\quad x\in\brn,
\end{aligned}
\right.
\end{equation}
then the functional 
\begin{equation}\label{eq:5-124}
    V\mapsto\int_{\brn}\Psi(t,x)dm(x)
\end{equation}
is continuous on $B_{T}$. We set $g(\tau,x):=D_{p}H(\tau,x,DV(\tau,x))$, and from Assumption~\ref{assumption:H} and the gwoth condition for $DV$ (see Table~\ref{table:1}), we know that $|g(\tau,x)|\le C(T)(1+|x|)$. Then, \eqref{eq:5-123} is an equation of the type \eqref{eq:5-101}, so from Proposition~\ref{prop5-5} and  we have the following estimates
\begin{align}
    &|\Psi(\tau,x)|\le C(s,T)(1+|x|^{2}),\quad t\le \tau\leq s, \quad x\in\brn, \label{eq:5-128}\\
    &\int_{t}^{s}\int_{\brn}|D\Psi(\tau,x)|^{2}\pi_{\gamma}(x)dxd\tau\le C(s,T), \label{eq:5-125} \\
    &\int_{t}^{s}\int_{\brn}\left|\dfrac{\partial\Psi}{\partial\tau}(\tau,x)\right|^{2}\pi_{\gamma+1}(x)dxds\le C(s,T), \label{eq:5-126}
\end{align}
with $\gamma\geq\dfrac{n}{2}+3$; and from Theorem~\ref{theo5-1}, we know that for any bounded smooth domain $\mathcal{O}\subset\brn$, there exists $0<\delta_{0}<\frac{1}{2}$ depending only on $\mathcal{\overline{O}}$, $T$ and the parameters in assumptions on the coefficients, such that
\begin{equation}\label{eq:5-127}
    \sup_{x_{1},x_{2}\in\mathcal{\overline{O}},\ s_{1},s_{2}\in[t,s]}\dfrac{|\Psi(s_{1},x_{1})-\Psi(s_{2},x_{2})|}{|x_{1}-x_{2}|^{2\delta_{9}}+|s_{1}-s_{2}|^{\delta_{0}}}\le C(s,T,\mathcal{O}).
\end{equation}
Suppose that $V^{h}$ is now a sequence in $B_{T}$ which converges to $V$ in $B_{T}$. We also set $g^{h}(\tau,x):=D_{p}H\left(\tau,x,DV^{h}(\tau,x)\right)$. Since $V^{h}\in B_{T}$, we have $\left|DV^{h}(\tau,x)\right|\le C(T)\left(1+|x|\right)$ and thus $\left|g^{h}(\tau,x)\right|\le C(T)(1+|x|)$. Noting that the constant $C(T)$ in independent of $h$, from the continuity of $D_pH$, the uniformly conditions for $V^h$ in Table~\ref{table:1}, the Arzela-Ascoli theorem and the convergence of $V^h$ in $B_T$, we know that 
\begin{align}\label{convergence:g}
    g^{h}(\tau,x)\rightarrow g(\tau,x)\ \text{pointwise up to a subsequence}.
\end{align}
Let $\Psi$ corresponds to $V$ and $\Psi^{h}(\tau,x)$ corresponding to $V^{h}$, which satisfies  the equation
\begin{equation}\label{eq:5-180}
\left\{
\begin{aligned}
    &-\dfrac{\partial\Psi^{h}}{\partial\tau}(\tau,x)+\mathcal{A}\Psi^{h}(\tau,x)=D\Psi^{h}(\tau,x)\cdot g^{h}(\tau,x),\quad t\le\tau<s<T,\\
    &\Psi^{h}(s,x)=\varphi(x),\quad x\in\brn.
\end{aligned}
\right.
\end{equation}
Again from Proposition~\ref{prop:5-5} and Theorem~\ref{theo5-1}, we know that the estimates \eqref{eq:5-128}, \eqref{eq:5-125}, \eqref{eq:5-126} and \eqref{eq:5-127} hold for $\Psi^{h}$ with the constants independent of $h$. Therefore, from the Arzela-Ascoli theorem, we know that
\begin{align}\label{eq:5-180'}
    \Psi^{h}(\tau,x)&\rightarrow\Psi(\tau,x)\ \text{pointwise}.
\end{align}
Recall that $B_T$ is a compact subset of $L^{2}\left(t,T;H_{\pi_{\gamma}}^{1}(\brn)\right)$, so up to a subsequence, we have 
\begin{align*}
    D\Psi^{h}&\rightharpoonup D\Psi\ \text{weakly in }\ L^{2}\left(t,s;L_{\pi_{\gamma}}^{2}(\brn)\right).
\end{align*}
From \eqref{eq:5-123} and \eqref{eq:5-180}, we know that
\begin{align*}
    &\dfrac{1}{2}\int_{\brn}\left|\left(\Psi^{h}-\Psi\right)(t,x)\right|^{2}\pi_{\gamma}(x)dx +\dfrac{1}{2}\sum_{i,j}^n a_{ij}\int_{t}^{s}\int_{\brn}\dfrac{\partial \left(\Psi^{h}-\Psi\right)}{\partial x_{j}}(\tau,x)\dfrac{\partial\left(\Psi^{h}-\Psi\right)}{\partial x_{i}}(\tau,x)\pi_{\gamma}(x)dxd\tau \notag \\
    =\ & \gamma\sum_{i,j}^n a_{ij}\int_{t}^{s}\int_{\brn}\dfrac{\partial\left(\Psi^{h}-\Psi\right)}{\partial x_{j}}(\tau,x)\left(\Psi^{h}-\Psi\right)(\tau,x)\dfrac{x_{i}}{1+|x|^{2}}\pi_{\gamma}(x)dxd\tau \notag \\
    &+\int_{t}^{s}\int_{\brn} \left(\Psi^{h}-\Psi\right)(\tau,x)\cdot \left[D\Psi^{h}(\tau,x) g^{h}(\tau,x) -D\Psi(\tau,x) g(\tau,x)\right] \pi_{\gamma}(x) dxds, 
\end{align*}
and from Assumption~\ref{eq:5-8} and the Cauchy–Schwarz inequality, we know that
\begin{align}
    &\int_{t}^{s}\int_{\brn} \left|D\Psi^{h}(\tau,x)-D\Psi(\tau,x)\right|^2 \pi_{\gamma}(x)dxd\tau \notag \\
    \le\ & C \bigg[\int_{\brn}\left|\left(\Psi^{h}-\Psi\right)(t,x)\right|^{2}\pi_{\gamma}(x)dx \notag \\
    &\quad + \sqrt{\int_{t}^{s}\int_{\brn} \left|D\Psi^{h} (\tau,x) - D \Psi (\tau,x)\right|^2 \pi_{\gamma}(x) dxd\tau} \notag \\
    &\quad\qquad \cdot \sqrt{\int_{t}^{s}\int_{\brn} \left|\Psi^{h}(\tau,x)-\Psi(\tau,x)\right|^2\dfrac{|x|^2}{(1+|x|^{2})^2}\pi_{\gamma}(x)dxd\tau} \notag \\
    &\quad +\sqrt{\int_{t}^{s}\int_{\brn} \left|D\Psi^{h}(\tau,x) g^{h}(\tau,x) -D\Psi(\tau,x)g(\tau,x)\right|^2 \frac{\pi_{\gamma}(x)}{1+|x|^2} dxds} \notag \\
    &\qquad \cdot \sqrt{\int_{t}^{s}\int_{\brn} \left|\Psi^{h}(\tau,x)-\Psi(\tau,x)\right|^2\left(1+|x|^2\right) \pi_{\gamma}(x) dxds} \bigg]. \label{eq:5-180''}
\end{align}
From the pointwise convergence of $\Psi^h$ in \eqref{eq:5-180} and the boundedness condition \eqref{eq:5-126}, %and the uniform growth estimate for $g^h$ %and the dominated convergence theorem, 
we know that the right hand side of \eqref{eq:5-180''} convergence to $0$ as $h$ goes to $+\infty$, therefore, we have
\begin{equation*}\label{eq:5-171}
    D\Psi^{h}\rightarrow D\Psi\ \text{ in }\ L^{2}\left(t,s;L_{\pi_{\gamma}}^{2}(\brn)\right).
\end{equation*}
Then, from equations \eqref{eq:5-123} and \eqref{eq:5-180}, the pointwise convergence and uniform growth estimate for $g^h$ in \eqref{convergence:g} and the growth estimates \eqref{eq:5-128}, \eqref{eq:5-125}, \eqref{eq:5-126} and \eqref{eq:5-127} for $\Psi^{h}$, we know that 
\begin{align*}
    \dfrac{\partial\Psi^{h}}{\partial\tau}&\rightharpoonup\dfrac{\partial\Psi}{\partial\tau}\ \text{weakly in }\ L^{2}\left(t,s;L_{\pi_{\gamma+1}}^{2}(\brn)\right). 
\end{align*}
That means we can now pass to the limit and $\Psi$ is the solution of \eqref{eq:5-123} which corresponds to the limit $V$ of $V^{h}$. Finally, from the uniform estimate \eqref{eq:5-128} for $\Psi^{h}$, we can assert, by using Lebesgue's dominated convergence theorem, recalling that $m\in\pr_2(\brn)$,
\begin{equation*}\label{eq:5-172}
    \int_{\brn}\Psi^{h}(t,x)dm(x)\rightarrow\int_{\brn}\Psi(t,x)dm(x),
\end{equation*}
which means, from Lemma~\ref{lem4-1},
\begin{align*}
    \int_{\brn}\varphi(x)dm^{h,m,t}_s\rightarrow\int_{\brn}\varphi(x)dm^{m,t}_s,
\end{align*}
and these altogether show that the map \eqref{eq:5-124} is continuous. Then, from \eqref{eq:1-100}, we know that the property \eqref{eq:5-122} holds true, which completes the proof of this main result. 
\end{proof}

\subsection{About Global-in-time Solvability of HJB equation}\label{subsec:global}

The last subsection gives a local-in-time solvability for $V$ when the time duration $T-t$ is not large. In this subsection, we aim to obtain a global-in-time solution. From Theorem \ref{theo5-2}, the restrictions \eqref{eq:5-14} and \eqref{eq:5-206} on $T-t$ are required for \textit{a priori} estimates for the the growth conditions for $|V|$ and $|DV|$ in Propositions \ref{prop5-1} and \ref{prop5-3}, respectively; so far, we did not impose any convexity assumptions on the cost coefficient functions. Indeed, if there is certain convexity on cost functions, we are capable of establishing the global-in-time solvability as as follows; by then, we do not require the conditions \eqref{eq:5-14} and \eqref{eq:5-206} on $T-t$ anymore.

\begin{assumption}[Convexity]\label{assumption:convex}
    (i) The running cost function $l$ satisfies Assumption~\ref{assumption:l}; and it is convex in $x$ and is strongly convex in $v$, such that there exists $\lambda>0$, for any $s\in[0,T]$ and $x,x',v,v'\in\brn$,
    \begin{align}\label{condition:lambda}
        l(s,x',v')-l(s,x,v)\geq D_xl(s,x,v)\cdot (x'-x)+D_vl(s,x,v)\cdot (v'-v)+\lambda|v'-v|^2. 
    \end{align}
    (ii) The terminal cost function $h$ is convex in $x$.\\
    (iii) For the other running cost functions $F$ and $F_T$, their linear functional derivatives $\frac{dF}{d\nu}(m)(x)$ and $\frac{dF_T}{d\nu}(m)(x)$ are convex in $x$.
\end{assumption}

\begin{remark}
    The convexity condition (iii) in Assumption~\ref{assumption:convex} can be viewed as part of the monotonicity conditions for the Bellman equation or master equation; and we refer to our another work \cite{AB12} for more detailed discussion on various monotonicity conditions for MFTC problems and MFGs.
\end{remark}

We shall first give the global-in-time \textit{a priori} estimate for the growth of $|DV|$, and then that for $|V|$, which altogether show that the small-time conditions \eqref{eq:5-14} and \eqref{eq:5-206} are not necessary in the presence of Assumption~\ref{assumption:convex}.

\subsubsection{\textit{A Priori} Estimate for the growth condition for $|DV|$}

Instead of the using approaches in Subsection \ref{subsec:DV} to consider the equation for $DV$ directly, we here use the derivative representation $P^{m,t}_s(x)=DV\left(s,Y^{m,t}_s(x)\right)$ in \eqref{P_DV'} in order to give the growth condition for $|DV|$. Recall that $Y^{m,t}_s(x)$ and $P^{m,t}_s(x)$ satisfy the following forward-backward system of stochastic differential equations arising from the maximum principle \eqref{maximum_principle}:
\begin{equation}\label{global:DV:1}
    \left\{
    \begin{aligned}
        &Y^{m,t}_s(x)=x+\int_t^s D_{p}H\left(\tau,Y^{m,t}_\tau(x),P^{m,t}_\tau(x)\right)d\tau+\sum_{j=1}^{n}\sigma^{j} \left(w^{j}_s-w^{j}_t\right),\\
        &P^{m,t}_s(x)= Dh\left(Y^{m,t}_T(x)\right)+D\dfrac{dF_{T}}{d\nu}\left(m_T\right)\left(Y^{m,t}_T(x)\right) - \sum_{j=1}^{n}\int_s^T Q^{m,t}_\tau(x) dw_\tau\\
        &\quad\qquad\qquad +\int_s^T \left[D_{x}H\left(\tau,Y^{m,t}_\tau(x),P^{m,t}_\tau(x)\right)+D\dfrac{dF}{d\nu}\left(m_\tau\right)\left(Y^{m,t}_\tau(x)\right)\right] d\tau.
    \end{aligned}
    \right.
\end{equation}
The well-posedness of the forward-backward systems of equations \eqref{global:DV:1} associated with this mean field type control problems has been studied in our previous works \cite{AB10,AB5}; or otherwise, it can also be obtained by the parallel arguments via PDE methods such as those found in \cite{Delarue2002,Ma-Yong}. Indeed, the convexity conditions in Assumption~\ref{assumption:convex} are crucial in proving the unique existence of the solution $\left(Y^{m,t}_s(x),P^{m,t}_s(x),Q^{m,t}_s(x)\right)$ and their $L^2$ growth estimate in $x$. Under the convexity assumption~\ref{assumption:convex}, and the regularity assumptions on the coefficients in the last subsections, FBSDEs \eqref{global:DV:1} has a unique adapted solution $\left(Y^{m,t}_s(x),P^{m,t}_s(x),Q^{m,t}_s(x)\right)$, and they satisfy
\begin{equation*}\label{estimate:YPQ}
    \sup_{t\le s\le T}\e\left[\left| { Y^{m,t}_s(x)}\right|^2 \right] + \sup_{t\le s\le T}\e\left[\left| { P^{m,t}_s(x)}\right|^2\right]+\e\left[\int_t^T \left| { Q^{m,t}_s(x)}\right|^2 ds \right] \le C(T,\lambda)\left(1+|x|^2\right);
\end{equation*}
for instance, we refer to \cite[Section 8]{AB10} for a proof. Then, together with the fact $DV\left(s,Y^{m,t}_s(x)\right)=P^{m,t}_s(x)$, we have the following \textit{a priori} estimate for the growth of $|DV|$ under Assumption~\ref{assumption:convex}, which does not require the time period $T-t$ to be small in a sense that relieves Condition \eqref{eq:5-206}. 

\begin{lemma}\label{lem:global:1}
Under Assumptions~\ref{eq:5-8}, \ref{assumption:h}, \ref{assumption:F} and \ref{assumption:convex}, we have
\begin{align}\label{lem:global:1-0}
    |DV(s,x)|\le C(\lambda,T) (1+|x|),\quad (s,x)\in[t,T]\times\brn,
\end{align}
where $C(\lambda,T)$ is a constant depending only on $T$, the constant $\lambda$ in the condition \eqref{condition:lambda}, and parameters in Assumptions~\ref{eq:5-8}, \ref{assumption:h}, \ref{assumption:F} and \ref{assumption:H}.
\end{lemma}

This result can also be obtained by the PDE arguments in \cite{Delarue2002,Ma-Yong}. Besides, we here also provide an alternative perspective on the statement $DV\left(s,Y^{m,t}_s(x)\right)=P^{m,t}_s(x)$ by using the HJB equation of $V$ directly. We define the process 
\begin{align*}
    \widetilde{P}^{m,t}_s(x):=DV\left(s,Y^{m,t}_s(x)\right),\ (s,x)\in[t,T]\times\brn,
\end{align*}
and it follows from the terminal condition of $V$ that
\begin{align}\label{BSDE:PQ-1}
    \widetilde{P}^{m,t}_T(x)= Dh\left(Y^{m,t}_T(x)\right) + D\frac{dF_T}{d\nu}\left(m_T\right)\left(Y^{m,t}_T(x)\right)=P^{m,t}_T(x).
\end{align}
Then, by applying It\^o's formula on $\widetilde{P}^{m,t}_s(x)$ and using the HJB equation \eqref{eq:5-150}, similar as in \eqref{BSDE:P'}, we have
\begin{align}
    d\widetilde{P}^{m,t}_s(x)=\ &  \sum_{j=1}^{n}\widetilde{Q}^{m,t,j}_s(x) dw^j_s -\left[D_{x}H\left(s,Y^{m,t}_s(x),\widetilde{P}^{m,t}_s(x)\right)+D\dfrac{dF}{d\nu}\left(m_s\right)\left(Y^{m,t}_s(x)\right)\right] ds \notag \\
    & + D^2V\left(s,Y^{m,t}_s(x)\right) \left[D_{p}H\left(s,Y^{m,t}_s(x),P^{m,t}_s(x)\right)-D_{p}H\left(s,Y^{m,t}_s(x),\widetilde{P}^{m,t}_s(x)\right) \right] ds, \label{BSDE:PQ-2}
\end{align}
where
\begin{align*}
    \widetilde{Q}^{m,t,j}_s(x):=D^2V\left(s,Y^{m,t}_s(x)\right)\sigma^j,\quad j=1,2,\dots,n.
\end{align*}
Since $Y^{m,t}_s(x)$ and $P^{m,t}_s(x)$ are already known, \eqref{BSDE:PQ-1} and \eqref{BSDE:PQ-2} form a BSDE for the couple $\left(\widetilde{P}^{m,t}_s(x),\widetilde{Q}^{m,t}_s(x)\right)$, and it is easy to see that $\left(P^{m,t}_s(x),Q^{m,t}_s(x)\right)$ is also a solution to this BSDE. Then, from the uniqueness result for the BSDE, we know that $\widetilde{P}^{m,t}_s(x)=P^{m,t}_s(x)$, which gives that $DV\left(s,Y^{m,t}_s(x)\right)=P^{m,t}_s(x)$.

\subsubsection{\textit{A Priori} Estimate for the growth condition for $|V|$}\label{subsubsec:global:|V|}

Under Assumptions~\ref{eq:5-8}, \ref{assumption:h}, \ref{assumption:F}, \ref{assumption:H} and \ref{assumption:convex}, from \eqref{lem:global:1-0}, we have
\begin{align*}
    |H(s,x,DV(s,x))|\le c(1+|x|^2)+\frac{\delta}{2} |DV(s,x)|^2 \le C(\lambda,T)(1+|x|^2).
\end{align*}
Then, as in \eqref{eq:5-10}, from the HJB equation \eqref{eq:5-150}, we know that
\begin{align*}
    &\left|-\dfrac{\partial V}{\partial s}(s,x)+\mathcal{A}V(s,x)\right| \le C(\lambda,T)(1+|x|^{2}).
\end{align*}
Similar as \eqref{eq:5-11}, this lead to the following majorant equation 
\begin{align*}
    -\dfrac{\partial}{\partial s}z(s,x)+\mathcal{A}z(s,x)&= C(\lambda,T)\left(1+|x|^{2}\right),\\
    z(T,x)&=2c_{T}(1+|x|^{2}).
\end{align*}
Again, as in \eqref{eq:5-12}, we search for the solution of the form 
\begin{align}\label{lem:global:2-1}
    z(s,x)=\beta(s)\dfrac{|x|^{2}}{2}+\mu(s),\quad s\in[t,T], 
\end{align}
with some $\beta(s),\mu(s)$ being bounded on $(t,T)$. It is easy to see that $\beta(s)$ satisfies
\begin{align}\label{lem:global:2-2}
    \beta(s)=4c_T + 2C(\lambda,T)(T-s),\quad s\in[t,T], 
\end{align}
which is bounded on $[t,T]$; and then, we also have
\begin{equation}\label{lem:global:2-3}
    \mu(s)=2c_{T}+\int_{s}^{T}\left\{C(\lambda,T)+[2c_T + C(\lambda,T)(T-\tau)]\text{tr}(a)\right\}d\tau,\quad s\in[t,T], 
\end{equation}
which is also bounded on $[t,T]$. Then, following the method in the proof of Proposition~\ref{prop5-1}, we have the following result, which does not require the time period $T-t$ to be small, and thus avoids from using Condition \eqref{eq:5-14}. The proof is exactly same as that for  Proposition~\ref{prop5-1}, and is omitted here.

\begin{lemma}\label{lem:global:2}
Under Assumptions~\ref{eq:5-8}, \ref{assumption:h}, \ref{assumption:F} and \ref{assumption:convex}, we define $z(s,x)$ as in \eqref{lem:global:2-1}, \eqref{lem:global:2-2} and \eqref{lem:global:2-3}. Then, we have the following estimate 
\begin{equation*}\label{lem:global:2-0}
    -z(s,x)\le V(s,x)\leq z(s,x),\quad (s,x)\in[t,T]\times\brn.
\end{equation*}
\end{lemma}

Now, we have the following global-in-time well-posedness result.

\begin{theorem}\label{theo5-3} 
Under Assumptions~\ref{eq:5-8}, \ref{assumption:h}, \ref{assumption:F} and \ref{assumption:convex}, there exists a solution $V$ of the HJB equation \eqref{eq:5-150} in the Hilbert space $L^{2}\left(t,T;H_{\pi_{\gamma}}^{1}(\brn)\right)$ satisfying the regularity conditions in Table~\ref{table:1}. 
\end{theorem}

\begin{proof}
Note that Assumption~\ref{assumption:convex} include Assumption~\ref{assumption:l}, and therefore it can imply Assumption~\ref{assumption:H}-(ii). Then, this result is a consequence of the above arguments in this subsection and Theorem~\ref{theo5-2}.
\end{proof}

Theorems~\ref{theo5-2} and \ref{theo5-3} respectively give the local-in-time and global-in-time solvability of the HJB equation \eqref{eq:5-150}, which are based on the usage of the well-known Schauder fixed point theorem. For the uniqueness result, we refer to the method in our previous work \cite{AB6,AB10,AB5} for the uniqueness results for equations arsing in the mean field theory; and this result is not limited within the probabilistic methods. Similar approach can be used for the HJB equation in this article under our settings via our analytical perspective, and we shall not repeat further here.

\section{From HJB Equation to Mean Field Type Control}\label{sec:m}

In this section, we use the well-posedness result of the HJB equation \eqref{eq:5-150} proven in the last section to further give the characterization of the measure flow $m^{m,t}$, and also recover the master equation and the Bellman equation for the  mean field type control problem, and finally give the optimal control. In this section, we always make assumptions in Theorem~\ref{theo5-2} or in Theorem~\ref{theo5-3}, to guarantee both the local-in-time and the global-in-time well-posedness of $V^{m,t}$.

\subsection{Characterization of the Probability $m^{m,t}_s$}\label{subsec:m}

We first define the process 
\begin{equation*}\label{eq:5-61}
\begin{aligned}
    M^{x,m,t}_s=\exp\bigg(&\int_{t}^{s}\sigma^{-1}\; D_p H\left(\tau, X^{t}_\tau(x),DV^{m,t}\left(\tau, X^{t}_\tau(x) \right)\right)\cdot dw(\tau)\\
    &-\dfrac{1}{2}\int_{t}^{s}\left|\sigma^{-1}\;D_p H\left(\tau,X^{t}_\tau(x),DV^{m,t}\left(\tau, X^{t}_\tau(x) \right)\right)\right|^{2}d\tau\bigg)
\end{aligned}
\end{equation*}
where 
\begin{align}\label{def:SDE}
    X^{t}_\tau(x)=x+\sum_{j=1}^n \sigma^{j}\left(w^{j}_\tau-w^{j}_t\right),\quad \tau\in[t,T].
\end{align}
To shorten notation, we write 
\begin{equation*}\label{eq:5-62}
    g^{m,t}(s,x)=\sigma^{-1}\;D_p H\left(s,x,DV^{m,t}(s,x)\right),\quad (s,x)\in[t,T]\times\brn. 
\end{equation*}
From Assumption~\ref{assumption:H}-(ii) and Theorem~\ref{theo5-2}, we can assert that there exits $\overline{g}>0$ such that
\begin{equation}\label{eq:5-63}
    \left|g^{m,t}(s,x)\right|\leq\overline{g}(1+|x|),\quad (s,x)\in[t,T]\times\brn.
\end{equation}
We begin by giving the following results.

\begin{lemma}
\label{lem5-1}The process $M^{x,m,t}_s$ is a $\mathcal{W}_{t}$-martingale, and it satisfies
\begin{equation}\label{eq:5-64}
    \e \left[M^{x,m,t}_s\right]=1. 
\end{equation}
\end{lemma}

\begin{proof}
We define
\begin{align*}
    \tau_{N}:=\inf \left\{s>t\ |M^{x,m,t}_s>N \right\},
\end{align*}
then, the process $M^{x,m,t}_{s\land\tau_{N}}$ is the solution of the following SDE
\begin{equation}\label{eq:5-65}
\begin{aligned}
    M^{x,m,t}_{s\land\tau_{N}}=1+\int_{t}^{s\land\tau_{N}}M^{x,m,t}_\tau g^{m,t}\left(\tau, X^{t}_\tau(x) \right)dw_\tau,\quad s\in(t,T],
\end{aligned}
\end{equation}
and hence $\e \left[M^{x,m,t}_{s\land\tau_{N}}\right]=1$. Noting that the process in \eqref{eq:5-65} is locally finite almost surely, therefore, we know that $\tau_{N}$ goes to $+\infty$ when 
as $N$ goes to $+\infty$, and thus $M^{x,m,t}_{s\land\tau_{N}}\rightarrow M^{x,m,t}_s$
as $N\rightarrow+\infty.$ By Fatou's Lemma, we obtain 
\begin{equation}
    \e \left[M^{x,m,t}_s\right]\leq 1. \label{eq:5-66}
\end{equation}
By applying It\^o's formula on $\dfrac{M^{x,m,t}_\tau}{1+\epsilon M^{x,m,t}_\tau}$, taking expectation and integrating on $\tau\in [t,s]$, we can write
\begin{equation}\label{eq:5-63'}
    \e\left[\dfrac{M^{x,m,t}_s}{1+\epsilon M^{x,m,t}_s}\right]=\dfrac{1}{1+\epsilon}-\e\left[\int_{t}^{s}\dfrac{\epsilon \left|M^{x,m,t}_\tau\right|^{2}\left|g^{m,t}\left(\tau, X^{t}_\tau(x) \right)\right|^{2}}{\left|1+\epsilon M^{x,m,t}_\tau \right|^{3}}d\tau \right]. 
\end{equation}
From \eqref{eq:5-63}, we get the fact that
\begin{align}\label{eq:5-63''}
    \dfrac{\epsilon\left|M^{x,m,t}_\tau\right|^{2}\left|g^{m,t}\left(\tau, X^{t}_\tau(x) \right)\right|^{2}}{\left|1+\epsilon M^{x,m,t}_\tau \right|^{3}}\le\ &  \dfrac{\epsilon\left|M^{x,m,t}_\tau\right|}{\left|1+\epsilon M^{x,m,t}_\tau \right|^{3}}\cdot M^{x,m,t}_\tau\cdot  \overline{g}^2\left(1+\left|X^{t}_\tau(x)\right| \right)^{2} \notag \\
    \le\ &  \frac{2\overline{g}^2}{3} M^{x,m,t}_\tau \left(1+ \left|X^{t}_\tau(x)\right|^2\right).
\end{align}
We first make the following claim, which will be proven as below:
\begin{equation}
    \e \left[M^{x,m,t}_s \left|X^{t}_s(x)\right|^{2}\right]\leq C(T). \label{eq:5-67}
\end{equation}
With this claim of  \eqref{eq:5-67}, we know that the right hand side of \eqref{eq:5-63''} is bounded in the $L^2$-norm.
%\begin{align*}
    %\e\left[\int_{t}^{s}\dfrac{\epsilon \left|M^{x,m,t}_\tau\right|^{2}\left|g^{m,t}\left(\tau, X^{t}_\tau(x) \right)\right|^{2}}{\left|1+\epsilon M^{x,m,t}_\tau \right|^{3}}d\tau \right]\le C \e\left[\int_{t}^{s} M^{x,m,t}_\tau \left(1+ \left|X^{t}_\tau(x)\right|^2\right) d\tau \right]\le C(T).
%\end{align*}
Then, from the dominated convergence theorem, we deduce that 
\begin{align*}
    \lim_{\epsilon\to 0}\e\left[\int_{t}^{s}\dfrac{\epsilon \left|M^{x,m,t}_\tau\right|^{2}\left|g^{m,t}\left(\tau, X^{t}_\tau(x) \right)\right|^{2}}{\left|1+\epsilon M^{x,m,t}_\tau \right|^{3}}d\tau \right] = \e\left[\int_{t}^{s}\lim_{\epsilon\to 0}\dfrac{\epsilon \left|M^{x,m,t}_\tau\right|^{2}\left|g^{m,t}\left(\tau, X^{t}_\tau(x) \right)\right|^{2}}{\left|1+\epsilon M^{x,m,t}_\tau \right|^{3}}d\tau \right] = 0.
\end{align*}
Therefore, from \eqref{eq:5-63'}, we have
\begin{align}\label{eq:5-63'''}
    \lim_{\epsilon\to 0}\e\left[\dfrac{M^{x,m,t}_s}{1+\epsilon M^{x,m,t}_s}\right]=1.
\end{align}
From the monotone convergence theorem, we know that
\begin{align*}
    \lim_{\epsilon\to 0}\e\left[\dfrac{M^{x,m,t}_s}{1+\epsilon M^{x,m,t}_s}\right]=\e\left[M^{x,m,t}_s\right];
\end{align*}
together with \eqref{eq:5-63'''}, we obtain \eqref{eq:5-64}.

We now conclude the whole proof by showing the claim \eqref{eq:5-67}; to this end, for any $\epsilon>0$, we can compute by It\^o's formula that
\begin{align*}
    &d\left[\dfrac{M^{x,m,t}_s\left|X^{t}_s(x)\right|^{2}}{1+\epsilon M^{x,m,t}_s\left|X^{t}_s(x)\right|^{2}}\right]\\
    =\ & \Bigg[ \dfrac{M^{x,m,t}_s\left[\text{tr}(a)+2\sigma g^{m,t}\left(s,X^{t}_s(x)\right)\cdot X^{t}_s(x)\right]}{\left(1+\epsilon M^{x,m,t}_s\left|X^{t}_s(x)\right|^{2}\right)^{2}} \\
    &\quad - \dfrac{\epsilon \left|M^{x,m,t}_s\right|^{2} \left|\left|X^{t}_s(x)\right|^{2}g^{m,t}\left(s,X^{t}_s(x)\right)+2\sigma^\top X^{t}_s(x)\right|^2}{\left(1+\epsilon M^{x,m,t}_s\left|X^{t}_s(x)\right|^{2}\right)^{3}} \Bigg] ds\\
    &+ \dfrac{M^{x,m,t}_s\left[\left|X^{t}_s(x)\right|^{2}g^{m,t}\left(s,X^{t}_s(x)\right)+2\sigma^\top X^{t}_s(x)\right]}{(1+\epsilon M^{x,m,t}_s\left|X^{t}_s(x)\right|^{2})^{2}}dw_s.
\end{align*} 
Noting that the second term of the right hand side of the last equality is negative and the third term is a martigale, from \eqref{eq:5-63} we know that 
\begin{align*}
    \dfrac{d}{ds}\;\e\left[\dfrac{M^{x,m,t}_s\left|X^{t}_s(x)\right|^{2}}{1+\epsilon M^{x,m,t}_s\left|X^{t}_s(x)\right|^{2}}\right]\le\ &  \e \left[\dfrac{M^{x,m,t}_s \left[\text{tr }a+2\sigma g^{m,t}\left(s,X^{t}_s(x)\right)\cdot X^{t}_s(x)\right]}{\left(1+\epsilon M^{x,m,t}_s\left|X^{t}_s(x)\right|^{2}\right)^{2}} \right]\\
    \le\ & C\e \left[\dfrac{M^{x,m,t}_s \left(1+\left|X^{t}_s(x)\right|^2 \right)}{1+\epsilon M^{x,m,t}_s\left|X^{t}_s(x)\right|^{2}} \right]\\
    \le\ & C\e \left[M^{x,m,t}_s + \dfrac{M^{x,m,t}_s \left|X^{t}_s(x)\right|^2 }{1+\epsilon M^{x,m,t}_s\left|X^{t}_s(x)\right|^{2}}\right]\\
    \le\ & C\left(1+ \e\left[\dfrac{M^{x,m,t}_s\left|X^{t}_s(x)\right|^{2}}{1+\epsilon M^{x,m,t}_s\left|X^{t}_s(x)\right|^{2}}\right] \right),
\end{align*}
where in the last inequality we used the property $\e\left[M^{x,m,t}_s\right]\le 1$. From this inequality and Gr\"owall's lemma,, we get
\begin{align*}
    \e\left[\dfrac{M^{x,m,t}_s\left|X^{t}_s(x)\right|^{2}}{1+\epsilon M^{x,m,t}_s\left|X^{t}_s(x)\right|^{2}}\right] \le C(T),
\end{align*}
and then, by Fatou's Lemma, we obtain \eqref{eq:5-67}. 
\end{proof}

Thanks to Lemma~\ref{lem5-1}, we can define a probability on $(\Omega,\mathcal{F})$, denoted by $\mathbb{P}^{x,m,t}$, on which $\mathcal{W}_{t}^{s}$ has a Girsanov (Radon-Nikodymn) derivative with respect to $\mathbb{P}$ given by the formula: 
\begin{equation*}\label{eq:5-68}
    \dfrac{d\;\mathbb{P}^{x,m,t}}{d\;\mathbb{P}}\bigg|_{\mathcal{W}_{t}^{s}}=M^{x,m,t}_s.
\end{equation*}
We next define the process 
\begin{equation}\label{eq:5-69}
    w^{x,m,t}_s=w_s-w_t-\int_{t}^{s}\sigma^{-1}\; D_p H\left(s,X^{x,t}_\tau,DV^{m,t}\left(s,X^{x,t}_\tau\right)\right)d\tau,\quad s\in[t,t].
\end{equation}
By Girsanov's theorem, on the probability space $\left(\Omega,\mathcal{F},\mathbb{P}^{x,m,t}\right)$, the process $w^{x,m,t}_\cdot$ becomes a standard Wiener process under $\mathbb{P}^{x,m,t}$, and $X^{t}_s(x)$ turns out to be the corresponding solution of the following SDE\footnote{The process $X^{x,m,t}_\cdot(x)$ in \eqref{eq:5-69} under $\mathbb{P}^{x,m,t}$ is sample pathwisely equivalent to $X^t_\cdot(x)$ in \eqref{def:SDE} under $\mathbb{P}$.}
\begin{equation}\label{eq:5-70}
    X^{x,m,t}_s(x)=x+\int_t^s D_p H\left(\tau,X^{x,m,t}_\tau(x),DV^{m,t}\left(\tau,X^{x,m,t}_\tau(x)\right)\right)d\tau+\sigma \left(w^{x,m,t}_s-w^{x,m,t}_t\right),\quad s\in[t,T].
\end{equation}
We now define the probability measure $m^{m,t}_s$ on $\brn$ by the following formula: for any bounded and continuous function $\varphi$,
\begin{equation}\label{eq:5-71}
    \int_{\brn}\varphi(\xi)dm^{m,t}_s (\xi)=\int_{\brn}\e^{x,m,t}\left[\varphi\left(X^{x,m,t}_s(x)\right)\right]dm(x)=\int_{\brn}\e \left[M^{x,m,t}_s\varphi\left(X^{t}_s(x)\right)\right]dm(x).
\end{equation}
Considering $\Psi$ the solution of the equation \eqref{eq:5-152}, we have $\int_{\brn}\varphi(\xi)dm^{m,t}_s(\xi)=\int_{\brn}\Psi^{m,t}(t,x)dm(x)$. Thus, $m^{m,t}_s$ coincides with the probability measure defined by \eqref{eq:5-151}. 

\subsection{Recovering Master Equation and Bellman Equations}

In this part, we use the solution $V^{m,t}$ of the HJB equation \eqref{eq:5-150} to solve the master equation and the Bellman equation. This has been done formally in Subsection~\ref{subsec:RECOVERING-ALL-EQUATIONS}, and we give the formal statements together with rigorous proofs as follows. 

We define the functional 
\begin{align}\label{def:U}
    U(s,x,m):=V^{m,s}(s,x),\quad (s,x,m)\in[t,T]\times\brn\times\pr_2(\brn).
\end{align}
In the next few subsections, we shall make use of the following condition on $U$.
\begin{condition}\label{Condition:U}
    The functional $U$ is linear functional differentiable, and its first and second derivatives $D_{\xi}\dfrac{dU}{d\nu}(s,x,m)(\xi)$ and $D_{\xi}^{2}\dfrac{dU}{d\nu}(s,x,m)(\xi)$ exist.
\end{condition}
In Subsection~\ref{subsec:add_regularity} (see Theorem~\ref{thm:D_zV}), we shall rigorously show that Condition~\ref{Condition:U} holds under some mild assumptions. We first give the following result.

\begin{proposition}\label{prop:6-1}
Under Condition~\ref{Condition:U}, the function $U$ defined in \eqref{def:U} is a solution of the following master equation 
\begin{equation}\label{eq:6-1}
    \left\{
        \begin{aligned}
            &-\dfrac{\partial U}{\partial s}(s,x,m)+\mathcal{A}U(s,x,m)+\int_{\brn}\mathcal{A}_\xi \frac{dU}{d\nu}(s,x,m)(\xi)dm(\xi)\\
            &=H\left(s,x,DU(s,x,m)\right)+\int_{\brn}D_p H \left(s,\xi,DU(s,\xi,m)\right)\cdot D_\xi\frac{dU}{d\nu}(s,x,m)(\xi) dm(\xi)+
            \frac{dF}{d\nu}(m)(x),\\
            &U(T,x,m)=h(x)+\frac{dF_T}{d\nu}(m)(x).
        \end{aligned}
    \right.
\end{equation} 
\end{proposition}

\begin{proof}
The terminal condition is clear. We first claim that 
\begin{align}\label{flow:0}
    V^{m^{m,t}_\tau,\tau}(s,x) = V^{m,t}(s,x), \quad t\le \tau\le s\le T.
\end{align}
Indeed, from the flow property, we know that 
\begin{align}\label{flow:1}
    m^{m,t}_s=m^{m^{m,t}_\tau,\tau}_s,\quad t\le \tau\le s\le T.
\end{align}
From the HJB equation \eqref{eq:5-150}, we know that the $V^{m^{m,t}_\tau,\tau}(s,x)$ (for national convenience we denote it by $v(s,x)$) satisfies the following equation:
\begin{equation*}
    \left\{
    \begin{aligned}
        &-\frac{\partial v}{\partial s}(s,x)+\mathcal{A}v(s,x)=H\left(s,x,Dv(s,x)\right)+\frac{dF}{d\nu}\left(m^{m^{m,t}_\tau,\tau}_s\right)(x),\quad s\in[t,T),\\
        &v(T,x)=h(x)+\dfrac{dF_T}{d\nu}\left(m^{m^{m,t}_\tau,\tau}_T\right)(x),\quad x\in\brn.   
    \end{aligned}
    \right.
\end{equation*}
From \eqref{flow:1}, we see that the last equation is equivalent to the HJB equation \eqref{eq:5-150} for $V^{m,t}(s,x)$, from which we obtain \eqref{flow:0}. Then, we see from \eqref{def:U} and \eqref{flow:0} that 
\begin{align*}
    U\left(s,x,m^{m,t}_s\right) = V^{m^{m,t}_s,s}(s,x) = V^{m,t}(s,x).
\end{align*}
So we can use the characterization of $m^{m,t}_s$ given by \eqref{eq:5-70} (or the first order condition) and \eqref{eq:5-71} to write, by using It\^o's formula in \eqref{eq:2-560} for $U\left(s,x,m^{m,t}_s\right)$ (with the process $X^{m,t}_s(x)$ in \eqref{eq:2-50} replaced by $X^{x,m,t}_s(x)$ being defined in \eqref{eq:5-70}), we have
\begin{align*}
    &\dfrac{\partial V^{m,t}}{\partial s}(s,x)= \dfrac{d }{d s}U\left(s,x,m^{m,t}_s\right)\\
    %=\ &  \dfrac{\partial U}{\partial s}\left(s,x,m^{m,t}_s\right)+ \int_{\brn}\bigg[ D_{\xi}\dfrac{dU}{d\nu}\left(s,x,m^{m,t}_s\right)\left(\xi\right)\cdot D_p H\left(s,\xi,DU\left(s,\xi,m^{m,t}_s\right)\right)\\
    %&\qquad\qquad\qquad\qquad\qquad -\mathcal{A}_{\xi}\dfrac{dU}{d\nu}\left(s,x,m^{m,t}_s\right)\left(\xi\right)\bigg]dm^{m,t}_s(\xi) \\
    =\ &  \int_{\brn}\e^{z,m,t}\bigg[ D_{\xi}\dfrac{dU}{d\nu}\left(s,x,m^{m,t}_s\right)\left(X^{x,m,t}_s(z)\right)\cdot D_p H\left(s,X^{x,m,t}_s(z),DU\left(s,X^{x,m,t}_s(z),m^{m,t}_s\right)\right)\\
    &\quad\qquad\qquad -\mathcal{A}_{\xi}\dfrac{dU}{d\nu}\left(s,x,m^{m,t}_s\right)\left(X^{x,m,t}_s(z)\right)\bigg]dm(z)+\dfrac{\partial U}{\partial s}\left(s,x,m^{m,t}_s\right),
\end{align*}
and from the HJB equation \eqref{eq:5-150} for $V^{m,t}$, we know that
\begin{align*}%\label{eq:6-3}
    \dfrac{\partial V^{m,t}}{\partial s}(s,x)=\ & \mathcal{A}V^{m,t}(s,x)-H\left(s,x,DV^{m,t}(s,x)\right)-\frac{dF}{d\nu}\left(m^{m,t}_s\right)(x)\\
    =\ & \mathcal{A}U\left(s,x,m^{m,t}_s\right)-H\left(s,x,DU\left(s,x,m^{m,t}_s\right)\right)-\dfrac{dF}{d\nu}\left(m^{m,t}_s\right).
\end{align*}
Writing this relation at $t=s$, we get immediately the master equation \eqref{eq:6-1}.
\end{proof}

We next define the following class of the anti-derivatives 
\begin{equation}\label{eq:6-4}
\begin{aligned}
    \mathcal{S}:=\bigg\{ \Phi:[t,T]\times\pr_2(\brn)\to\br \ \bigg| \ & \dfrac{d\Phi}{d\nu}(s,m)(x)=U(s,x,m), \ \forall (s,m,x)\in [t,T]\times\pr_2(\brn)\times\brn, \\
    &\Phi(T,m)=\int_{\brn}h(\xi)dm(\xi)+F_{T}(m),\ \forall m \in  \pr_2(\brn) \bigg\}.
\end{aligned}
\end{equation}
The set $\mathcal{S}$ is non-empty, indeed, we see that the following function
\begin{align*}
    \Phi(s,m):=h(0)+F_T(\delta_0)+\int_0^1\int_\brn  U\left(s,x,\delta_0+h (m-\delta_0)\right)d(m-\delta_0)(x) dh
\end{align*}
belongs to $\mathcal{S}$. Here, $\delta_0$ is the Dirac distribution with the point mass at $0$. We now recover the Bellman equation as follows. 

%In light of \cite[Chapter 5]{book_mfg}, the function $\Phi$ in \eqref{eq:6-4} is well-defined up to a function independent of $m$, and it is then determined by the terminal condition at $T$, under the regularity for $U$ specified in Condition~\ref{Condition:U}, and also the growth condition for $DU$ being warranted by the growth condition for $DV^{m,t}$ as shown before and summarized in Table~\ref{table:1}. 

\begin{proposition}\label{prop:6-2} 
Under Condition~\ref{Condition:U}, there exists a function $\Phi\in\mathcal{S}$, such that it is the solution of the following Bellman equation: 
\begin{equation}\label{eq:6-5}
    \left\{
        \begin{aligned}
            &-\dfrac{\partial \Phi}{\partial s}(s,m)+\int_{\brn}\mathcal{A}\frac{d\Phi}{d\nu}(s,m)(\xi)dm(\xi)=\int_{\brn}H\left(s,\xi,D\dfrac{d\Phi}{d\nu}(s,m)(\xi)\right)dm(\xi)+F(m),\\
            &\Phi(T,m)=\int_{\brn}h(\xi)dm(\xi)+F_{T}(m).
        \end{aligned}
    \right.
\end{equation} 
\end{proposition}

\begin{proof}
Pick one $\Phi'\in\mathcal{S}$, we plug the relation $U(s,x,m)=\dfrac{d\Phi'}{d\nu}(s,m)(x)$ into the master equation \eqref{eq:6-1} to obtain 
\begin{align*}
    0=\ &-\dfrac{\partial}{\partial s}\dfrac{d\Phi'}{d\nu}(s,m)(x)+\mathcal{A}\frac{d\Phi'}{d\nu}(s,m)(x)+\int_{\brn}\mathcal{A}_{\xi}\dfrac{d^{2}\Phi'}{d\nu^{2}}(s,m)(x,\xi)dm(\xi)- \dfrac{dF}{d\nu}(m)(x)\\
    & - H\left(s,x,D\dfrac{d\Phi'}{d\nu}(s,m)(x)\right) - \int_{\brn}D_{\xi}\dfrac{d^{2}\Phi'}{d\nu^{2}}(s,m)(x,\xi)\cdot D_p H\left(s,\xi,D\dfrac{d\Phi'}{d\nu}(s,m)(\xi)\right)dm(\xi),
\end{align*}
and we see immediately that the right hand side of the last equation is the linear functional derivative of the following
\begin{equation*}
    -\dfrac{\partial \Phi'}{\partial s}(s,m)+\int_{\brn}\mathcal{A}\dfrac{d\Phi'}{d\nu}(s,m)(\xi)dm(\xi)-\int_{\brn}H\left(s,\xi,D\dfrac{d\Phi'}{d\nu}(s,m)(\xi)\right)dm(\xi)-F(m).
\end{equation*}
Therefore, we know that there exist some function $C(s)$ depending solely on $s$ but not $m$ such that 
\begin{equation}\label{eq:4-13}
    -\dfrac{\partial \Phi'}{\partial s}(s,m)+\int_{\brn}\mathcal{A}\dfrac{d\Phi'}{d\nu}(s,m)(\xi)dm(\xi)-\int_{\brn}H\left(s,\xi,D\dfrac{d\Phi'}{d\nu}(s,m)(\xi)\right)dm(\xi)-F(m)=C(s).
\end{equation}
We now define the function $\Phi(s,m):=\Phi'(s,m)-\int_s^T C(\tau)d\tau$. It is easy to see that 
\begin{align*}
    \Phi(T,m)=\Phi'(T,m),\quad \frac{d\Phi}{d\nu}(s,m)(x)=\frac{d\Phi'}{d\mu}(s,m)(x),
\end{align*}
therefore, we know that $\Phi\in\mathcal{S}$. Moreover, from \eqref{eq:4-13}, we see that $\Phi(s,m)$ satisfies
\begin{equation*}
    -\dfrac{\partial \Phi}{\partial s}(s,m)+\int_{\brn}\mathcal{A}\dfrac{d\Phi}{d\nu}(s,m)(\xi)dm(\xi)-\int_{\brn}H\left(s,\xi,D\dfrac{d\Phi}{d\nu}(s,m)(\xi)\right)dm(\xi)-F(m)=0,
\end{equation*}
which is the Bellman equation. 
\end{proof}

We shall come back to the the sufficiency leading to Condition~\ref{Condition:U} in Subsection~\ref{subsec:add_regularity} below. 

%\subsection{A Formula for $\Phi(t,m)$}

\subsection{Revisiting Mean Field Type Control Problem}

We begin by giving a relation between the function $\Phi$ defined in Proposition~\ref{prop:6-2} which satisfies the Bellman equation, and the mean field type control problem under Condition~\ref{Condition:U}. From Propositions~\ref{prop:6-1} and \ref{prop:6-2}, we know that the Bellman equation \eqref{eq:6-5} also reads 
\begin{equation*}\label{eq:6-7}
    \left\{
        \begin{aligned}
            &-\dfrac{\partial \Phi}{\partial s}(s,m)+\int_{\brn}\mathcal{A}U(s,x,m)dm(x)=\int_{\brn}H\left(s,x,DU(s,x,m)\right)dm(x)+F(m),\\
            &\Phi(T,m)=\int_{\brn}h(x)dm(x)+F_{T}(m).
        \end{aligned}
    \right.
\end{equation*} 
Substituting $m$ by $m^{m,t}_s$ yields that
\begin{equation}\label{eq:6-8}
    -\dfrac{\partial \Phi}{\partial s}\left(s,m^{m,t}_s\right)+\int_{\brn}\mathcal{A}U\left(s,x,m^{m,t}_s\right)dm^{m,t}_s(x)=\int_{\brn}H\left(s,x,DU\left(s,x,m^{m,t}_s\right)\right)dm^{m,t}_s(x)+F\left(m^{m,t}_s\right).
\end{equation}
By using It\^o's formula and the very definition of $\Phi$ in terms of $U$, we know that
\begin{align}
    \dfrac{d}{ds}\left[\Phi\left(s,m^{m,t}_s\right)\right]=\ & \dfrac{\partial\Phi}{\partial s}\left(s,m^{m,t}_s\right)-\int_{\brn}\mathcal{A}\dfrac{d\Phi}{d\nu}\left(s,m^{m,t}_s\right)(x)dm^{m,t}_s(x) \notag\\
    &+\int_{\brn}D\dfrac{d\Phi}{d\nu}\left(s,m^{m,t}_s\right)(x)\cdot D_p H\left(s,x,DU\left(s,x,m^{m,t}_s\right)\right)dm^{m,t}_s(x) \notag\\
    =\ & \dfrac{\partial\Phi}{\partial s}\left(s,m^{m,t}_s\right)-\int_{\brn}\mathcal{A}U\left(s,x,m^{m,t}_s\right)dm^{m,t}_s(x)\notag\\
    &+\int_{\brn}DU\left(s,x,m^{m,t}_s\right)\cdot D_pH\left(s,x,DU\left(s,x,m^{m,t}_s\right)\right)dm^{m,t}_s(x). \label{eq:6-9}
\end{align}
Combining \eqref{eq:6-8} and \eqref{eq:6-9}, by using the definition of $\widehat{v}$ in \eqref{def:hat_v} and the first relation in \eqref{D_pH}, we obtain 
\begin{align*}
    \dfrac{d}{ds}\left[\Phi\left(s,m^{m,t}_s\right)\right]=\ & \int_{\brn}DU\left(s,x,m^{m,t}_s\right)\cdot D_p H\left(s,x,DU\left(s,x,m^{m,t}_s\right)\right)dm^{m,t}_s(x)\\
    &-\int_{\brn}H\left(s,x,DU\left(s,x,m^{m,t}_s\right)\right)dm^{m,t}_s(x)-F\left(m^{m,t}_s\right)\\
    =\ & -\int_{\brn}l\left(s,x,\widehat{v}\left(s,x,DU\left(s,x,m^{m,t}_s\right)\right)\right)dm^{m,t}_s(x)-F\left(m^{m,t}_s\right).
\end{align*}
By integrating the last equality between $t$ and $T$, we obtain the following formula for $\Phi$:
\begin{align}
    \Phi(t,m)=\ & \int_{t}^{T}\left[\int_{\brn}l\left(s,x,\widehat{v}\left(s,x,DU\left(s,x,m^{m,t}_s\right)\right)\right)dm^{m,t}_s(x)+F\left(m^{m,t}_s\right)\right]ds \notag \\
    &+\int_{\brn}h(x)dm^{m,t}_T(x)+F_{T}\left(m^{m,t}_T\right). \label{eq:6-11}
\end{align}

We consider a family of stochastic processes, denoted by  $v^{m,t}_s(x),\ (s,x)\in[t,T]\times\brn$, with values in $\brn$ and adapted to the filtration $\mathcal{W}_{t}^{s}$, and they satisfy 
\begin{equation}\label{eq:6-12}
    \left|v^{m,t}_s(x)\right|\le C\left(1+\left|X^{t}_s(x)\right|\right),\quad \text{a.s.},
\end{equation}
for some $C>0$, where $X^{t}_s(x)$ is defined in \eqref{def:SDE}. The constraint \eqref{eq:6-12} defines a closed subspace of the Hilbert space $L_{\mathcal{W}_{t}}^{2}\left(t,T;L_{m}^{2}(\Omega,\mathcal{F},P;\brn)\right)$, we denote this by $\mathcal{V}^{m,t}$. We then define the process 
\begin{equation}\label{eq:6-13}
    M^{v^{m,t}_\cdot(x)}_s=\exp\left(\int_{t}^{s}\sigma^{-1}\; v^{m,t}_\tau(x)\cdot dw(\tau)-\dfrac{1}{2}\int_{t}^{s}\left|\sigma^{-1}\; v^{m,t}_\tau(x)\right|^{2}d\tau\right),\quad s\in[t,T].
\end{equation}
%then, similar as in \eqref{eq:5-66}, we have $\e\left[M^{v^{m,t}_\cdot(x)}_s\right]\le 1$. 
Following a similar approach as in the proof leading to Lemma~\ref{lem5-1}, we know that for a control $v^{m,t}_\cdot(\cdot)\in \mathcal{V}^{m,t}$ , this super-martingale $M^{v^{m,t}_\cdot(x)}_s$ is indeed a martingale and it satisfies $\e \left[M^{v^{m,t}_\cdot(x)}_s\right]=1$.
%\begin{equation*}\label{eq:6-15}
    %\e \left[M^{v^{m,t}_\cdot(x)}_s\right]=1;
%\end{equation*}
We then define a probability on $(\Omega,\mathcal{F})$, denoted by $\mathbb{P}^{v^{m,t}_\cdot(x)}$, by the Girsanov (Radon-Nikodym) derivative 
\begin{equation*}\label{eq:6-16}
    \dfrac{d\mathbb{P}^{v^{m,t}_\cdot(x)}}{d\mathbb{P}}\bigg|_{\mathcal{W}_{t}^{s}}=M^{v^{m,t}_\cdot(x)}_s .
\end{equation*}
We next define the process 
\begin{equation*}\label{eq:6-17}
    w^{v^{m,t}_\cdot(x)}_s=w_s-w_t-\int_{t}^{s}\sigma^{-1}\;v^{m,t}_\tau(x)d\tau,\quad \forall (s,x)\in[t,T]\times\brn.
\end{equation*}
By Girsanov's theorem, in the probability space $\left(\Omega,\mathcal{F},\mathbb{P}^{v^{m,t}_\cdot(x)}\right)$, the process $w^{v^{m,t}_\cdot(x)}_\cdot$ is a Wiener process, and the process $X^{t}_s(x)$ has the following It\^o's integral representation:
\begin{equation*}\label{eq:6-18}
    X^{t}_s(x)=x+\int_t^s v^{m,t}_\tau(x)d\tau+w^{v^{m,t}_\cdot(x)}_s-w^{v^{m,t}_\cdot(x)}_t,\quad \forall s\in[t,T].
\end{equation*}
We next define on the product space $\left(\Omega\times \brn,\mathcal{F}\times\mathcal{B}(\brn)\right)$ the product probability measure $\mathbb{P}^{v^{m,t}_s(\cdot)}\otimes m$ defined by the expectation of a random variable $Z(x,\omega)$:
\begin{equation*}\label{eq:6-19}
    \e^{P^{v^{m,t}_s(\cdot)}\otimes m}[Z]=\int_{\brn}\e \left[M^{v^{m,t}_\cdot(x)}_s Z(x)\right]dm(x).
\end{equation*}
We then denote by $m^{v,m,t}_s$ the pushed forward probability $X^{t}_s\#\left(\mathbb{P}^{v^{m,t}_s(\cdot)}\otimes m\right)$, therefore, we can write for any test function $\varphi$ which is continuous and satisfies $|\varphi(x)|\le C\left(1+|x|^{2}\right)$,
\begin{equation}\label{eq:6-20}
    \int_{\brn}\varphi(\xi)dm^{v,m,t}_s(\xi)=\int_{\brn}\e \left[M^{v^{m,t}_\cdot(x)}_s\varphi\left(X^{t}_s(x)\right)\right]dm(x).
\end{equation}
From the very definition of the payoff functional and \eqref{eq:6-20}, we know that the cost associated with the control $v^{m,t}_\cdot(\cdot)$ can be written as
\begin{equation}\label{eq:6-21}
\begin{aligned}
    J_{mt}\left(v^{m,t}\right)=\ & \int_{t}^{T}\left(\int_{\brn}\e \left[M^{v^{m,t}_\cdot(x)}_s \;l\left(s,X^{t}_s(x),v^{m,t}_s(x)\right)\right]dm(x)+F\left(m^{v,m,t}_s\right)\right)ds\\
    &+\int_{\brn}\e\left[M^{v^{m,t}_\cdot(x)}_T \;h\left(X^{t}_T\right)\right]dm(x)+F_{T}\left(m^{v,m,t}_T\right).
\end{aligned}
\end{equation}
%Under Assumption~\ref{assumption:l}, 
To simplify the notation, we denote by $u^{m,t}_s(x):=\widehat{v}\left(s,X^{t}_s(x),DU\left(s,X^{t}_s(x),m^{m,t}_s\right)\right)$, clearly, $m^{u,m,t}_s=m^{m,t}_s$. Noting the facts that $\widehat{v}(t,x,p)=D_p H(t,x,p)$ and $DU(t,x,m)=DV^{m,t}(t,x)$ (see \eqref{def:U}), from Assumption~\ref{assumption:H} and the fact that $V^{m,t}\in B_T$, we know that $\widehat{v}(t,x,DU(t,x,m))$ has a linear growth. Then, from \eqref{eq:5-71} and the definition of the payoff functional, we know that $u^{m,t}_\cdot(\cdot)\in\mathcal{V}^{m,t}$ and 
%We have proven in Lemma~\ref{lem5-1} that $u^{m,t}_\cdot(\cdot)\in\mathcal{V}^{m,t}$. This is due to the fact that the feedback $\widehat{v}(x,DU(s,x,m),s)$ has linear growth and thanks to estimate (\ref{eq:5-67}). 
\begin{align*}
    J_{mt}\left(u^{m,t}\right)=\ &  \int_{t}^{T}\bigg(\int_{\brn}\e \left[M^{u^{m,t}_\cdot(x)}_s\;l\left(s,X^{t}_s(x),\widehat{v}\left(s,X^{t}_s(x),DU\left(s,X^{t}_s(x),m^{m,t}_s\right)\right)\right)\right]dm(x)\\
    &\qquad +F\left(m^{m,t}_s\right)\bigg)ds + +\int_{\brn}\e\left[M^{u^{m,t}_\cdot(x)}(T)\;h\left(X^{t}_T(x)\right)\right]dm(x)+F_{T}\left(m^{m,t}_T\right)\\
    =\ & \int_{t}^{T}\left[\int_{\brn}l\left(s,x,\widehat{v}\left(s,x,DU\left(s,x,m^{m,t}_s\right)\right)\right)dm^{m,t}_s(x)+F\left(m^{m,t}_s\right)\right]ds \\
    &+ \int_{\brn}h(x)dm^{m,t}_T(x)+F_{T}\left(m^{m,t}_T\right).
\end{align*}
Then, by comparing with the formula \eqref{eq:6-11}, we obtain that
\begin{equation}\label{eq:6-25}
    J_{mt}\left(u^{m,t}\right)=\Phi(t,m),\quad \forall (t,x)\in[0,T]\times\pr_2(\brn).
\end{equation}
We can now give the following result, which shows that $u^{m,t}$ is an optimal control for the mean field type control problem, and then $\Phi$ is the value function due to the uniqueness result of the Bellman equation \eqref{eq:6-5}; see \cite{AB10} for instance.  

\begin{theorem}\label{theo6-1}
Under Condition~\ref{Condition:U}, the control $u^{m,t}$ minimizes the cost functional $J_{mt}(\cdot)$ defined by \eqref{eq:6-21}. 
\end{theorem}

\begin{proof} 
In view of \eqref{eq:6-25}, we only have to prove that 
\begin{equation}\label{eq:6-26}
   \Phi(t,m)\le J_{mt}\left(v^{m,t}\right),\qquad \forall\;v^{m,t}_\cdot(\cdot)\in\mathcal{V}^{m,t}.
\end{equation}
We first write the Bellman equation \eqref{eq:6-5} with the argument $m=m^{v,m,t}_s$; and substitute $m=m^{v,m,t}_T$ in the terminal condition to obtain 
\begin{align*}
    &-\dfrac{\partial \Phi}{\partial s}\left(s,m^{v,m,t}_s\right)+\int_{\brn}\mathcal{A}\frac{d\Phi}{d\nu}\left(s,m^{v,m,t}_s\right)(\xi)dm^{v,m,t}_s(\xi)\\
    &\;\;\ \quad\qquad\qquad =\int_{\brn}H\left(s,\xi,D\dfrac{d\Phi}{d\nu}\left(s,m^{v,m,t}_s\right)(\xi)\right)dm^{v,m,t}_s(\xi)+F\left(m^{v,m,t}_s\right),\quad s\in[t,T);\\
    &\Phi\left(T,m^{v,m,t}_T\right)= \int_{\brn}h(\xi)dm^{v,m,t}_T(\xi)+F_{T}\left(m^{v,m,t}_T\right).
\end{align*}
They can also be rewritten as (from \eqref{eq:6-20}):
\begin{equation}\label{eq:6-28}
    \begin{aligned}
        &-\dfrac{\partial \Phi}{\partial s}\left(s,m^{v,m,t}_s\right)+\int_{\brn}\e \left[M^{v^{m,t}_\cdot(x)}_s\mathcal{A}\dfrac{d\Phi}{d\nu}\left(s,m^{v,m,t}_s\right)\left(X^{t}_s(x)\right)\right]dm(x)\\
        &\qquad =\int_{\brn}\e\left[M^{v^{m,t}_\cdot(x)}_sH\left(s,X^{t}_s(x),D\dfrac{d\Phi}{d\nu}\left(s,m^{v,m,t}_s\right)(X^{t}_s(x))\right)\right]dm(x)+F\left(m^{v,m,t}_s\right),\quad s\in[t,T);\\
        &\Phi\left(T,m^{v,m,t}_T\right)= \int_{\brn} \e  \left[M^{v^{m,t}_\cdot(x)}_T h\left(X^{t}_T(x)\right)\right]dm(x)+F_{T}\left(m^{v,m,t}_T\right).
    \end{aligned}
\end{equation}
On the other hand, from It\^o's formula in \eqref{eq:2-560}, we have
\begin{align}
    \dfrac{d}{ds}\left[\Phi\left(s,m^{v,m,t}_s\right)\right]=\ & \dfrac{\partial \Phi}{\partial s}\left(s,m^{v,m,t}_s\right)+\int_{\brn} \e \left[M^{v^{m,t}_\cdot(x)}_s\:D\dfrac{d\Phi}{d\nu}\left(s,m^{v,m,t}_s\right)\left(X^{t}_s(x)\right)\cdot v^{m,t}_s(x)\right]dm(x) \notag \\
    &-\int_{\brn}\e\left[M^{v^{m,t}_\cdot(x)}_s\mathcal{A}\dfrac{d\Phi}{d\nu}\left(s,m^{v,m,t}_s\right)\left(X^{t}_s(x)\right)\right]dm(x).  \label{eq:6-29}
\end{align}
Same as before, combining \eqref{eq:6-28} and \eqref{eq:6-29}, from the definition of the Hamiltonian $H$,  we can write 
\begin{align*}
    \dfrac{d}{ds}\left[\Phi\left(s,m^{v,m,t}_s\right)\right] =\ & \int_{\brn} \e \left[M^{v^{m,t}_\cdot(x)}_s\:D\dfrac{d\Phi}{d\nu}\left(s,m^{v,m,t}_s\right)\left(X^{t}_s(x)\right)\cdot v^{m,t}_s(x) \right] dm(x)\\
    &-\int_{\brn}\e\left[M^{v^{m,t}_\cdot(x)}_s H\left(s,X^{t}_s(x),D\dfrac{d\Phi}{d\nu}\left(s,m^{v,m,t}_s\right)\left(X^{t}_s(x)\right) \right)\right]dm(x)-F\left(m^{v,m,t}_s\right)\\
    \geq\ & -\int_{\brn}\e\left[ M^{v^{m,t}_\cdot(x)}_s l\left(s,X^{t}_s(x),v^{m,t}_s(x)\right)\right]dm(x)-F\left(m^{v,m,t}_s\right),
\end{align*}
where the last inequality is due to the fact that $\hat{v}(\cdot)$ is the minimizer of $H$. Integrating $s$ from $t$ to $T$ and using the condition at $T$ in \eqref{eq:6-28}, we obtain immediately \eqref{eq:6-26}, which completes the proof. 
\end{proof}

\subsection{Higher Order Regularities for Linear Functional Derivatives of $V^{m,t}(s,x)$}\label{subsec:add_regularity}

In the statements of Propositions~\ref{prop:6-1} and \ref{prop:6-2}, we demanded the existence of the derivatives 
\begin{align*}
    \dfrac{dV^{m,t}}{d\nu}(s,x)(z),\quad D_{z}\dfrac{dV^{m,t}}{d\nu}(s,x)(z),\quad D_{z}^{2}\dfrac{dV^{m,t}}{d\nu}(s,x)(z).
\end{align*}
To guarantee the well-definitiveness of these derivatives, we shall first derive the equation (see \eqref{eq:6-51} below) for $\dfrac{dV^{m,t}}{d\nu}(s,x)(z)$, which is new in the literature, and it can be obtained by taking the linear functional derivative of the HJB equation \eqref{eq:5-150} for $V^{m,t}$; and as a solution, its regularity can be inferred from that of the coefficient functions of the corresponding PDE. It can be studied by the methods developed in Section~\ref{sec:main} of this paper. %We shall give the details in a separate paper. We first give the full result. 
To simplify the notation, we denote by 
\begin{align*}
    \overline{V}^{m,t}(s,x,z):=\dfrac{dV^{m,t}}{d\nu}(s,x)(z),\quad (s,x,z)\in[t,T]\times\brn\times\brn. 
\end{align*}
By differentiating the HJB equation \ref{eq:5-150} in $m$, we obtain the following equation for $\overline{V}^{m,t}(s,x)(z)$: 
\begin{equation}\label{eq:6-51}
    \left\{
    \begin{aligned}
        &-\dfrac{\partial \overline{V}^{m,t}}{\partial s}(s,x,z)+\mathcal{A}_{x}\overline{V}^{m,t}(s,x,z)=D_{p}H\left(s,x,DV^{m,t}(s,x)\right)\cdot D_{x}\overline{V}^{m,t}(s,x,z)\\
        &\quad\qquad\qquad\qquad\qquad\qquad\qquad\qquad\qquad +\dfrac{\partial}{\partial m}\dfrac{dF}{d\nu}\left(m^{m,t}_s\right)(x)(z),\quad s\in[t,T),\\
        &\overline{V}^{m,t}(T,x,z)=\dfrac{\partial}{\partial m}\dfrac{dF_{T}}{d\nu}\left(m^{m,t}_T\right)(x)(z),\quad (x,z)\in\brn\times\brn.
    \end{aligned}
    \right.
\end{equation}
As $V^{m,t}$ is already known now, Equation \eqref{eq:6-51} is a non-local and linear equation for $\overline{V}^{m,t}$; it has a similar form as the HJB equation \ref{eq:5-150}, except the two terms $\dfrac{\partial}{\partial m}\dfrac{dF}{d\nu}\left(m^{m,t}_s\right)(x)(z)$ and $\dfrac{\partial}{\partial m}\dfrac{dF_{T}}{d\nu}\left(m^{m,t}_T\right)(x)(z)$, which are functional derivatives of the respective maps $m \mapsto \dfrac{dF}{d\nu}\left(m^{m,t}_s\right)(x)$ and $m \mapsto \dfrac{dF_T}{d\nu}\left(m^{m,t}_T\right)(x)$ in the measure argument $m$. In the HJB equations \eqref{eq:5-150} for $V^{m,t}$, the distribution $m_s^{m,t}$ in the expressions of $\dfrac{dF}{d\nu}\left(m^{m,t}_s\right)(x)$ and $\dfrac{dF_T}{d\nu}\left(m^{m,t}_T\right)(x)$ is described by \eqref{eq:5-151}-\eqref{eq:5-152} which is uniquely characterized by $DV^{m,t}$. In a similar way, we shall next characterize $\dfrac{\partial}{\partial m}\dfrac{dF}{d\nu}\left(m^{m,t}_s\right)(x)(z)$ and $\dfrac{\partial}{\partial m}\dfrac{dF_{T}}{d\nu}\left(m^{m,t}_T\right)(x)(z)$  by $D\overline{V}^{m,t}$. 

For the characterization of $\dfrac{\partial}{\partial m}\dfrac{dF}{d\nu}\left(m^{m,t}_s\right)(x)(z)$ and $\dfrac{\partial}{\partial m}\dfrac{dF_{T}}{d\nu}\left(m^{m,t}_T\right)(x)(z)$ and the well-definitiveness of the function $\overline{V}^{m,t}(s,x)(z)$ (and also its derivatives in $z$, denoted by $D_z\overline{V}^{m,t}(s,x)(z)$ and $D_z^2\overline{V}^{m,t}(s,x)(z)$), we need the following assumptions, which are the regularity-enhanced version of previous assumptions on $F$ and $F_T$ from respective Assumptions~\ref{assumption:F} and \ref{assumption:H}). Indeed, the following Assumption~\ref{assumption:regularity}-(i) is to guarantee the well-definitiveness of $\overline{V}^{m,t}(s,x)(z)$ and $D_z\overline{V}^{m,t}(s,x)(z)$, while Assumption~\ref{assumption:regularity}-(ii) is to guarantee that of $D_z^2 \overline{V}^{m,t}(s,x)(z)$.

\begin{assumption}[Enhanced regularity]\label{assumption:regularity}
\text{ }\\
(i) For $\overline{V}^{m,t}(s,x)(z)$ and $D_z\overline{V}^{m,t}(s,x)(z)$: The derivatives
\begin{align*}
    D_{p}^{2}H(s,x,p),\quad D_xD_{p}H(s,x,p),\quad D_{x}^{2}H(s,x,p)
\end{align*}
exist, and they are continuous in $(s,x,p)$ and uniformly bounded. The twice linear functional derivatives $\dfrac{d^{2}F}{d\nu^{2}}(m)(x,z)$ and $\dfrac{d^{2}F_T}{d\nu^{2}}(m)(x,z)$ exist, and they are continuously differentiable in $z$, and the derivatives satisfy
\begin{align*}
    \left|D_z\frac{d^2F}{d\nu^2}(m)(x,z)\right| \le c (1+|x|),\quad \left|D_z\frac{d^2F_T}{d\nu^2}(m)(x,z)\right| \le c_T (1+|x|). %\quad \forall(m,x,z)\in\pr_2(\brn)\times\brn\times\brn.
\end{align*}
(ii) For $D_z^2 \overline{V}^{m,t}(s,x)(z)$: The derivatives 
\begin{align*}
    D_{p}^{3}H(s,x,p),\quad D_xD_{p}^{2}H(s,x,p),\quad D_x^2 D_{p}H(s,x,p),\quad D_x^3 H(s,x,p)
\end{align*}
exist, and they are continuous in $(s,x,p)$ and uniformly bounded. The derivatives $D_z^2\dfrac{d^{2}F}{d\nu^{2}}(m)(x,z)$ and $D_z^2\dfrac{d^{2}F_T}{d\nu^{2}}(m)(x,z)$ exist, and they are continuous and satisfy 
\begin{align*}
    \left|D_z^2\frac{d^2F}{d\nu^2}(m)(x,z)\right| \le c (1+|x|),\quad \left|D_z^2\frac{d^2F_T}{d\nu^2}(m)(x,z)\right| \le c_T (1+|x|). %\quad \forall(m,x,z)\in\pr_2(\brn)\times\brn\times\brn.
\end{align*}
%\begin{align*}
    %\left|D_xD_z\frac{d^2F}{d\nu^2}(m)(x,z)\right|\le c,\quad \left|D_xD_z\frac{d^2F_T}{d\nu^2}(m)(x,z)\right|\le c_T.
%\end{align*}
\end{assumption}

We next introduce the Green function $G^{m,t}(\tau,z;s,\zeta)$ for $t\le \tau\le s$, which is the solution of the following equation: for $(s,\zeta)\in[t,T]\times\brn$,
\begin{equation}\label{eq:6-31}
\left\{
\begin{aligned}
    &-\dfrac{\partial G^{m,t}}{\partial\tau}(\tau,z;s,\zeta)+\mathcal{A}_{z}G^{m,t}(\tau,z;s,\zeta)=D_{z}G^{m,t}(\tau,z;s,\zeta)\cdot D_p H\left(\tau,z,DV^{m,t}(\tau,z)\right),\quad \tau\in[t,s),\\
    &G^{m,t}(s,z;s,\zeta)=\delta_0(z-\zeta),\quad z\in\brn.
\end{aligned}
\right.
\end{equation}
Again, $\delta_0$ is the Dirac function. We begin by giving the formulation of $\dfrac{\partial}{\partial m}\dfrac{dF}{d\nu}\left(m^{m,t}_s\right)(x)(z)$, and that of $\dfrac{\partial}{\partial m}\dfrac{dF_T}{d\nu}\left(m^{m,t}_T\right)(x)(z)$ can be obtained similarly. Under Assumption~\ref{assumption:regularity}-(i), we define the function
\begin{align}
    \Psi^{m,t,s}(x)(\tau,z):=\int_{\brn}G^{m,t}(\tau,z;s,\zeta)\dfrac{d^{2}F}{d\nu^{2}}\left(m^{m,t}_s\right)(x,\zeta)d\zeta, \quad t\le\tau\le s< T,\quad x,z\in\brn,\label{eq:6-32}
\end{align}
and then the function
\begin{align}
    \widetilde{\Psi}^{m,t,s}(x)(\tau,\xi)(z)=\int_{\tau}^{s}\int_{\brn}G^{m,t}(\tau,\xi;\theta,\zeta)\bigg(&D_{\zeta}\Psi^{m,t,s}(x)(\theta,\zeta) \notag \\
    &\cdot D_{p}^{2}H\left(\theta,\zeta,DV^{m,t}(\theta,\zeta)\right)D_{\zeta} \overline{V}^{m,t}(\theta,\zeta,z)\bigg)d\zeta d\theta. \label{eq:6-32'}
\end{align}
%Note that 
%\[D_{p}^{2}H(s,x,p)=D_{p}\widehat{v}(s,x,p).\]
By simple computation, we know that the function $\Psi^{m,t,s}(x)(\tau,z)$ defined in \eqref{eq:6-32} satisfies the following linear equation: for $(s,x)\in[t,T)\times\brn$,
\begin{equation}\label{PDE:Psi}
\left\{
\begin{aligned}
    &-\dfrac{\partial \Psi^{m,t,s}}{\partial\tau}(x)(\tau,z) +\mathcal{A}_{z} \Psi^{m,t,s}(x)(\tau,z) \\
    &=D_{z}\Psi^{m,t,s}(x)(\tau,z) \cdot D_p H\left(\tau,z,DV^{m,t}(\tau,z)\right),\quad \tau\in[t,s),\\
    &\Psi^{m,t,s}(x)(s,z)=\dfrac{d^{2}F}{d\nu^{2}}\left(m^{m,t}_s\right)(x,z),\quad z\in\brn.
\end{aligned}
\right.
\end{equation}
We have the following result.

\begin{proposition}\label{prop:6-3} 
Under Assumption~\ref{assumption:regularity}-(i), the derivative $\dfrac{\partial}{\partial m}\dfrac{dF}{d\nu}\left(m^{m,t}_s\right)(x)(z)$ is given by the formula: 
\begin{align}
    \dfrac{\partial}{\partial m}\dfrac{dF}{d\nu}\left(m^{m,t}_s\right)(x)(z)=\ & \Psi^{m,t,s}(x)(t,z)+\int_{\brn}\widetilde{\Psi}^{m,t,s}(x)(t,\xi)(z)dm(\xi),\quad (s,x,z)\in[t,T]\times\brn\times\brn. \label{eq:6-33}
\end{align}
\end{proposition}

\begin{proof}
For $m,m'\in\pr_2(\brn)$ and any $\epsilon\in(0,1)$, we denote by 
\begin{equation*}\label{eq:6-34}
    \de^{\epsilon}(s,x):=\dfrac{1}{\epsilon}\left[\dfrac{dF}{d\nu}\left(m^{m+\epsilon(m'-m),t}_s(x)\right)-\dfrac{dF}{d\nu}\left(m^{m,t}_s(x)\right)\right], 
\end{equation*}
and set
\begin{equation*}\label{eq:6-35}
    F^{\epsilon}(x)(s,z):=\int_{0}^{1}\dfrac{d^{2}F}{d\nu^{2}}\left(m^{m,t}_s+\lambda\left(m^{m+\epsilon(m'-m),t}_s-m^{m,t}_s\right)\right)(x,z)d\lambda.
\end{equation*}
Then, by the fundamental theorem of calculus, we can write that 
\begin{equation*}\label{eq:6-36}
    \de^{\epsilon}(s,x)=\dfrac{1}{\epsilon}\left[\int_{\brn}F^{\epsilon}(x)(s,z)d\left(m^{m+\epsilon(m'-m),t}_s(z)\right) -\int_{\brn}F^{\epsilon}(x)(s,z)dm^{m,t}_s(z)\right]. 
\end{equation*}
We introduce 
\begin{equation}\label{eq:6-37}
    \begin{aligned}
        \Psi_{s}^{\epsilon,1}(x)(\tau,z):=\ & \int_{\brn}G^{m+\epsilon(m'-m),t}(\tau,z;s,\zeta)F^{\epsilon}(x)(s,\zeta)d\zeta,\\
        \Psi_{s}^{\epsilon,2}(x)(\tau,z):=\ & \int_{\brn}G^{m,t}(\tau,z;s,\zeta)F^{\epsilon}(x)(s,\zeta)d\zeta,
    \end{aligned}
\end{equation}
which are the solutions of the following respective equations:
\begin{equation*}
\left\{
\begin{aligned}
    &-\dfrac{\partial \Psi^{\epsilon,1}}{\partial\tau}(x)(\tau,z) +\mathcal{A}_{z} \Psi^{\epsilon,1}(x)(\tau,z) \\
    &\qquad =D_{z}\Psi^{\epsilon,1}(x)(\tau,z) \cdot D_p H\left(\tau,z,DV^{m+\epsilon(m'-m),t}(\tau,z)\right),\quad \tau\in[t,s),\\
    &\Psi^{\epsilon,1}(x)(s,z)=F^{\epsilon}(x)(s,z),\quad z\in\brn,
\end{aligned}
\right.
\end{equation*}
and
\begin{equation*}
\left\{
\begin{aligned}
    &-\dfrac{\partial \Psi^{\epsilon,2}}{\partial\tau}(x)(\tau,z) +\mathcal{A}_{z} \Psi^{\epsilon,2}(x)(\tau,z) =D_{z}\Psi^{\epsilon,2}(x)(\tau,z) \cdot D_p H\left(\tau,z,DV^{m,t}(\tau,z)\right),\quad \tau\in[t,s),\\
    &\Psi^{\epsilon,1}(x)(s,z)=F^{\epsilon}(x)(s,z),\quad z\in\brn.
\end{aligned}
\right.
\end{equation*}
Then, from the responding characterization of $m^{m,t}_s$ in \eqref{eq:5-151}-\eqref{eq:5-152}, we see immediately that
\begin{align}
    \de^{\epsilon}(s,x)=\ & \dfrac{1}{\epsilon}\left[\int_{\brn}\Psi_{s}^{\epsilon,1}(x)(t,z)\left[dm(z)+\epsilon(dm'(z)-dm(z))\right]-\int_{\brn}\Psi_{s}^{\epsilon,2}(x)(t,z)dm(z)\right] \notag \\
    =\ & \int_{\brn}\Psi_{s}^{\epsilon,1}(x)(t,z)\left(dm'(z)-dm(z)\right)+\int_{\brn}\dfrac{\Psi_{s}^{\epsilon,1}(x)(t,\xi)-\Psi_{s}^{\epsilon,2}(x)(t,\xi)}{\epsilon}dm(\xi). \label{eq:6-38}
\end{align}
From \eqref{eq:6-32} and \eqref{eq:6-37}, it is clear that $\Psi_{s}^{\epsilon,1}(x)(\tau,z)$ and $\Psi_{s}^{\epsilon,2}(x)(\tau,z)$ converge to $\Psi^{m,t,s}(x)(\tau,z)$ as $\epsilon$ goes to $0$. We next denote by
\begin{equation*}\label{eq:6-40}
    \Psi_{s}^{\epsilon}(x)(\tau,\xi):=\dfrac{\Psi_{s}^{\epsilon,1}(x)(\tau,\xi)-\Psi_{s}^{\epsilon,2}(x)(\tau,\xi)}{\epsilon},
\end{equation*}
and set 
\begin{equation*}\label{eq:6-41}
    \widetilde{G}^{\epsilon}(\tau,\xi;s,\zeta):=\dfrac{G^{m+\epsilon(m'-m),t}(\tau,\xi;s,\zeta)-G^{m,t}(\tau,\xi;s,\zeta))}{\epsilon}.
\end{equation*}
Then, from \eqref{eq:6-37}, we can write that 
\begin{equation}\label{eq:6-42}
    \Psi_{s}^{\epsilon}(x)(\tau,\xi)=\int_{\brn}\widetilde{G}^{\epsilon}(\tau,\xi;s,\zeta)F^{\epsilon}(x)(s,\zeta)d\zeta,
\end{equation}
and \eqref{eq:6-38} can be written as
\begin{align}
    \de^{\epsilon}(s,x)=\ & \int_{\brn}\Psi_{s}^{\epsilon,1}(x)(t,z)\left(dm'(z)-dm(z)\right)+\int_{\brn} \Psi_{s}^{\epsilon}(x)(t,\xi) dm(\xi). \label{eq:6-38'}
\end{align}
From Equation \ref{eq:6-31} for $G^{m,t}$, the fundamental theorem of calculus and the chain rule, we can write the following relations 
\begin{align*}
    &-\dfrac{\partial \widetilde{G}^{\epsilon}}{\partial\tau}(\tau,\xi;s,\zeta)+\mathcal{A}_{\xi}\widetilde{G}^{\epsilon}(\xi,\tau;\zeta,s)\\
    =\ & \dfrac{1}{\epsilon}\bigg[D_{\xi}G^{m+\epsilon(m'-m),t}(\tau,\xi;s,\zeta)\cdot D_p H\left(\tau,\xi,DV^{m+\epsilon(m'-m),t}(\tau,\xi)\right)\\
    &\quad -D_{\xi}G^{m,t}(\tau,\xi;s,\zeta)\cdot D_pH \left(\tau,\xi,DV^{m,t}(\tau,\xi)\right)\bigg]\\
    =\ & D_{\xi}\widetilde{G}^{\epsilon}(\tau,\xi;s,\zeta)\cdot D_pH \left(\tau,\xi,DV^{m+\epsilon(m'-m),t}(\tau,\xi)\right)\\
    &+D_{\xi}G^{m,t}(\xi,\tau;\zeta,s)\cdot \dfrac{D_p H \left(\tau,\xi,DV^{m+\epsilon(m'-m),t}(\tau,\xi)\right)-D_p H \left(\tau,\xi,DV^{m,t}(\tau,\xi)\right)}{\epsilon}\\
    =\ & D_{\xi}\widetilde{G}^{\epsilon}(\tau,\xi;s,\zeta)\cdot D_pH \left(\tau,\xi,DV^{m+\epsilon(m'-m),t}(\tau,\xi)\right)\\
    &+D_{\xi}G^{m,t}(\tau,\xi;s,\zeta)\cdot \int_{0}^{1}D_{p}^{2}H\left(\tau,\xi,\left(DV^{m,t}+\lambda(DV^{m+\epsilon(m'-m),t}-DV^{m,t})\right)(\tau,\xi) \right)d\lambda\\
    &\;\;\qquad\qquad\qquad\qquad \cdot \int_{\brn}\int_{0}^{1}D_\xi\overline{V}^{m+\epsilon\mu(m'-m),t}(\tau,\xi,z)d\mu\,(dm'(z)-dm(z)),\quad \tau\in[t,s),
\end{align*}
and the terminal condition $\widetilde{G}^{\epsilon}(s,\xi;s,\zeta)=0$. Then, from \eqref{eq:6-42} and the second equation of \eqref{eq:6-37}, by using the integration-by-parts, we see that $\Psi_{s}^{\epsilon}(x)(\tau,\xi)$ is the solution of the following equation
\begin{align*}
    &-\dfrac{\partial \Psi_{s}^{\epsilon}}{\partial\tau}(x)(\tau,\xi)+\mathcal{A}_{\xi}\Psi_{s}^{\epsilon}(x)(\tau,\xi)\\
    =\ & D_{\xi}\Psi_{s}^{\epsilon}(x)(\tau,\xi)\cdot D_p H \left(\tau,\xi,DV^{m+\epsilon(m'-m),t}(\tau,\xi) \right)\\
    &+ D_{\xi}\Psi_{s}^{\epsilon,2}(x)(\tau,\xi)\cdot \int_{0}^{1}D_{p}^{2}H\left(\tau,\xi, \left(DV^{m,t}+\lambda(DV^{m+\epsilon(m'-m),t}-DV^{m,t})\right)(\tau,\xi) \right) d\lambda \\
    &\ \ \quad\qquad\qquad\qquad \cdot \int_{\brn}\int_{0}^{1} D_\xi\overline{V}^{m+\epsilon\mu(m'-m),t}(\tau,\xi,z)d\mu\,(dm'(z)-dm(z)),\quad \tau\in[t,s),
\end{align*}
with the terminal condition $\Psi_{s}^{\epsilon}(x)(s,\xi)=0$.  We pass the last equation to the limit as $\epsilon\rightarrow0$ to obtain the following equation for a function $\widetilde{\Psi}_s(x)(\tau,\xi)$: 
\begin{align*}
    &-\dfrac{\partial \widetilde{\Psi}_s}{\partial\tau}(x)(\tau,\xi)+\mathcal{A}_{\xi}\widetilde{\Psi}_s(x)(\tau,\xi) - D_{\xi}\widetilde{\Psi}_s(x)(\tau,\xi)\cdot D_p H \left(\tau,\xi,DV^{m,t}(\tau,\xi) \right)\\
    =\ & D_{\xi}\Psi^{m,t,s}(x)(\tau,\xi)\cdot  D_{p}^{2}H\left(\tau,\xi, DV^{m,t}(\tau,\xi) \right)  \int_{\brn} D_\xi\overline{V}^{m,t}(\tau,\xi,z)(dm'(z)-dm(z)),\quad \tau\in[t,s),
\end{align*}
with the terminal condition $\widetilde{\Psi}_s(x)(s,\xi)=0$, and its solution $\widetilde{\Psi}_s(x)(\tau,\xi)$ is the limit of $\Psi_{s}^{\epsilon}(x)(\tau,\xi)$ as $\epsilon$ goes to $0$. Therefore, from \eqref{eq:6-31}, we know that $\widetilde{\Psi}_s(x)(\tau,\xi)$ can be written as
\begin{equation*}\label{eq:6-48}
\begin{aligned}
    \widetilde{\Psi}_{s}(x)(\tau,\xi)=\int_{\tau}^{s} \int_\brn &G^{m,t}(\tau,\xi;\theta,\zeta)D_{\zeta}\Psi^{m,t,s}(x)(\theta,\zeta) \\
    &\cdot D_{p}^{2}H\left(\theta,\zeta,DV^{m,t}(\theta,\zeta)\right)\int_{\brn}D_{\zeta}\overline{V}^{m,t}(\theta,\zeta,z)(dm'(z)-dm(z)) d\zeta d\theta,
\end{aligned}
\end{equation*}
and then, from the definition of the function $\widetilde{\Psi}^{m,t,s}(x)(\tau,\xi)(z)$ in Equation \eqref{eq:6-32'}, we know that 
\begin{align}\label{add-1}
    \widetilde{\Psi}_{s}(x)(\tau,\xi)=\int_{\brn}\widetilde{\Psi}^{m,t,s}(x)(\tau,\xi)(z)(dm'(z)-dm(z)).
\end{align}
From approaches and the method of majorants used in Section~\ref{sec:main}, we have the boundedness or growth estimates of the functions $\Psi_{s}^{\epsilon,1}$ and $\Psi_{s}^{\epsilon}$. Then, from \eqref{eq:6-38'} and \eqref{add-1} and Lebesgue's dominated convergence theorem, we deduce that
\begin{align*}
    \lim_{\epsilon\to 0}\de^{\epsilon}(s,x)=\ & \int_{\brn} \lim_{\epsilon\to 0} \Psi_{s}^{\epsilon,1}(x)(t,z)\left(dm'(z)-dm(z)\right)+\int_{\brn} \lim_{\epsilon\to 0}\Psi_{s}^{\epsilon}(x)(t,\xi) dm(\xi)\\
    =\ & \int_{\brn} \Psi^{m,t,s}(x)(t,z)(dm'(z)-dm(z))+ \int_{\brn} \widetilde{\Psi}_{s}(x)(t,\xi) dm(\xi)\\
    =\ & \int_{\brn}\left(\Psi^{m,t,s}(x)(t,z)+\int_{\brn}\widetilde{\Psi}^{m,t,s}(x)(t,\xi)(z)dm(\xi)\right)(dm'(z)-dm(z)),
\end{align*}
from which we obtain \eqref{eq:6-33}. 
\end{proof}

Similar to the proof of Proposition~\ref{prop:6-3}, we can also obtain the representation for the term $\dfrac{\partial}{\partial m}\dfrac{d F_{T}}{d\nu}\left(m^{m,t}_T\right)(x)(z)$ as
\begin{align}
    \dfrac{\partial}{\partial m}\dfrac{dF_T}{d\nu}\left(m^{m,t}_T\right)(x)(z)=\ & \Psi^{m,t}_T(x)(t,z)+\int_{\brn}\widetilde{\Psi}^{m,t}_T(x)(t,\xi)(z)dm(\xi), \label{eq:6-33'}
\end{align}
where 
\begin{align}
    \Psi^{m,t}_T(x)(\tau,z):=\ & \int_{\brn}G^{m,t}(\tau,z;T,\zeta)\dfrac{d^{2}F_T}{d\nu^{2}}\left(m^{m,t}_T\right)(x,\zeta)d\zeta, \quad t\le\tau\le T,\quad x,z\in\brn,\label{eq:6-32''} \\
    \widetilde{\Psi}^{m,t}_T(x)(\tau,\xi)(z):=\ & \int_{\tau}^{T}\int_{\brn}G^{m,t}(\tau,\xi;\theta,\zeta)D_{\zeta}\Psi^{m,t}_T(x)(\theta,\zeta) \notag \\
    &\quad\qquad \cdot D_{p}^{2}H\left(\theta,\zeta,DV^{m,t}(\theta,\zeta)\right)D_{\zeta} \overline{V}^{m,t}(\theta,\zeta,z)d\zeta d\theta. \label{eq:6-32'''}
\end{align}
Similar as \eqref{PDE:Psi}, the function $\Psi^{m,t}_T(x)(\tau,z)$ defined in \eqref{eq:6-32''} satisfies the following equation: 
\begin{equation}\label{PDE:Psi_T}
\left\{
\begin{aligned}
    &-\dfrac{\partial \Psi^{m,t}_T}{\partial\tau}(x)(\tau,z) +\mathcal{A}_{z} \Psi^{m,t}_T(x)(\tau,z) =D_{z}\Psi^{m,t}_T(x)(\tau,z) \cdot D_p H\left(\tau,z,DV^{m,t}(\tau,z)\right),\quad \tau\in[t,T),\\
    &\Psi^{m,t}_T(x)(s,z)=\dfrac{d^{2}F_T}{d\nu^{2}}\left(m^{m,t}_T\right)(x,z),\quad z\in\brn.
\end{aligned}
\right.
\end{equation}
We can now obtain the equation of $\overline{V}_{mt}(s,x)(z)$, given by \eqref{eq:6-51}, in terms of $\dfrac{\partial}{\partial m}\dfrac{dF}{d\nu}\left(m^{m,t}_s\right)(x)(z)$ (see \eqref{eq:6-33}) and $\dfrac{\partial}{\partial m}\dfrac{dF_{T}}{d\nu}\left(m^{m,t}_T\right)(x)(z)$ (see \eqref{eq:6-33'}), with which we can give the well-posedness result for $\overline{V}^{m,t}$.

\begin{theorem}\label{thm:D_zV}
    Under Assumptions~\ref{eq:5-8}, \ref{assumption:h}, \ref{assumption:F}, \ref{assumption:H} and \ref{assumption:regularity}-(i), Equation \eqref{eq:6-51} is well-posed and it has a solution $\overline{V}^{m,t}$ in $L^{2}\left(t,T;H_{\pi_{\gamma}}^{1}(\brn)\right)$ when the interval $T-t$ is small; indeed, this solution is the linear functional derivative of $V^{m,t}$, the solution of the HJB equation \eqref{eq:5-150}. The same equation \eqref{eq:6-51} has a global-in-time solution under Assumption~\ref{assumption:convex}.
\end{theorem}

Theorem~\ref{thm:D_zV} gives the well-posedness of linear functional derivative $\dfrac{dV^{m,t}}{d\nu}(s,x)(z)$ of $V^{m,t}(s,x)$, which is also the linear functional derivative $\frac{dU}{d\nu}(s,x,m)(z)$ of the function $U$ defined in \eqref{def:U}. Therefore, this result guarantees the first condition in Condition~\ref{Condition:U}.

\begin{proof}
The proof of the well-posedness for Equation \eqref{eq:6-51} follows by a similar approach developed in Section~\ref{sec:main} leading to the well-posedness of the HJB equation \eqref{eq:5-150}. Indeed, a solution $\overline{V}^{m,t}$ can be considered as a fixed point of a map from $L^{2}\left(t,T;H_{\pi_{\gamma}}^{1}(\brn)\right)$ to $L^{2}\left(t,T;H_{\pi_{\gamma}}^{1}(\brn)\right)$: for any function $\overline{V}^{input}\in L^{2}\left(t,T;H_{\pi_{\gamma}}^{1}(\brn)\right)$, we define $\overline{V}^{output}(s,x)$ by the linear equation 
\begin{equation*}
    \left\{
    \begin{aligned}
        &-\dfrac{\partial \overline{V}^{output}}{\partial s}(s,x,z)+\mathcal{A}_{x}\overline{V}^{output}(s,x,z)=D_{p}H\left(s,x,DV^{output}(s,x)\right)\cdot D_{x}\overline{V}^{output}(s,x,z)\\
        &\qquad\qquad\qquad\qquad\qquad\qquad\qquad\qquad\qquad +\overline{F}(s,x,z),\quad s\in[t,T),\\
        &\overline{V}^{output}(T,x,z)=\overline{F}_T(s,x,z),\quad (x,z)\in\brn\times\brn;
    \end{aligned}
    \right.
\end{equation*}
where the functions $\overline{F}(s,x,z)$ and $\overline{F}_T(s,x,z)$ are defined by \eqref{eq:6-33} and \eqref{eq:6-33'} respectively with $\overline{V}^{m,t}$ in the respective expressions of \eqref{eq:6-32'} and \eqref{eq:6-32'''} being replaced by $\overline{V}^{input}$. If $\overline{V}^{input}=\overline{V}^{output}$, then it gives a solution $\overline{V}^{m,t}$ of \eqref{eq:6-51}. The proof is based on \textit{a priori} estimates corresponding to those for Equation \eqref{eq:6-51}, exactly the same as what we have done in Subsections~\ref{subsec:priori:V}-\ref{subsec:V_x_i}, we shall not repeat and we simply omit them. Then, by following a similar approach as the proof of Proposition~\ref{prop:6-2} and using the uniqueness result of HJB equation \eqref{eq:5-150} (see Subsubsection~\ref{subsubsec:global:|V|}), we see that the unique solution $\overline{V}^{m,t}$ is the linear functional derivative of $V^{m,t}$. 
\end{proof}

Furthermore, we can also move on to study the equation for $D_z\overline{V}^{m,t}(s,x,z)$ as follows:
\begin{equation}\label{equation:D_zV}
    \left\{
    \begin{aligned}
        &-\dfrac{\partial \left(D_z\overline{V}^{m,t}\right)}{\partial s}(s,x,z)+\mathcal{A}_{x}\left(D_z\overline{V}^{m,t}\right)(s,x,z)\\
        &\qquad =D_{p}H\left(s,x,DV^{m,t}(s,x)\right)\cdot D_{x}\left(D_z\overline{V}^{m,t}\right)(s,x,z)+D_z\dfrac{\partial}{\partial m}\dfrac{dF}{d\nu}\left(m^{m,t}_s\right)(x)(z),\quad s\in[t,T),\\
        &D_z\overline{V}^{m,t}(T,x,z)=D_z\dfrac{\partial}{\partial m}\dfrac{dF_{T}}{d\nu}\left(m^{m,t}_T\right)(x)(z),\quad (x,z)\in\brn\times\brn,
    \end{aligned}
    \right.
\end{equation}
where the functions $D_z\dfrac{\partial}{\partial m}\dfrac{dF}{d\nu}\left(m^{m,t}_s\right)(x)(z)$ and $D_z\dfrac{\partial}{\partial m}\dfrac{dF_{T}}{d\nu}\left(m^{m,t}_T\right)(x)(z),\  (x,z)\in\brn\times\brn$, have representations 
%via \eqref{eq:6-33} and \eqref{eq:6-33'} 
as follows: 
\begin{align*}
    &D_z\dfrac{\partial}{\partial m}\dfrac{dF}{d\nu}\left(m^{m,t}_s\right)(x)(z)\\
    =\ & D_z\Psi^{m,t,s}(x)(t,z)+\int_{\brn}\int_{t}^{s}\int_{\brn}G^{m,t}(t,\xi;\tau,\zeta) D_{\zeta}\Psi^{m,t,s}(x)(\tau,\zeta)\\
    &\qquad\qquad\qquad\qquad \cdot D_{p}^{2}H\left(\tau,\zeta,DV^{m,t}(\tau,\zeta)\right)D_{\zeta}\left(D_z\overline{V}^{m,t}\right)(\tau,\zeta,z) d\zeta d\tau dm(\xi),\\
    &D_z\dfrac{\partial}{\partial m}\dfrac{dF_T}{d\nu}\left(m^{m,t}_T\right)(x)(z)\\
    =\ & D_z\Psi^{m,t}_T(x)(t,z)+\int_{\brn}\int_{t}^{T}\int_{\brn}G^{m,t}(t,\xi;\tau,\zeta) D_{\zeta}\Psi^{m,t}_T(x)(\tau,\zeta)\\
    &\qquad\qquad\qquad\qquad \cdot D_{p}^{2}H\left(\tau,\zeta,DV^{m,t}(\tau,\zeta)\right)D_{\zeta}\left(D_z\overline{V}^{m,t}\right)(\tau,\zeta,z) d\zeta d\tau dm(\xi).
\end{align*}
The study of the well-posedness of Equation \eqref{equation:D_zV} can give the well-posedness of the derivative $D_z\dfrac{dV^{m,t}}{d\nu}(s,x)(z)$, which is also the derivative $D_z\frac{dU}{d\nu}(s,x,m)(z)$ of the function $U$ defined in \eqref{def:U}. Particularly, the validity of the following Proposition~\ref{thm:D_zV'} for $D_z\overline{V}^{m,t}(s,x,z)$ can guarantee the second condition in Condition~\ref{Condition:U}. 

Finally, the equation for $D_z^2\overline{V}^{m,t}(s,x,z)$ is as follows:
\begin{equation}\label{equation:D_z^2V}
    \left\{
    \begin{aligned}
        &-\dfrac{\partial \left(D_z^2\overline{V}^{m,t}\right)}{\partial s}(s,x,z)+\mathcal{A}_{x}\left(D_z^2\overline{V}^{m,t}\right)(s,x,z)\\
        &\qquad =D_{p}H\left(s,x,DV^{m,t}(s,x)\right)\cdot D_{x}\left(D_z^2\overline{V}^{m,t}\right)(s,x,z) +D_z^2\dfrac{\partial}{\partial m}\dfrac{dF}{d\nu}\left(m^{m,t}_s\right)(x)(z),\quad s\in[t,T),\\
        &D_z^2\overline{V}^{m,t}(T,x,z)=D_z^2 \dfrac{\partial}{\partial m}\dfrac{dF_{T}}{d\nu}\left(m^{m,t}_T\right)(x)(z),\quad (x,z)\in\brn\times\brn,
    \end{aligned}
    \right.
\end{equation}
where the functions $D_z^2\dfrac{\partial}{\partial m}\dfrac{dF}{d\nu}\left(m^{m,t}_s\right)(x)(z)$ and $D_z^2\dfrac{\partial}{\partial m}\dfrac{dF_{T}}{d\nu}\left(m^{m,t}_T\right)(x)(z),\  (x,z)\in\brn\times\brn$, have representations as follows: 
\begin{align*}
    &D_z^2\dfrac{\partial}{\partial m}\dfrac{dF}{d\nu}\left(m^{m,t}_s\right)(x)(z)\\
    =\ & D_z^2\Psi^{m,t,s}(x)(t,z)+\int_{\brn}\int_{t}^{s}\int_{\brn}G^{m,t}(t,\xi;\tau,\zeta) D_{\zeta}\Psi^{m,t,s}(x)(\tau,\zeta)\\
    &\qquad\qquad\qquad\qquad \cdot D_{p}^{2}H\left(\tau,\zeta,DV^{m,t}(\tau,\zeta)\right)D_{\zeta}\left(D_z^2\overline{V}^{m,t}\right)(\tau,\zeta,z) d\zeta d\tau dm(\xi),\\
    &D_z^2\dfrac{\partial}{\partial m}\dfrac{dF_T}{d\nu}\left(m^{m,t}_T\right)(x)(z)\\
    =\ & D_z^2\Psi^{m,t}_T(x)(t,z)+\int_{\brn}\int_{t}^{T}\int_{\brn}G^{m,t}(t,\xi;\tau,\zeta) D_{\zeta}\Psi^{m,t}_T(x)(\tau,\zeta)\\
    &\qquad\qquad\qquad\qquad \cdot D_{p}^{2}H\left(\tau,\zeta,DV^{m,t}(\tau,\zeta)\right)D_{\zeta}\left(D_z^2\overline{V}^{m,t}\right)(\tau,\zeta,z) d\zeta d\tau dm(\xi).
\end{align*}
Again, the study of the well-posedness of Equation \eqref{equation:D_z^2V} can give the well-posedness of this higher-order derivative $D_z^2\dfrac{dV^{m,t}}{d\nu}(s,x)(z)$, which is also the derivative $D_z^2\frac{dU}{d\nu}(s,x,m)(z)$ of the function $U$ defined in \eqref{def:U}. Particularly, the validity of the following Proposition~\ref{thm:D_zV'} for $D_z^2\overline{V}^{m,t}(s,x,z)$ can warrant the third condition in Condition~\ref{Condition:U}. Finally, we give the following result for the well-posedness results for Equations \eqref{equation:D_zV} and  \eqref{equation:D_z^2V}; its proof is similar to that of Theorems~\ref{theo5-2} and \ref{theo5-3}, and is omitted here. 

\begin{proposition}\label{thm:D_zV'}
Suppose that Assumptions~\ref{eq:5-8}, \ref{assumption:h}, \ref{assumption:F}, \ref{assumption:H} and \ref{assumption:regularity}-(i) are valid, then, we have the following results. \\
(a) Equation \eqref{equation:D_zV} is well-posed and it has a solution $D_z\overline{V}^{m,t}$ in $L^{2}\left(t,T;H_{\pi_{\gamma}}^{1}(\brn)\right)$ when the interval $T-t$ is small; indeed, this solution is the derivative of $\overline{V}^{m,t}$ in $z$. The same equation \eqref{equation:D_zV} has a global-in-time solution under Assumption~\ref{assumption:convex}. \\
(b) If moreover Assumption~\ref{assumption:regularity}-(ii) is satisfied, then Equation \eqref{equation:D_z^2V} is well-posed and it has a solution $D_z^2\overline{V}^{m,t}$ in $L^{2}\left(t,T;H_{\pi_{\gamma}}^{1}(\brn)\right)$ when the interval $T-t$ is small; this solution is the twice derivative of $\overline{V}^{m,t}$ in $z$. The same equation \eqref{equation:D_z^2V} has a global-in-time solution under Assumption~\ref{assumption:convex}. 
\end{proposition}

As a direct consequence of Propositions~\ref{prop:6-1} and \ref{prop:6-2}, Theorem~\ref{thm:D_zV} and Proposition~\ref{thm:D_zV'}, we have the following corollary.

\begin{corollary}\label{corollary}
Under Assumptions~\ref{eq:5-8}, \ref{assumption:h}, \ref{assumption:F}, \ref{assumption:H} and \ref{assumption:regularity}, $U(s,x,m):=V^{m,s}(s,x)$ is a solution of the master equation \eqref{eq:6-1}, and the value function $\Phi$ is a solution of the Bellman equation \eqref{eq:6-5}, when the interval $T-t$ is small. Furthermore, the master equation and Bellman equation have their respective global-in-time solutions under Assumption~\ref{assumption:convex}.
\end{corollary} 

For the uniqueness results for Equations \eqref{eq:6-51}, \eqref{equation:D_zV} and \eqref{equation:D_z^2V} and also the master equation \eqref{eq:6-1} and the Bellman equation \eqref{eq:6-5}, we refer to the method in our previous work \cite{AB6,AB10,AB5} for the uniqueness results for equations arsing in the mean field theory; and these results are not limited within the scope of using probabilistic methods. Similar arguement can be written down from the analytical perspective as shown in this whole article, and we shall not repeat further here.

\section*{Acknowledgement}

Alain Bensoussan is supported by the National Science Foundation under grant NSF-DMS-2204795. Ziyu Huang acknowledges the financial supports as a postdoctoral fellow from Department of Statistics and Data Science of The Chinese University of Hong Kong. Phillip Yam acknowledges the financial supports from HKGRF-14301321 with the project title ``General Theory for Infinite Dimensional Stochastic Control: Mean field and Some Classical Problems'', HKGRF-14300123 with the project title ``Well-posedness of Some Poisson-driven Mean Field Learning Models and their Applications'', and HKGRF-14300025 with the project title ``A Generic Theory for Stochastic Control against Fractional Brownian Motions''. The work described in this article was also supported by a grant from the Germany/Hong Kong Joint Research Scheme sponsored by the Research Grants Council of Hong Kong and the German Academic Exchange Service of Germany (Reference No. G-CUHK411/23). He also thanks The University of Texas at Dallas for the kind invitation to be a Visiting Professor in Naveen Jindal School of Management.

%%%%%%%%%%%%%%%%%%%%%%%%%%%%%%%%%%%%%%%%%%%%%%%%%%%%%%%%%%%%%%%%%%%%%%%%%%%%

%%%%%%%%%%%%%%%%%%%%%%%%%%%%%%%%%%%%%%%%%%%%%%%%%%%%%%%%%%%%%%%%%%%%%%%%%%%%

%\newpage
%\normalsize

%\appendix
%\appendixpage

\end{document}